\documentclass[a4paper,11pt]{amsart}
\usepackage{amsmath,amssymb,epsfig,rotating}
\usepackage[all]{xy}
\usepackage{pinlabel}
\usepackage{xcolor}

\evensidemargin \oddsidemargin \marginparwidth 0.5in 
\usepackage[
  pdftitle={Whitney polygons, symbol homology and cobordism maps},
  pdfauthor={Bijan Sahamie},
  citecolor=blue,
  colorlinks,
  linkcolor=blue,
  urlcolor=red,
  ]{hyperref}

\title{Whitney polygons, symbol homology and cobordism maps}

\author{Bijan Sahamie} 
\address{Mathematisches Institut der LMU M\"unchen, 
Theresienstrasse 39, 80333 M\"unchen Germany}
\email{sahamie@math.lmu.de}
\urladdr{http://www.math.lmu.de/~sahamie}
\theoremstyle{plain} 
\newtheorem{theorem}{Theorem}[section]   
\newtheorem{lem}[theorem]{Lemma}         
\newtheorem{prop}[theorem]{Proposition}
\newtheorem{cor}[theorem]{Corollary}

\theoremstyle{definition}
\newtheorem{definition}[theorem]{Definition}   
\newtheorem{rem}{Remark}

\newtheorem{example}{Example}[section]            
\numberwithin{equation}{section}

\DeclareMathSizes{11}{11}{8}{7}

\begin{document}
%
%
\newcommand{\bbh}{\mathfrak{sh}}
\newcommand{\bh}{\bbh_*}
\newcommand{\bhu}{\bbh_{U,*}}
\newcommand{\mSu}{\mathcal{S}_U}
\newcommand{\bbl}{\mathbb{L}}
\newcommand{\Z}{\mathbb{Z}}
\newcommand{\R}{\mathbb{R}}
\newcommand{\C}{\mathbb{C}}
\newcommand{\Q}{\mathbb{Q}}
\newcommand{\phinat}{\phi_{nat}}
\newcommand{\hfkhatabs}{\underline{\overline{\rm HFK}}}
\newcommand{\MCG}{\mbox{\rm MCG}}
\newcommand{\ehat}{\widehat{\rm E}}
\newcommand{\phinathat}{\widehat{\phinat}}
\newcommand{\E}{\Delta}
\newcommand{\gltz}{\mbox{\rm GL}_2(\Z)}
\newcommand{\delhatw}{\widehat{\partial}^w}
\newcommand{\afat}{{\boldsymbol\alpha}}
\newcommand{\bfat}{{\boldsymbol\beta}}
\newcommand{\gfat}{{\boldsymbol\gamma}}
\newcommand{\pfat}{{\boldsymbol p}}
\newcommand{\pfu}{{\boldsymbol {\underline p}}}
\newcommand{\pfuprime}{{\boldsymbol {\underline p}'}}
\newcommand{\qfat}{{\boldsymbol q}}
\newcommand{\qfu}{{\boldsymbol {\underline q}}}
\newcommand{\qfuprime}{{\boldsymbol {\underline q}'}}
\newcommand{\rfat}{{\boldsymbol r}}
\newcommand{\tfat}{{\boldsymbol t}}
\newcommand{\rfu}{{\boldsymbol {\underline r}}}
\newcommand{\xfat}{{\boldsymbol x}}
\newcommand{\xfu}{{\boldsymbol {\underline x}}}
\newcommand{\yfat}{{\boldsymbol y}}
\newcommand{\yfu}{{\boldsymbol {\underline y}}}
\newcommand{\zfat}{{\boldsymbol z}}
\newcommand{\sspinc}{\mbox{\rm {\tiny Spin}}^c}
\newcommand{\parsh}{\partial_{\bbh}}
\newcommand{\pshdot}{\parsh^{\bullet}}
\newcommand{\pshbox}{\parsh^{\boxtimes}}
\newcommand{\pco}{\partial_{\mbf}}
\newcommand{\parcoone}{\partial^{\mbox{\rm\tiny codim } 1}}
\newcommand{\mbp}{\mathbb{P}}
\newcommand{\mbf}{\mathbb{F}}
\newcommand{\csum}{+}
\newcommand{\cp}{\bullet}
\newcommand{\cohom}{\mathfrak{f}_*}
\newcommand{\mpq}{\mathcal{X}}
\newcommand{\hpq}{\mathcal{X}_*}
\newcommand{\os}{\mathbb{O}}
\newcommand{\oab}{\mathbb{O}_{(\afat,\bfat)}}
\newcommand{\oaibi}{\mathbb{O}_{(\afat_i,\bfat_i)}}
\newcommand{\oapbp}{\mathbb{O}_{(\afat',\bfat')}}
\newcommand{\coll}{\mathfrak{coll}}
\def\overdot#1{\overset{\mbox{\bf .}}{#1}}
\def\overup#1{\overset{\uparrow}{#1}}
\def\overdown#1{\overset{\downarrow}{#1}}
\newcommand{\fin}{F^{\downarrow}}
\newcommand{\fout}{F^{\uparrow}}
\newcommand{\marked}{M^{\mbox{\rm\tiny\textbullet}}}
\newcommand{\pointed}{P^{\mbox{\rm\tiny\textbullet}}}
\newcommand{\bolddot}{\mbox{\rm\textbullet}}
\newcommand{\finprime}{F^{\downarrow}\,\!'}
\newcommand{\foutprime}{F^{\uparrow}\,\!'}
\newcommand{\markedprime}
{M^{\mbox{\rm\tiny\textbullet}}\,\!'}
\newcommand{\pointedpp}
{P^{\mbox{\rm\tiny\textbullet}}\,\!''}
\newcommand{\finpp}{F^{\downarrow}\,\!''}
\newcommand{\foutpp}{F^{\uparrow}\,\!''}
\newcommand{\markedpp}
{M^{\mbox{\rm\tiny\textbullet}}\,\!''}
\newcommand{\pointedprime}
{P^{\mbox{\rm\tiny\textbullet}}\,\!'}
\newcommand{\uniq}{\mathfrak{U}^\downarrow}
\newcommand{\mD}{\mathfrak{D}}
\newcommand{\mff}{\mathfrak{F}}
\newcommand{\pu}{\,\widehat{+}\,}
\newcommand{\cu}{\,\widehat{\circ}\,}
\newcommand{\gamfull}{\mathcal{M}^{\mu}_{B}}
\newcommand{\gafull}{\mathcal{M}^{\mu}}
\newcommand{\gamfullw}{\Gamma^{\mu;w}_{B}}
\newcommand{\ev}{\mathfrak{ev}}
\newcommand{\evf}{\mathfrak{ct}_{\mbf}}
\newcommand{\mbT}{\mathbb{T}}
\newcommand{\pol}{\mathbb{P}}
\newcommand{\inmultc}{\mathfrak{m}^{\mC}}
\newcommand{\outmultc}{\mathfrak{m}_{\mC}}
\newcommand{\inmultr}{\mathfrak{m}^{\mR}}
\newcommand{\outmultr}{\mathfrak{m}_{\mR}}
\newcommand{\mm}{\mathfrak{m}}
%
%

\newcommand{\sone}{\mathbb{S}^1}
\newcommand{\lra}{\longrightarrow}
\newcommand{\lmt}{\longmapsto}
%

%
%
\newcommand{\ztwo}{\mathbb{Z}_2}
\newcommand{\RP}{\mbox{\rm RP}}
\newcommand{\spinc}{\mbox{\rm Spin}^c}
\def\co{\colon\thinspace}
\newcommand{\stwo}{\mathbb{S}^2}
\newcommand{\pd}{\text{PD}}
\newcommand{\sprung}{\\[0.3cm]}
\newcommand{\fund}{\pi_1}
\newcommand{\kerg}{\mbox{\rm Ker}_G\,}
\def\Ker#1{\mbox{\rm Ker}_{#1}\,}
\newcommand{\im}{\mbox{\rm Im}\,}
\newcommand{\id}{\mbox{\rm id}}
\newcommand{\sothree}{\mathbb{SO}_3}
\newcommand{\sthree}{\mathbb{S}^{3}}
\newcommand{\disc}{\mbox{\rm D}}
\newcommand{\systwo}{C^{\infty}(Y,\stwo)}
\newcommand{\inner}{\text{int}}
\newcommand{\bk}{\backslash}

%
%
\newcommand{\contstand}{(\mathbb{R}^3,\xi_0)}
\newcommand{\cont}{(M,\xi)}
\newcommand{\sym}{\xi_{sym}}
\newcommand{\xistd}{\xi_{std}}
\newcommand{\cyl}{\mbox{\rm Cyl}_{r_0}^\mu}
\newcommand{\stap}{\mbox{\rm S}_+}
\newcommand{\stam}{\mbox{\rm S}_-}
\newcommand{\stapm}{\mbox{\rm S}_\pm}
\newcommand{\modulo}{\;\;\mbox{\rm mod}\,}

%
%
\newcommand{\crit}{\mbox{\rm Crit}}
\newcommand{\MC}{\mbox{\rm MC}}
\newcommand{\bmorse}{\partial^{\mbox{\rm \begin{tiny}M\!C\end{tiny}}}}
\newcommand{\M}{\mathcal{M}}
\newcommand{\stable}{\mbox{\rm W}^s}
\newcommand{\unstable}{\mbox{\rm W}^u}
\newcommand{\ind}{\mbox{\rm ind}}
\newcommand{\Mhat}{\widehat{\M}}
\newcommand{\What}{\widehat{W}}
\newcommand{\modphi}{\mathcal{M}_\phi}
%
%
%

%
%
\newcommand{\moduli}{\mathcal{M}_{J_s}(x,y)}
\newcommand{\modulit}{\mathcal{M}_{J_{s,t}}}
\newcommand{\modulittau}{\mathcal{M}_{J_{s,t(\tau)}}}
\newcommand{\modulittaubig}{\mathcal{M}_{J_{s,t,\tau}}}
\newcommand{\modhat}{\widehat{\mathcal{M}}_{J_s}(x,y)}
\newcommand{\modulihat}{\widehat{\mathcal{M}}}
\newcommand{\modhatone}{\widehat{\mathcal{M}}_{J_{s,1}}}
\newcommand{\modhatzero}{\widehat{\mathcal{M}}_{J_{s,0}}}
\newcommand{\modhatphi}{\widehat{\mathcal{M}}_{\phi}}
\newcommand{\moduliiso}{\mathcal{M}^{\Psi_t}}
\newcommand{\modspace}{\mathcal{M}}
\newcommand{\phihat}{\widehat{\phi}}
\newcommand{\Phiinfty}{\Phi^\infty}
\newcommand{\Dhat}{\widehat{\D}}
\newcommand{\cops}{\partial_{J_s}}
\newcommand{\talpha}{\mathbb{T}_{\boldsymbol{\alpha}}}
\newcommand{\tbeta}{\mathbb{T}_{\boldsymbol{\beta}}}
\newcommand{\tgamma}{\mathbb{T}_{\boldsymbol{\gamma}}}
\def\marge#1{\marginpar{\scriptsize{#1}}}
\def\br#1{\begin{rotate}{90}#1\end{rotate}}
\newcommand{\hfhat}{\widehat{\mbox{\rm HF}}}
\newcommand{\sfh}{\mbox{\rm SFH}}
\newcommand{\sbottom}{\underline{s}}
\newcommand{\tbottom}{\underline{t}}
\newcommand{\cfhat}{\widehat{\mbox{\rm CF}}}
\newcommand{\cfinfty}{\mbox{\rm CF}^\infty}
\def\cfbb#1{\mbox{\rm CF}^+_{\leq #1}}
\def\cfbo#1{\mbox{\rm CF}^-_{\geq -#1}}
\def\bigboxplus#1{\underset{#1}{\mbox{\Huge\raisebox{-0.5ex}{$\boxplus$}}}}
\def\bbigboxplus#1#2{
\overset{#2}{\underset{#1}{\mbox{\Huge\raisebox{-0.5ex}{$\boxplus$}}}}
}
\def\bigboxtimes#1{\underset{#1}{\mbox{\Huge\raisebox{-0.5ex}{$\boxtimes$}}}}
\def\bbigboxtimes#1#2{
\overset{#2}{\underset{#1}{\mbox{\Huge\raisebox{-0.5ex}{$\boxtimes$}}}}
}
\def\bigcp#1#2{
\overset{#2}{\underset{#1}{\mbox{\Huge\raisebox{-0.32ex}{$\cp$}}}}
}
\newcommand{\cfleq}{\mbox{\rm CF}^{\leq 0}}
\newcommand{\cfcirc}{\mbox{\rm CF}^\circ}
\newcommand{\cfkcirc}{\mbox{\rm CFK}^\circ}
\newcommand{\cfkmcirc}{\mbox{\rm CFK}^{\circ,0}}
\newcommand{\cfkleq}{\mbox{\rm CFK}^{\leq 0}}
\newcommand{\cfkinfty}{\mbox{\rm CFK}^{0,\infty}}
\newcommand{\cfkminfty}{\mbox{\rm CFK}^{\infty,0}}

\newcommand{\cfkhat}{\widehat{\mbox{\rm CFK}}}
\newcommand{\cfkminus}{\mbox{\rm CFK}^{\bullet,-}}
\newcommand{\cfkmminus}{\mbox{\rm CFK}^{-,0}}

\newcommand{\cfkpstar}{\mbox{\rm CFK}^{+,*}}
\newcommand{\cfkostar}{\mbox{\rm CFK}^{\bullet,*}}
\newcommand{\hfkcirc}{\mbox{\rm HFK}^\circ}

\newcommand{\cfkoc}{\mbox{\rm CFK}^{\bullet,\circ}}
\newcommand{\cfkoo}{\mbox{\rm CFK}^{\bullet,\bullet}}
\newcommand{\hfkoc}{\mbox{\rm HFK}^{\bullet,\circ}}

\newcommand{\cfkco}{\mbox{\rm CFK}^{\circ,\bullet}}
\newcommand{\hfkco}{\mbox{\rm HFK}^{\circ,\bullet}}

\newcommand{\hfkoo}{\mbox{\rm HFK}^{\bullet,\bullet}}

\newcommand{\hfkmcirc}{\mbox{\rm HFK}^{\circ,0}}

\newcommand{\hfkhat}{\widehat{\mbox{\rm HFK}}}
\newcommand{\hfkplus}{\mbox{\rm HFK}^{\bullet,+}}
\newcommand{\hfkmplus}{\mbox{\rm HFK}^{+,\bullet}}
\newcommand{\hfkminus}{\mbox{\rm HFK}^{\bullet,-}}
\newcommand{\hfkmminus}{\mbox{\rm HFK}^{-,\bullet}}

\newcommand{\hfkinfty}{\mbox{\rm HFK}^{\bullet,\infty}}
\newcommand{\hfkminfty}{\mbox{\rm HFK}^{\infty,\bullet}}

\newcommand{\fco}{f^{\circ,\bullet}}
\newcommand{\Fco}{F^{\circ,\bullet}}

\newcommand{\Fcirc}{F^\circ}
\newcommand{\foc}{f^{\bullet,\circ}}
\newcommand{\Foc}{F^{\bullet,\circ}}
\newcommand{\Foo}{F^{\bullet,\bullet}}
\newcommand{\oc}{{\bullet,\circ}}

\newcommand{\gund}{g^{\bullet,\circ}}
\newcommand{\Ghatund}{G^{\bullet,\circ}}
\newcommand{\eund}{e^{\bullet,\circ}}
\newcommand{\Ehatund}{E^{\bullet,\circ}}

\def\cfinftyfilt#1#2{\mbox{\rm CFK}^{#1,#2}}
\newcommand{\hfinfty}{\mbox{\rm HF}^\infty}
\newcommand{\hfinftwist}{\underline{\mbox{\rm {HF}}^\infty}}
\newcommand{\fhat}{\widehat{f}}
\newcommand{\fcirc}{f^\circ}
\newcommand{\Fhat}{\widehat{F}}
\newcommand{\Jhat}{\widehat{J}}
\newcommand{\That}{\widehat{T}}
\newcommand{\Hhat}{\widehat{H}}

\newcommand{\hattheta}{\widehat{\Theta}}
\newcommand{\shattheta}{\widehat{\theta}}
\newcommand{\cfminus}{\mbox{\rm CF}^-}
\newcommand{\hfminus}{\mbox{\rm HF}^-}
\newcommand{\cfplus}{\mbox{\rm CF}^+}
\newcommand{\hfplus}{\mbox{\rm HF}^+}
\newcommand{\hfcirc}{\mbox{\rm HF}^\circ}
\def\hfbb#1{\hfplus_{\leq #1}}
\def\hfbo#1{\hfminus_{\geq -#1}}
\newcommand{\Hs}{\mathcal{H}_s}
\newcommand{\mh}{\mathcal{H}}
\newcommand{\gr}{\mbox{gr}}
\newcommand{\parinfty}{\partial^\infty}
\newcommand{\parhat}{\widehat{\partial}}
\newcommand{\parplus}{\partial^+}
\newcommand{\parminus}{\partial^-}
\newcommand{\symg}{\mbox{\rm Sym}^g(\Sigma)}
\newcommand{\symgg}{\mbox{\rm Sym}^{2g}(\Sigma)}
\newcommand{\symc}{\mbox{\rm Sym}^g(\mathbb{C})}
\newcommand{\pitwo}{\pi_2}
\newcommand{\pitwoham}{\pi_2^{\Psi_t}}
\newcommand{\symcon}{\mbox{\rm Sym}^g(\Sigma_1\#\Sigma_2)}
\newcommand{\symgone}{\mbox{\rm Sym}^{g_1}(\Sigma_1)}
\newcommand{\symgtwo}{\mbox{\rm Sym}^{g_2}(\Sigma_2)}
\newcommand{\symgmo}{\mbox{\rm Sym}^{g-1}(\Sigma)}
\newcommand{\symggmo}{\mbox{\rm Sym}^{2g-1}(\Sigma)}
\newcommand{\dom}{\mathcal{D}}
\newcommand{\bigtrans}{\left.\bigcap\hspace{-0.27cm}\right|\hspace{0.1cm}}
\newcommand{\tlt}{\times\ldots\times}
\newcommand{\Isotopy}{\Gamma^\infty_{\Psi_t}}
\newcommand{\orient}{\mathnormal{o}}
\newcommand{\ob}{\mathnormal{ob}}
\newcommand{\SL}{\mbox{\rm SL}}
\newcommand{\rhotilde}{\widetilde{\rho}}
\newcommand{\domstar}{\dom_*}
\newcommand{\domststar}{\dom_{**}}
\newcommand{\betaprime}{\beta'}
\newcommand{\betapp}{\beta''}
\newcommand{\betatilde}{\widetilde{\beta}}
\newcommand{\deltaprime}{\delta'}
\newcommand{\tbetaprime}{\mathbb{T}_{\bfat'}}
\newcommand{\talphaprime}{\mathbb{T}_{\afat'}}
\newcommand{\tdelta}{\mathbb{T}_{\delta}}
\newcommand{\phidelta}{\phi^{\Delta}}
\newcommand{\domtilde}{\widetilde{\dom}}
\newcommand{\loss}{\widehat{\mathcal{L}}}
\newcommand{\bargamma}{\overline{\Gamma}}
\newcommand{\alphaprime}{\alpha'}
\newcommand{\ga}{\Gamma_{\alpha;\beta',\beta''}}
\newcommand{\gbone}{\Gamma_{\alpha;\beta,\widetilde{\beta}}^{w,1}}
\newcommand{\gbtwo}{\Gamma_{\alpha;\beta,\widetilde{\beta}}^{w,2}}
\newcommand{\gbthree}{\Gamma_{\alpha;\beta,\widetilde{\beta}}^{w,3}}
\newcommand{\gbfour}{\Gamma_{\alpha;\beta,\widetilde{\beta}}^{w,4}}
\newcommand{\gcone}{\Gamma_{\alpha;\delta,\delta'}^{w,1}}
\newcommand{\gctwo}{\Gamma_{\alpha;\delta,\delta'}^{w,2}}
\newcommand{\gcthree}{\Gamma_{\alpha;\delta,\delta'}^{w,3}}
\newcommand{\gcfour}{\Gamma_{\alpha;\delta,\delta'}^{w,4}}

\newcommand{\eab}{\epsilon_{\alpha\beta}}
\newcommand{\ead}{\epsilon_{\alpha\delta}}
\newcommand{\hqhat}{\widehat{\mbox{\rm HQ}}}
\newcommand{\cupb}{\cup_\partial}
\newcommand{\oa}{\overline{a}}
\newcommand{\ab}{\alpha\beta}
\newcommand{\ad}{\alpha\delta}
\newcommand{\adb}{\alpha\delta\beta}
\newcommand{\tila}{\widetilde{a}}
\newcommand{\tilb}{\widetilde{b}}
\def\pdehn#1#2{D_{#1}^{+,#2}} 
\def\ndehn#1#2{D_{#1}^{-,#2}} 
\newcommand{\fraks}{\mathfrak{s}}
\newcommand{\frakso}{\mathfrak{s}_0}
\newcommand{\frakt}{\mathfrak{t}}
\newcommand{\Ghat}{\widehat{G}}
\newcommand{\pointswap}{\phi^{\rm PS}}
\newcommand{\genA}{A_{(\pointed,\fin,\fout)}}
\newcommand{\genB}{B_{(\pointedprime,\finprime,\foutprime)}}
\newcommand{\mA}{\mathcal{A}}
\newcommand{\mB}{\mathcal{B}}
\newcommand{\mC}{\mathcal{C}}
\newcommand{\mE}{\mathcal{E}}
\newcommand{\mF}{\mathcal{F}}
\newcommand{\mI}{\mathcal{I}}
\newcommand{\mJ}{\mathcal{J}}
\newcommand{\mM}{\mathcal{M}}
\newcommand{\mMhat}{\widehat{\mathcal{M}}}
\newcommand{\mN}{\mathcal{N}}
\newcommand{\mH}{\mathcal{H}}
\newcommand{\mK}{\mathcal{K}}
\newcommand{\mG}{\mathcal{G}}
\newcommand{\mP}{\mathcal{P}}
\newcommand{\mQ}{\mathcal{Q}}
\newcommand{\mR}{\mathcal{R}}
\newcommand{\mS}{\mathcal{S}}
\newcommand{\mShat}{\widehat{\mathcal{S}}}
\newcommand{\mT}{\mathcal{T}}
\newcommand{\mThat}{\widehat{\mathcal{T}}}
\newcommand{\sh}{\mbox{\rm h}}
\newcommand{\shb}{\mbox{\rm \underline{h}}}
\newcommand{\link}{\mathbb{L}}
\newcommand{\mKinq}{\mK^{\downarrow,\fout}}
\newcommand{\mKind}{\mK^{\downarrow,Q}}
\def\mKin#1{\mK^{\downarrow,#1}}
\newcommand{\mKoutq}{\mK^{\uparrow}}
\newcommand{\mKinr}{\mK^{\downarrow,\foutprime}}
\newcommand{\map}{\mathbb{MOR}}
\newcommand{\maphat}{\widehat{\mathbb{MOR}}}
\newcommand{\ohat}{\widehat{\Omega}}
\newcommand{\pto}{\widehat{\Psi}_{to}}
\newcommand{\pot}{\widehat{\Psi}_{ot}}
%
%
\newcommand{\bund}{\mathcal{P}}
\newcommand{\diag}{\Delta^{\!\!E}}
\newcommand{\inter}{m_{\diag}}
\newcommand{\ozs}{Ozsv\'{a}th}
\newcommand{\sza}{Szab\'{o}}
%

%
%
\fontsize{11}{14}\selectfont

\begin{abstract} 
We define a new homology theory we call symbol homology by
using decorated moduli spaces of Whitney polygons. By 
decorating different types of moduli spaces we obtain 
different flavors of this homology theory together with 
morphisms between them. Each of these flavors encodes the 
properties of a different type of Heegaard Floer homology. 
The morphisms between the symbol homologies enable 
us to push properties from one Floer theory to a 
different one. Furthermore, we obtain a new presentation 
of Heegaard Floer theory in which maps correspond 
to multiplication from the right with suitable elements of
our symbol homology. Finally, we present the construction
of cobordism maps in knot Floer theories and apply the
tools from symbol homology to give an invariance
proof.
\end{abstract}
\maketitle

\section{Introduction}\label{parone}
In \cite{OsZa01}, Ozsv\'{a}th and Szab\'{o} assign to a pointed
Heegaard diagram $(\Sigma,\afat,\bfat)$ the homology theories $\hfhat$,
$\hfminus$ and $\hfinfty$ by studying holomorphic disks in the symmetric product
$\symg$. These homologies turn out to be invariants of the $3$-manifold
associated to the Heegaard diagram. By altering the construction process,
different flavors of this theory were introduced like for instance knot
Floer homology or sutured Floer homology. Each of these theories are based
on studying moduli spaces of Whitney polygons. The $0$-dimensional 
moduli spaces are used to define maps, and the boundaries of 
the $1$-dimensional moduli spaces provide relations/properties of these maps 
and of the theory. We call this the {\it Floer theoretic scheme}, which is
common among all existing Floer theories. 
More precisely, almost all properties of maps in Floer theory are proved by 
applying the moduli space machinery, i.e.~deriving a statement on the 
moduli space level, and then interpreting this statement on the Floer chain 
level by counting components of moduli spaces.
This 
observation suggests that Floer theory takes place in 
an algebraic object formed by moduli spaces themselves, but which 
is hidden by the procedure of counting elements. The main goal of this 
paper is to construct such an algebraic object.\vspace{0.3cm}\\
We will decorate moduli spaces of Whitney polygons 
with data. More precisely, we classify three types of decorations and 
attach these decorations to the vertices of the polygons 
(see~Definition~\ref{def:generators} or cf.~\S\ref{sec:eai}). These decorated 
spaces will be used to generate an algebraic object $(\mThat,\boxplus,\boxtimes)$ 
in which we identify a substructure $(\mS,\boxplus,\boxtimes)$ we call 
{\bf symbol algebra} (see~Definition~\ref{def:symalg}). 
The connection between the symbol algebra and the Floer theoretic 
level is a {\it morphism}
$
 \ev
 \co
 (\mS,\boxplus,\boxtimes)\lra(\map,+,\circ)
$
where $\map$ should be thought of as the set of maps in 
Floer homology (see~Proposition~\ref{intro:ev} or cf.~\S\ref{sec:eai}). In 
this way, every map in Floer homology which is defined by counting 
$0$-dimensional
components of moduli spaces of Whitney polygons can be described
as an element in the symbol algebra. Moreover, in this algebraic
setting the codimension-$1$ boundaries of moduli spaces of Whitney
polygons can be described as elements in $\mS$: so, we obtain a 
differential $\parsh\co\mS\lra\mS$ on the symbol
algebra. We denote by $\bh$ the homology theory associated to
$(\mS,\parsh)$ and call it {\bf symbol homology}.
The morphism $\ev$ vanishes
on boundaries and, hence, a given map $f$ on the Floer theoretic 
level can be represented by an element $\fraks$ in the symbol homology 
in the sense that $\ev(\fraks)=f$.
By the Floer theoretic scheme mentioned above, a property of a map $f$
--~which is proved by applying the moduli space machinery~-- can be encoded
into a polynomial expression $P$ with coefficients in $\bh$ in such
a way that $f$ fulfills the property 
if $P(\fraks)=0$ (cf.~Example~\ref{example01}).
Hence, the symbol homology captures the information provided by the 
Floer theory and, therefore, seems to be a reasonable candidate for 
the algebraic object mentioned in the first paragraph.

The first construction of a symbol 
homology is given in \S\ref{sec:tshp}
and done using moduli spaces which are relevant for 
the $\hfhat$-theory.\\
A particular appealing feature of this theory is that we obtain 
a nice interpretation of Heegaard Floer theory in terms
of our algebraic setting. In fact, given a Heegaard diagram 
$\mH=(\Sigma,\afat,\bfat)$ we find a module $\mpq^0_{(\afat,\bfat)}$ in 
$\mThat$ which is naturally equipped with a differential 
$\partial_\mpq$, such that
\begin{equation}
 (\mpq^0_{(\afat,\bfat)}
 \otimes_{\cohom}\ztwo,\partial_\mpq\otimes\id)
 \cong
 (\cfhat(\mH),\parhat_{\mH})
 \label{eq:equiv}
\end{equation}
as chain complexes (see Theorem~\ref{thm:alterhf}). Furthermore, 
suppose we are given another
Heegaard diagram $\mH'$ and a map $f$ between 
the Heegaard Floer chain modules $\cfhat(\mH)$ and 
$\cfhat(\mH')$, where we denote by $\fraks_f$ the symbol which represents
$f$, then multiplication from the right with this symbol defines a map 
$
 \cdot\;\boxtimes\fraks_f\boxtimes\os
\co\mpq^0_{(\afat,\bfat)}\lra\mpq^0_{(\afat',\bfat')}
$
such that the following diagram commutes (see~Theorem~\ref{thm:alterhf})
\begin{equation}
 \xymatrix@C=1.5pc@R=0.5pc{
 \mpq^0_{(\afat,\bfat)}
 \otimes_{\cohom}
 \ztwo
 \ar[rr]^\cong
 \ar[ddd]_{(\cdot\;\boxtimes\fraks_f\boxtimes\os)\otimes\id}
 & &
 \cfhat(\mH)\ar[ddd]_{f}\\
 & \\ \\
 \mpq^0_{(\afat',\bfat')}
 \otimes_{\cohom}
 \ztwo
 \ar[rr]^\cong
 & &
 \cfhat(\mH'),
 }\label{eq:comdiag}
\end{equation}
where $\os$ is an element in the symbol homology which just depends 
on the pair of attaching circles $(\afat',\bfat')$. 
\subsection{Applications} The first construction of a symbol homology is
given in \S\ref{sec:tshp} by decorating moduli spaces that are relevant for
the $\hfhat$-homology. The construction does not depend on a particular
setup or on a particular type of moduli spaces and thus can be done with 
moduli spaces relevant for other Heegaard Floer homologies. These homologies
are then also described in terms which are analogous to \eqref{eq:equiv} 
and \eqref{eq:comdiag}.
We outline the construction for moduli spaces relevant to the
$\hfkhat$-homology (and the $\hfkminus$-homology). To fix notation,
we will denote by $\mS^w$ the symbol algebra which is generated using
moduli spaces relevant for the $\hfkhat$-homology and denote by $\bh^w$ 
the associated symbol homology. Recall
from above that a property of a map $f$ in the $\hfhat$-theory can be 
encoded into a polynomial expression $P$ with coefficients in $\bh$
such that if $\fraks$ is a symbol representing $f$, then $P(\fraks)=0$
implies that $f$ has the property. So, we say that $f$ has property $P$
or $\fraks$ has property $P$. Defining symbol algebras for different
flavors of Heegaard Floer theory naturally give rise to morphisms
between them. In this particular case, we define a morphism
$\mff\co\mS\lra\mS^w$ the so-called {\bf filtering morphism}. We will 
see that this morphism is a chain map
and, thus, descends to a map $\mff_*$ between the associated symbol 
homology theories. With this morphism at hand we are able to prove the 
following statement (cf.~also~Example~\ref{example02}).
\begin{theorem}\label{main} If a map $\fraks$ from a tensor product of 
Heegaard Floer chain complexes to another Heegaard Floer chain complex 
fulfills a property $P$, then the filtered map $\mff_*(\fraks)$ between 
the corresponding $\cfkhat$-knot Floer chain complexes fulfills the filtered 
property $P_\mff$.
\end{theorem}
We point the reader to \S\ref{sec:ppomas} for a precise 
definition of $P_\mff$.
A similar statement can be formulated and proved for a relation between the
$\hfhat$-theory and $\hfkminus$ (see~Theorem~\ref{main2}). 
One of the benefits of the symbol homology theory is that it unifies the 
Floer chain level and the moduli space level into one object 
(see~discussion in \S\ref{sec:eai}, cf.~\S\ref{sec:rhftfsh} 
and \S\ref{sec:shafh}). A consequence of this unification is that 
it provides a systematic and immediate way to transfer properties between 
different Floer theoretic settings without difficulty. This is indicated 
by Theorem~\ref{main} (see~also Theorem~\ref{main2}).
Proofs of properties which need the moduli space machinery now do not 
need to be repeated in different settings but can now just be accepted 
by pointing to the results of this paper. This allows a systematic
transfer of properties between different flavors of Heegaard Floer 
theory and can in principle be done for other Floer homologies using the 
techniques from this paper. To demonstrate how the transfer is set up when 
explicitly worked out, we give a construction of cobordism maps for knot 
Floer homologies and apply the techniques from symbol homologies in the 
proof of the following theorem.
\begin{theorem}\label{cobmapinvar} For a $\spinc$-structure $\fraks$ 
over $W$ we have that
\[
  \Foc_{W;\fraks}
  \co
  \hfkoc(Y,K;\left.\fraks\right|_Y)
  \lra
  \hfkoc(Y',K';\left.\fraks\right|_{Y'})
\]
is uniquely defined up to sign.
\end{theorem}
For a precise definition of all notions and the maps we point the 
reader to \S\ref{sec:knotcob}. For a definition of $\hfkoc$ we point the
reader to \S\ref{knotfloerhomology}. In fact, we also give further demonstrations
in Example~\ref{example02}, Corollary~\ref{knotinv} and in the surgery 
exact triangle given in Theorem~\ref{thm:set} which is a generalization
of \cite[Theorem~8.2]{OsZa04} (cf.~also \cite[Theorem~2.7]{OsZa03}).

\subsection{Organization} The best method to read this article is probably to
combine a linear reading with jumps into \S\ref{sec:eai}. In 
\S\ref{sec:eai} we present a somewhat informal introduction to some of the
ideas which will help the reader to familiarize with the objects. In fact,
we introduce some notational conventions for the decorations in that section
which make the whole construction intuitive. Additionally, we present
two calculations of symbol homologies in easy situations 
(see~Example~\ref{ex:calc1} and Example~\ref{ex:calc2}).\\
In \S\ref{sec:wpash} we give the construction of a symbol homology 
modeled on the $\hfhat$-theory. In \S\ref{sec:shafh} we provide the 
operation $\ev$ that connects the symbol algebra with the maps 
between Floer chain modules. In \S\ref{sec:fsh} 
we give the construction of the filtered symbol algebra, provide 
the filtering morphism in \S\ref{sec:tfm} and describe 
in \S\ref{sec:ppomas} the implications of the filtering morphism.
In \S\ref{sec:rhftfsh} we give the new presentation of 
Floer theory in terms of our symbol homology introduced in 
\S\ref{sec:wpash}. What is done in that section can potentially be 
done with every other flavor of symbol homology. In \S\ref{sec:uesh} 
we define another flavor of symbol homology capturing the theory 
$\hfkminus$. This should serve as a model for how to introduce a 
$U$-variable into the theory. In \S\ref{sec:psh} we outline how to 
bring moduli spaces with dynamic boundary conditions into the theory. 
This is necessary to apply the symbol homologies in the proof of Theorem~\ref{cobmapinvar}, since the invariances in Heegaard Floer 
theory mainly come from morphisms that are defined by considering 
polygons with dynamic boundary conditions. In \S\ref{sec:knotcob} we 
provide the construction of cobordism maps in knot Floer theory, 
specifically focusing on $\hfkhat$ and $\hfkminus$. What is done 
there can also be applied to all other knot Floer theories with 
slight adaptions (cf.~also \S\ref{sec:okft}). Finally,
in \S\ref{sec:implications} we present a surgery exact triangle for
maps induced by knot cobordisms.
\subsection*{Acknowledgments} We wish to thank Kai Cieliebak for helpful
conversations.

\section{Preliminaries}\label{preliminaries}
\subsection{Heegaard Floer homologies}\label{prelim:01:1}
In Heegaard Floer theory one assigns to a closed, oriented $3$-manifold
the homology groups $\hfhat(Y)$, $\hfminus(Y)$, $\hfinfty(Y)$ 
and $\hfplus(Y)$ (see~\cite{OsZa01}). In the following,
we will give a brief review of the two versions $\hfhat(Y)$ 
and $\hfminus(Y)$.\vspace{0.3cm}\\
A Heegaard diagram $\mH$ is a triple $(\Sigma,\afat,\bfat)$, where 
$\Sigma$ is an oriented genus-$g$ surface 
and $\afat=\{\alpha_1,\dots,\alpha_g\}$, 
$\bfat=\{\beta_1,\dots,\beta_g\}$ are two sets of pairwise disjoint, 
simple, closed curves in $\Sigma$ called {\bf attaching circles}. 
Both $\afat$ and $\bfat$ are required to consist of curves which represent
linearly independent classes in $H_1(\Sigma,\Z)$. In the 
following we will talk about the curves in the set $\afat$ (resp.~$\bfat$) as  
{\bf $\afat$-curves} (resp.~{\bf $\bfat$-curves}). Without loss 
of generality, we may assume that the 
$\afat$-curves and $\bfat$-curves intersect 
transversely.\\
Now suppose we are given a closed, oriented $3$-manifold $Y$. To $Y$ we 
associate a Heegaard diagram $\mH$ and use it to construct a homology theory
as follows: To the Heegaard diagram $\mH$ we associate the triple $(\symg,\talpha,\tbeta)$ consisting of 
the $g$-fold symmetric power
\[
  \symg=\Sigma^{\times g}/S_g
\] 
of $\Sigma$, the submanifold $\talpha=\alpha_1\times\dots\times\alpha_g$ of
$\symg$ and the submanifold $\tbeta=\beta_1\times\dots\times\beta_g$ of
$\symg$. We define $\cfminus(\mH)$ as the $\ztwo[U]$-module generated by 
the set $\talpha\cap\tbeta$. A map $\phi\co D^2\lra\symg$ ($D^2\subset\C$) 
is called {\bf Whitney disk} if $\phi(D^2\cap\{Re<0\})\subset\talpha$ 
and $\phi(D^2\cap\{Re>0\})\subset\tbeta$. We call 
$D^2\cap\{Re<0\}$ the {\bf $\afat$-boundary of $\phi$} and
$D^2\cap\{Re>0\}$ the {\bf $\bfat$-boundary of $\phi$}. For two 
points $\xfat,\yfat\in\talpha\cap\tbeta$ we say that a 
Whitney disk {\bf connects $\xfat$ with $\yfat$} if $\phi(i)=\xfat$ 
and $\phi(-i)=\yfat$. Denote by $\pitwo(\xfat,\yfat)$ 
the set of homology classes of Whitney disks that connect $\xfat$ with $\yfat$.
Note that $\pitwo(\xfat,\yfat)$ can be interpreted as a subgroup of 
$H_2(\symg,\talpha\cup\tbeta)$. We endow 
$\symg$ with a symplectic structure~$\omega$. By choosing a path of 
almost complex structures $\mJ_s$ on $\symg$ suitably (cf.~\cite{OsZa01}),
all moduli spaces of holomorphic Whitney disks are Gromov-compact manifolds.
We fix a point $z\in\Sigma\backslash(\afat\cup\bfat)$ and use it to define 
the map 
$
  n_z\co\pitwo(\xfat,\yfat)\lra\Z
$
which assigns to a Whitney disk $\phi$ its intersection number with the
submanifold $V_z=\{z\}\times\symgmo$. In fact, the path of 
almost complex structures $\mJ_s$ is chosen in such a way
that $V_z$ is a complex submanifold of the symmetric product. A path of
almost complex structures for which $V_z$ is a complex submanifold is called 
{\bf $z$-respectful}.
For two points $\xfat$, $\yfat\in\talpha\cap\tbeta$ denote by 
$\mMhat^i_{(\afat,\bfat)}(\xfat,\yfat)$ the set of unparametrized 
holomorphic Whitney disks $\phi$ which connect $\xfat$ with $\yfat$ such that $n_z(\phi)=i$. For a point $\xfat\in\talpha\cap\tbeta$ we define 
\[
 \parminus_\mH\xfat
  =
  \sum_{\yfat\in\talpha\cap\tbeta,i\geq0}
  \!\!\!\!\#\Bigl(\mMhat^{i}_{(\afat,\bfat)}(\xfat,\yfat)\Bigr)\cdot
  U^{i}\yfat
\]
and we extend $\parminus_\mH$ to $\cfminus(\mH)$ as a morphism of 
$\ztwo[U]$-modules. The map $\parminus_\mH$ is a differential. The 
associated homology theory $H_*(\cfminus(\mH),\parminus_\mH)$ is denoted 
by $\hfminus(Y)$ and it is a topological invariant of $Y$. By setting $U=0$, 
we obtain a different flavor of this theory: The associated chain 
module is denoted by $\cfhat(\mH)$ and it can be interpreted as the 
$\ztwo$-module generated by $\talpha\cap\tbeta$. Restricting $\parminus_\mH$
to the module $\cfhat(\mH)$, we obtain a map 
$\parhat_\mH\co\cfhat(\mH)\lra\cfhat(\mH)$, which is still a differential.
In fact, for every $\xfat\in\talpha\cap\tbeta$, the equality
\[
 \parhat_\mH\xfat
  =
  \sum_{\yfat\in\talpha\cap\tbeta}
  \!\!\!\!\#\Bigl(\mMhat^{0}_{(\afat,\bfat)}(\xfat,\yfat)\Bigr)\cdot \yfat 
\]
holds. The homology theory $H_*(\cfhat(\mH),\parhat_\mH)$ will be denoted 
by $\hfhat(Y)$. 
We would like to note that not all Heegaard diagrams are suitable
for defining the Heegaard Floer homology groups. There is an additional
condition that has to be imposed called {\it admissibility}. 
A detailed knowledge of this condition 
is not important in the remainder of the present article since all
constructions are done nicely so that there will never be a 
problem. We advise the interested reader to \cite{OsZa01} .

\subsection{Knot Floer Homology}\label{knotfloerhomology}
Given a knot $K\subset Y$, we can specify a certain subclass of 
Heegaard diagrams.
\begin{definition} \label{knotdiagram} A Heegaard 
diagram $(\Sigma,\afat,\bfat)$ is said to
be {\bf adapted} to the knot $K$ if $K$ is isotopic to a knot lying
in $\Sigma$ and $K$ intersects $\beta_1$ once transversely and is
disjoint from the other $\bfat$-curves.
\end{definition}
Every pair $(Y,K)$ admits a Heegaard diagram adapted to $K$.
Having fixed such a Heegaard diagram $(\Sigma,\afat,\bfat)$, we can encode 
the knot $K$ in a pair of points: After isotoping $K$ onto $\Sigma$, 
we fix a small interval $I$ in $K$ containing the intersection point 
$K\cap\beta_1$. This interval is chosen small enough such 
that $I$ does not contain any other intersections of $K$ with other 
attaching curves. The boundary $\partial I$ lies in the complement
of the attaching circles and consists of two points we denote by $z$ 
and $w$ such that $\partial I=z-w$ as chains. Here, the orientation of 
$I$ is given by the knot orientation. In this way, we associate to the pair
$(Y,K)$ a doubly-pointed Heegaard diagram $(\Sigma,\afat,\bfat,w,z)$.\\
Conversely, given a doubly-pointed Heegaard diagram $(\Sigma,\afat,\bfat,w,z)$,
we denote by $Y$ the manifold represented by the underlying Heegaard diagram.
We connect $w$ with $z$ with an arc $\delta$ in  $\Sigma\backslash(\afat\cup\bfat\backslash\beta_1)$ that crosses 
$\beta_1$, once. Then, we connect $z$ with $w$ in $\Sigma\backslash\bfat$
using an arc $\gamma$. The union $\delta\cup\gamma$ is a knot $K$ we equip
with the orientation such that $\partial\delta=z-w$. Hence, we obtain a 
pair $(Y,K)$. \vspace{0.3cm}\\
Suppose we are given a doubly-pointed Heegaard diagram 
$\mH=(\Sigma,\afat,\bfat,w,z)$. The knot chain module $\cfkminus(\mH)$ is 
the free $\ztwo[U]$-module generated by the intersection points 
$\talpha\cap\tbeta$. Analogous to the definitions given above, 
we define $\mMhat^{(i,j)}_{(\afat,\bfat)}(\xfat,\yfat)$ as the set of 
unparametrized holomorphic Whitney disks $\phi$ that connect $\xfat$ with 
$\yfat$ such that $(n_z(\phi),n_w(\phi))$ equals $(i,j)$. For $\xfat\in\talpha\cap\tbeta$ we define
\[ 
 \partial^{\bullet,-}_{\mH}\xfat 
 =
 \sum_{\yfat\in\talpha\cap\tbeta,j\geq0} 
 \!\!\!\!\#\Bigl(\mMhat^{(0,j)}_{(\afat,\bfat)}(\xfat,\yfat)\Bigr)
 \cdot U^{j}\yfat.
\]
We extend $\partial^{\bullet,-}_{\mH}$ to $\cfkminus(\mH)$ 
as a morphism of $\ztwo[U]$-modules. The associated homology theory 
$H_*(\cfkminus(\mH),\partial^{\bullet,-}_\mH)$ is denoted by 
$\hfkminus(Y,K)$ and is an invariant of the pair $(Y,K)$. By setting 
$U=0$ as before, we obtain a new theory which we denote by $\hfkoo(Y,K)$ 
or, alternatively, $\hfkhat(Y,K)$. It is also
possible to define variants such as $\hfkplus$ and 
$\hfkinfty$. For details we point the reader to \cite{OsZa03}.
\vspace{0.3cm}\\
To justify our notation, observe, that it is possible to {\it swap the roles}
of $z$ and $w$ by defining a differential $\partial^{-,\bullet}_\mH$ via 
\[
 \partial^{-,\bullet}_{\mH}\xfat 
 =
 \sum_{\yfat\in\talpha\cap\tbeta,j\geq0} 
 \!\!\!\!\#\Bigl(\mMhat^{(j,0)}_{(\afat,\bfat)}(\xfat,\yfat)\Bigr)
 \cdot U^{j}\yfat.
\] 
The associated homology theory is denoted by $\hfkmminus(Y,K)$. As in 
the previous case, we can also define the variants $\hfkminfty$ and 
$\mbox{\rm HFK}^{+,\bullet}$.

\section{The Symbol Homology Package}\label{sec:tshp}
\subsection{Whitney Polygons and Symbol Homology}\label{sec:wpash} 
Let $\Sigma$ be a surface of genus $g$ and $\afat_1,\dots,\afat_n$ sets 
of attaching circles on this surface. A map
\[
 \phi
 \co
 D^2
 \lra
 \symg
\]
with $p_i\in\partial D^2$, $i=1,\dots,n$ with boundary conditions 
in $\afat_1,\dots,\afat_n$ is called {\bf Whitney polygon} of 
degree $n$. We may think of $D^2$ itself as a polygon with 
vertices $p_1,\dots,p_n$. For $B=\{\afat_1,\dots,\afat_n\}$ we consider
the disjoint union
\[
 \bigsqcup_{i=2,\dots,|B|}
 B^{\times i}\backslash\Delta_i
\]
where $\Delta_i$ is the diagonal in $B^{\times i}$. Every element
$a\in B^{\times i}\backslash\Delta_i$ specifies boundary conditions on the
edges of the polygon $\phi$ by the following algorithm: For 
$i=1,\dots, n-1$ the $i$-th component of $a$, $\afat$ say, specifies 
that the edge of $\phi$ between the vertices $p_i$ and $p_{i+1}$ 
has to be mapped into $\talpha$. Analogous, the $n$-th component 
of $a$ specifies the boundary condition of the edge between $p_{n}$ 
and $p_1$.\\
We say that two elements $a$, $b$ of $B^{\times i}\backslash\Delta_i$ 
are equivalent if the boundary conditions they specify can be identified 
by a rotation of the polygon. This defines an equivalence relation $\sim$
and we denote by $\mI_B$ the set of equivalence classes
\[
 \Bigl(\bigsqcup_{i=2,\dots,|B|}
 B^{\times i}\backslash\Delta_i\Bigr)/\sim.
\]
We call $\mI_B$ the {\bf index set} or the {\bf set of boundary conditions}. 
For $B'\in\mI_B$ we denote by 
$
 \pitwo(\xfat_1,\dots,\xfat_n;B')
$
the set of homotopy classes of Whitney polygons connecting the 
$\xfat_i$ with boundary conditions given by $B'$. 
We denote by
$
 \gafull_{B'}(\xfat_1,\dots,\xfat_l)
$
the space of $\mJ_s$-holomorphic Whitney polygons that 
connect the $\xfat_i$ with boundary conditions given by $B'$ and 
Maslov-index $\mu$.\\
The set $B'$ specifies boundary conditions on the edges of every
polygon $\phi\in\gafull_{B'}$. We can additionally impose conditions
on the vertices of $\phi$ which is usually indicated by attaching 
points into the notation like for instance $\gafull_{B'}(\dots,\qfat,\dots)$.
If conditions on the vertices are specified we call the space
{\bf pointed}. 
Suppose we are given a pointed moduli space
 $\gafull_{B'}(\dots,\qfat,\dots)$ of $n$-gons. Each vertex of 
the polygons $\phi\in\gafull_{B'}$ can be specified by a
pair $(a_i,a_{i+1})$ of attaching circles $a_i$, $a_{i+1}$ (from $B'$) 
by the following algorithm: The element 
$B'$ specifies boundary conditions on the Whitney polygons as described 
above. If we think of the polygon as sitting in $\R^2$, then the standard
orientation on $\R^2$ and, hence, of the polygon induces an 
orientation on the boundary, i.e.~a preferred direction of travel along the boundary.
The vertex specified by $(a_i,a_{i+1})$ is the vertex sitting between
the edges to which the boundary conditions $a_i$ and $a_{i+1}$ have been
attached, such that traversing through the boundary of the polygon
in the preferred direction of travel, we will pass the vertex by approaching
from $a_i$ and moving over the vertex to $a_{i+1}$. By abuse of notation,
we will call the pair $(a_i,a_{i+1})$ {\bf vertex}. A set $\{(a_i,a_{i+1}),\qfat\}$
consisting of a vertex $(a_i,a_{i+1})$ and a point $\qfat\in\mbT_{a_i}\cap\mbT_{a_{i+1}}$ is called a {\bf pointing}.
\begin{definition}\label{def:generators}
Let $\mA$ be the symbol $\genA$ consisting of the following 
data:
\begin{enumerate}
\item[(1)] We have that $A$ represents an unparametrized moduli space 
$\gafull_{B'}$ where $B'$ equals $[(a_1,\dots,a_n)]$ which is an 
element of $\mI_B$. Furthermore, we require that $\mu\leq 2$ if $n=2$ 
and $\mu\leq 1$, otherwise.
\item[(2)] The sets $\pointed$ and $\fout$ consist of a collection of 
pointings for $A$. We call $\pointed$ the {\bf set of pointings} and $\fout$ the
set of {\bf flow-out vertices}.
\item[(3)] The set $\fin$ is called the set of 
{\bf flow-in vertices} and consists of a collection of vertices of $A$. 
\end{enumerate}
We require that each vertex of $A$ 
appears in the set $\pointed$, in the set $\fin$
and in the set $\fout$ at most once.
If each vertex of $A$ is either a pointed, a flow-in or a flow-out 
vertex, then we call $\mA$ a {\bf pre-generator}.  A pre-generator 
is called a {\bf generator} if $\#(\fin)>0$ and $\#(\fout)=1$. A 
pre-generator with no flow-in vertices and no flow-out vertices is 
called {\bf fully pointed}.
\end{definition}
For a pre-generator $\mA=A_{(\pointed,\fin,\fout)}$ we 
define $\pi(\mA)$ to be the moduli space we obtain
from $A$ after attaching the slot points 
given in $\pointed$ and $\fout$ as boundary conditions 
to the vertices of the polygons in $A$. Now consider the 
commutative free polynomial $\ztwo$-algebra
generated by fully pointed pre-generators, where we denote the sum 
by $+$ and the product by $\cp$. For each pre-generator $\mA$ with 
the property that $\pi(\mA)=\emptyset$ we introduce the relation 
$\mA=0$. Furthermore, we introduce the relation $\mA=1$ 
if $\pi(\mA)\not=\emptyset$ or if the elements of $A$ are bigons
with Maslov-index $0$. The algebra we obtain after introducing these 
relations will be denoted by $\mbf$ and  called the
{\bf coefficient algebra}. Then, we consider 
the free non-commutative $\mbf$-algebra generated by the 
set of pre-generators which are not fully pointed, using the disjoint union
$+$ as the sum and the Cartesian product $\times$ as the product operation.
Denote by $\mT$ 
the algebra we obtain after introducing the relation $\mA=0$ for every
pre-generator $\mA$ for which $\pi(\mA)$ is empty.
We denote by $\mThat$ the set $\mT\cup\{\ohat\}$.
\begin{definition}\label{def:semialgebra} Denote by $\mR$ a ring.
\begin{enumerate}
\item[(a)] A set $P$ is a {\bf $\mR$-semimodule} if $P$ is a 
semigroup together with a map
\[
 \bolddot\co \mR\times P\lra P
\]
such that for all $v,w\in P$ and $a,b\in\mR$ we have that 
$1_{\mR}\cp v=v$, $a\cp (v+w)=a\cp v+a\cp w$, $a\cp(b\cp v)=(ab)\cp v$ 
and $(a+b)\cp v=a\cp v + b\cp v$.
\item[(b)] A set $(P,+)$ is a {\bf $\mR$-semialgebra} if $P$ 
is a $\mR$-semimodule with a multiplication, $\times$ say, 
which is associative and distribute with  respect to $+$ 
and such that for $a,b\in\mR$ and $v,w\in P$ we have  that 
$(a\cp v)\times(b\cp w)$ equals $ab\cp (v\times w)$.
\end{enumerate}
A $\mR$-semialgebra will also be called {\bf semialgebra} if the 
ring $\mR$ is specified in the context.
\end{definition}
It will be our goal to define a $\mbf$-semialgebra structure 
on $\mThat$. We will do this in the following, but we need 
to provide a couple of definitions beforehand so that we will be 
able to give clean definitions. The semialgebra structure will 
be modeled on the disjoint union and the Cartesian product, but 
the they will also include modifying the 
decorations (cf.~\S\ref{sec:eai}).
\begin{definition}\label{def:operators} Suppose we are given a 
pre-generator $\mA=\genA$.
\begin{itemize}
\item[(a)] For a set $Q$ consisting of pointings (like for 
instance $\fout$), we define $\mD(Q)$ as the set we obtain from 
$Q$ by forgetting the slot points.
\item[(b)] Given a set $Q$ consisting of pointings, we define 
the operator $\mKind$ as follows:
\[ 
\mKind(\pointed,\fin,\fout)
=
(\pointed\cup Q,\fin\backslash\mD(Q),\fout).
\]
\item[(c)] We define the operator $\mKoutq$ as follows:
\[
\mKoutq(\pointed,\fin,\fout)
=
(\pointed\cup\fout,\fin,\emptyset).
\]
\item[(d)] Given another pre-generator 
$\mB=B_{(\pointedprime,\finprime\foutprime)}$,
we call the vertices 
in $\mD(\fout)\cap\finprime$ {\bf the common 
vertices of $\mA$ and $\mB$}.
\end{itemize}
Correspondingly, we 
define $\mKind(\mA)=A_{\mKind(\pointed,\fin,\fout)}$ and $\mKoutq(\mA)=A_{\mKoutq(\pointed,\fin,\fout)}$.
\end{definition}
Let $\mC=\mC_1\times\dots\times\mC_k$ be a product 
of pre-generators. Given a vertex $s$, we define
\[
\inmultc(s)
=
\#
\bigl\{
q\in\bigsqcup_{i=1}^k\fin(\mC_i)
\,
|
\,
q=s
\bigr\}
\]
which we call the {\bf inward multiplicity of $\mC$ at $s$}. 
Correspondingly, we define
\[
\outmultc(s)
=
\#
\bigl\{
q\in\bigsqcup_{i=1}^k\mD(\fout(\mC_i))
\,
|
\,
q=s
\bigr\}
\]
and call it the {\bf outward multiplicity of $\mC$ at $s$}.
\begin{lem}\label{multres01} For a product of pre-generators 
$\mC=\mC_1\times\dots\times\mC_k$ we have that
\[
  \inmultc(s)=\sum_{i=1}^k\mm^{\mC_i}(s)
  \;\;\mbox{\rm and }\;\;
  \outmultc(s)=\sum_{i=1}^k\mm_{\mC_i}(s).
\]
\end{lem}
\begin{proof} This readily follows from the definition of 
the inward multiplicities and the outward multiplicities.
\end{proof}
Now we will define a sum on $\mThat$: First, 
we set 
$f\cp\mA\boxplus\ohat=f\cp\ohat\boxplus\mA=\ohat$ 
for all $\mA\in\mThat$ and all $f\in\mathbb{F}$. 
Second, for given products of pre-generators 
$\mC=\mC_1\times\dots\times\mC_k$, $\mR=\mR_1\times\dots\times\mR_l$
and elements $f,g\in\mbf$ we define
\[
  f\cp\mR\boxplus g\cp\mC
  =
  \left\{
  \begin{array}{cll}
  f\cp\mR+g\cp\mC 
  &\mbox{\rm , if }
  \inmultc\equiv\inmultr\;\mbox{\rm and }\;\outmultc\equiv\outmultr\\
  \ohat
  & \mbox{\rm , otherwise} &
  \end{array}
  \right. .
\]
Finally, for elements $\sum_{i=1}^k f_i\cp q_i$ and 
$\sum_{j=1}^l g_j\cp r_j$ where the $q_i$ and $r_j$ are defined as 
products of pre-generators and the $f_i$ and $g_j$ are elements of 
$\mathbb{F}$, we have that
\[
  \Bigl(\sum_{i=1}^k f_i\cp q_i\Bigr)
  \boxplus
  \Bigl(\sum_{j=1}^l g_j\cp r_j\Bigr)
  =
  \Bigl(\bigboxplus{i=1,\dots,k} f_i\cp q_i\Bigr)
  \boxplus
  \Bigl(\bigboxplus{j=1,\dots,l} g_j\cp r_j\Bigr)
\]
These definitions uniquely define a map
\[
 \boxplus
 \co
 \mThat
 \times\mThat
 \lra\mThat.
\]
As we see from the definition, whether a $\boxplus$-sum of elements 
equals $\ohat$ or not depends on the decorations of the polygons. If 
the $\boxplus$-sum does not equal $\ohat$, we say that {\bf the data of 
the summands match}.
\begin{prop}\label{intro:plus} The map $\boxplus$ is commutative 
and associative, i.e.~$(\mThat,\boxplus)$ is a commutative semigroup.
\end{prop}
\begin{proof} 
Commutativity of $\boxplus$ readily follows from its definition. To prove 
associativity, suppose we are given elements
\[ 
  a
  =
  \sum_{i=1}^k f_i\cp q_i,
  \;\;
  b=\sum_{j=1}^l g_j\cp r_j,
  \;\;\mbox{\rm and}\;\;
  c=\sum^m_{n=1} h_n\cp s_n
\]
where the $q_i$, $r_j$ and $s_n$ are suitable products of pre-generators 
and the $f_i$, $g_j$ and $h_n$ suitable elements in the coefficient 
algebra $\mbf$. Then, $(a\boxplus b)\boxplus c$ equals 
$\ohat$ unless all the inward multiplicities of the $q_i$, $r_j$ 
and $s_n$ agree and all outward multiplicities of the $q_i$, $r_j$ 
and $s_n$ agree. But this is equivalent to the statement that 
$a\boxplus(b\boxplus c)$ does not equal $\ohat$. Hence, if the 
data match in the specified sense, we have that
\[
 (a\boxplus b)\boxplus c
 =
 (a+b)+c 
 =
 a+(b+c)
 =
 a\boxplus(b\boxplus c)
\]
which completes the proof.
\end{proof}
The definition of a product on the set $\mThat$ is slightly more 
complicated. We point the reader to \S\ref{sec:eai} for a discussion
of the idea behind the product. As before, for 
all $\mA\in\mThat$ and $f\in\mbf$ we define
$(f\cp\mA)\boxtimes\ohat=\ohat\boxtimes(f\cp\mA)=\ohat$. For 
pre-generators $\mA$ and $\mB$ with 
$\mD(\fout(\mA))\subset\finprime(\mB)$ and 
elements $f$, $g\in\mbf$, we have that
\begin{equation}
 (f\cp\mA)\boxtimes
 (g\cp\mB)
 =
 \begin{cases}
 \bigl(f\cp\mKoutq(\mA)\bigr)
 \times
 \bigl(g\cp\mKinq(\mB)\bigr),\;\mbox{\rm if $\mKoutq(\mA),\mKinq(\mB)\not\in\mbf$}\\
 \bigl(f\cp\mKoutq(\mA)\bigr)
 \cp
 \bigl(g\cp\mKinq(\mB)\bigr),\;\mbox{\rm if $\mKoutq(\mA)\in\mbf$}\\
 \bigl(f\cp\mKoutq(\mA)\bigr)
 \cp
 \bigl(g\cp\mKinq(\mB)\bigr),\;\mbox{\rm if $\mKinq(\mB)\in\mbf$}\\

\end{cases}
\label{eq:boxdouble}
\end{equation}
and if $\mD(\fout(\mA))\not\subset\finprime(\mB)$, then we set
$
(f\cp\mA)\boxtimes
 (g\cp\mB)
 =\ohat
$.
Now, given products of pre-generators 
$\mC=\mC_1\times\dots\times\mC_k$ and
$\mR=\mR_1\times\dots\times\mR_l$ and elements $f,g\in\mbf$, we 
require that the product
$
 (f\cp\mR)
 \boxtimes
 (g\cp\mC)
$
equals
\[
 (fg)\cp\mR_1\times\dots\times R_{l-1}
 \cp
 (\mR_l\boxtimes \mC_1)
 \cp
 \mC_2\times\dots\times\mC_k
\]
if $\mKoutq(\mR_l)$ and $\mKinq(\mC_1)$ are in $\mbf$ and we
require that it equals
\[
 (fg)\cp\mR_1\times\dots\times\mR_{l-1}
 \times
 (\mR_l\boxtimes\mC_1)
 \times
 \mC_2\times\dots\times\mC_k,
\]
otherwise. Finally, given elements $\mA$, $\mB$ 
and $\mC$ in $\mThat$ we set
\begin{eqnarray*}
  \mA\boxtimes(\mB+\mC)
  &=&
  (\mA\boxtimes \mB)
  \boxplus
  (\mA\boxtimes \mC)\\
  (\mB+\mC)\boxtimes\mA
  &=&
  (\mB\boxtimes \mA)
  \boxplus
  (\mC\boxtimes \mA).
\end{eqnarray*}
These definitions provide a map
\[
 \boxtimes\co\mThat\times\mThat\lra\mThat.
\]
From the definition we see that whether a $\boxtimes$-product 
of elements equals $\ohat$ or not depends on
the decorations of the polygons. If the $\boxtimes$-product of 
elements does not equal $\ohat$, we say 
that {\bf the data of the factors match}, or simply that 
{\bf the data match}.
\begin{prop}\label{intro:times} The map $\boxtimes$ is associative
and $\boxplus$-bilinear, i.e.~the equalities
\begin{eqnarray*}
 \mP
 \boxtimes
 \bigl(
 f\cp\mR\boxplus g\cp\mC
 \bigr)
 &=&
 f\cp(\mP\boxtimes\mR)
 \boxplus
 g\cp(\mP\boxtimes\mC)\\
 \bigl(
 f\cp\mR\boxplus g\cp\mC
 \bigr)
 \boxtimes
 \mP
 &=&
 f\cp(\mR\boxtimes\mP)
 \boxplus
 g\cp(\mC\boxtimes\mP)
\end{eqnarray*}
hold for $\mP$, $\mR$, $\mC\in\mThat$ and $f,g\in\mbf$. Hence, the triple
$(\mThat,\boxplus,\boxtimes)$ is a $\mbf$-semialgebra.
\end{prop}
\begin{proof} 
For pre-generators $\mA$, $\mB$ and $\mC$
the following equalities hold:
\begin{eqnarray*}
  (\mA\times\mB)
  \boxtimes
  \mC
  &=&
  \mA
  \times
  (\mB\boxtimes\mC)\\
  \mA
  \boxtimes
  (\mB\times\mC)
  &=&
  (\mA\boxtimes\mB)
  \times
  \mC
\end{eqnarray*}
Hence, if all data match and all factors that appear in \eqref{eq:chain01} 
are not fully pointed, we have the following chain of equalities
\begin{equation}
\begin{array}{ccl}
 \bigl(\mA\boxtimes\mB\bigr)\boxtimes\mC
 &=&
 \Bigl(
 \mKoutq(\mA)
 \times
 \mKinq(\mB)
 \Bigr)
 \boxtimes\mC\\
 &=&
 \hspace{0.2cm}
 \mKoutq(\mA)
 \times
 \bigl(
 \mKinq(\mB)
 \boxtimes\mC
 \bigr)\\
 &=&
 \hspace{0.2cm}
 \mKoutq(\mA)
 \times
 \mKoutq(\mKinq(\mB))
 \times
 \mKinr(\mC)\\
 &=&
 \hspace{0.2cm}
 \mKoutq(\mA)
 \times
 \mKinq(\mKoutq(\mB))
 \times
 \mKinr(\mC)\\
 &=&
 \mA\boxtimes\bigl(\mKoutq(\mB)\times\mKinr(\mC)\bigr)\\
 &=&
 \mA\boxtimes\bigl(\mB\boxtimes\mC\bigr)
\end{array}\label{eq:chain01}
\end{equation}
where the fourth equality holds since 
$\mKoutq(\mKinq(\mB))=\mKinq(\mKoutq(\mB))$. If some of the factors 
that appear in \eqref{eq:chain01} are 
fully pointed, then some of the $\times$-symbols 
have to be replaced by $\cp$, but the calculation remains essentially 
the same. If some of the data do not match, then both
$(\mA\boxtimes\mB)\boxtimes\mC$ and $\mA\boxtimes(\mB\boxtimes\mC)$ 
equal $\ohat$. Hence, we have proved associativity for products of 
pre-generators. The general statement can easily be derived from this 
special case. This proves associativity.\\
To prove the $\boxplus$-linearity, suppose we are given $\mP$, $\mR$ 
and $\mC$ which are all products of pre-generators, i.e.~
\begin{eqnarray*}
\mP&=&\mP_1\times\dots\times\mP_k \\
\mR&=&\mR_1\times\dots\times\mR_l \\
\mC&=&\mC_1\times\dots\times\mC_m.
\end{eqnarray*}
The product $\mP\boxtimes (\mR\boxplus\mC)$ equals $\ohat$ if the 
data of $\mR$ and $\mC$ do not match. This means there is a vertex 
$s$ such that one of the following equalities is violated
\begin{eqnarray}
 \inmultc(s)=\inmultr(s)\label{eq:vio1}\\
 \outmultc(s)=\outmultr(s)\label{eq:vio2}.
\end{eqnarray}
We would like to see that in each of these cases we have
\begin{equation}
 (\mP\boxtimes\mR)\boxplus (\mP\boxtimes\mC)
 =
 \ohat.\label{eq:goal}
\end{equation}
Assuming a violation of (\ref{eq:vio1}), we 
may suppose without 
loss of generality the inequality 
$\inmultr(s)>\inmultc(s)$ to hold. Recall 
from its definition that
\begin{eqnarray*}
 \mP\boxtimes\mR
 &=&
 \begin{cases}
 (\mP\boxtimes\mR_1)\cp\mR_2\times\dots\times\mR_l 
 & 
 \mbox{\rm , if }
 \mKoutq(\mP), \mKinq(\mR_1)\in\mbf\\
 (\mP\boxtimes\mR_1)\times\mR_2\times\dots\times\mR_l 
 & 
 \mbox{\rm otherwise}
 \end{cases}
 \\
 \mP\boxtimes\mC
 &=&
 \begin{cases}
 (\mP\boxtimes\mC_1)\cp\mC_2\times\dots\times\mC_m 
 & 
 \mbox{\rm , if } 
 \mKoutq(\mP), \mKinq(\mC_1)\in\mbf\\
 (\mP\boxtimes\mC_1)\times\mC_2\times\dots\times\mC_m 
 & 
 \mbox{\rm otherwise}
 \end{cases}.
\end{eqnarray*}
Supposing that the data of $\mP$ and $\mR_1$ as well as the data 
of $\mP$ and $\mC_1$ match (the other cases are uninteresting), this 
implies
\[
 \mathfrak{m}^{\mP\boxtimes\mR_1}(s)-\mathfrak{m}^{\mR_1}(s)
 =
 \mathfrak{m}^{\mP\boxtimes\mC_1}(s)-\mathfrak{m}^{\mC_1}(s).
\]
By Lemma~\ref{multres01} we conclude
$
 \mm^{\mP\boxtimes\mR}(s)
 >
 \mm^{\mP\boxtimes\mC}(s)
$
which shows that (\ref{eq:goal}) holds in this case. Assuming a violation 
of (\ref{eq:vio2}), there are two cases to consider. If either 
the data of $\mP$ and $\mR_1$, or the data of $\mP$ and $\mC_1$ do not 
match, the equality (\ref{eq:goal}) is satisfied. If the data of 
$\mP$ and $\mR_1$ as well as the data of $\mP$ of $\mC_1$ match, we 
have
\[
  \mm_{\mP\boxtimes\mR}(s)
  =
  \mm_{\mR}(s)
  >
  \mm_{\mC}(s)
  =
  \mm_{\mP\boxtimes\mC}(s)
\]
which implies (\ref{eq:goal}). 

Supposing that both 
(\ref{eq:vio1}) and (\ref{eq:vio2}) are fulfilled, we have the following 
chain of equalities
\[
  \mP\boxtimes (\mR\boxplus\mC)
  =
  \mP\boxtimes (\mR+\mC)
  =
  (\mP\boxtimes\mR)
  \boxplus
  (\mP\boxtimes\mC).
\]
This shows that 
the product $\boxtimes$ is $\boxplus$-linear in the second component. An analogue
line of arguments shows the linearity in the first component.
\end{proof}
\begin{definition} Let $A$ be a $R$-semialgebra and $B$ a $T$-semialgebra.
A map $\phi\co A\lra B$ is called a {\bf $(R,T)$-semialgebra morphism} if
there is a morphism $\psi\co R\lra T$ such that
\begin{eqnarray*}
 \phi(r\cp a)
 &=&
 \psi(r)\cp\phi(a)\\
 \phi(a+b)
 &=&
 \phi(a)+\phi(b)\\
 \phi(a\times b)
 &=&
 \phi(a)\times\phi(b)
\end{eqnarray*}
for $a\in A$, $b\in B$ and $r\in R$. If $R=T$, then we call $\phi$
an {\bf $R$-semialgebra morphism} and in case the coefficients are defined 
in the context just {\bf semialgebra morphism}.
\end{definition}
Define a map
\begin{equation}
 \Phi
 \co
 (\mThat,+,\times)
 \lra
 (\mThat,\boxplus,\boxtimes)
 \label{eq:gensymbal}
\end{equation}
by sending $+$ to $\boxplus$, $\times$ to $\boxtimes$ and
by requiring that 
$\Phi(\ohat)=\ohat$. This is a $\mbf$-semialgebra morphism. 
The image
of this morphism, which is a $\mbf$-semialgebra, will be the set we 
are interested in.
\begin{definition}\label{def:symalg}
We denote by $\mG$ the set of all generators and denote by $\mQ$ 
the subalgebra of $\mT$ generated by the elements of $\mG$. We denote 
by $\mS$ the set $\Phi(\mQ)\cap\mT$ and by $\mShat$ the 
set $\Phi(\mQ)$. We call $\mShat$ the {\bf symbol algebra}.
\end{definition}
Alternatively, we can think of $\mShat$ as the 
$\mbf$-subsemialgebra of $(\mThat,\boxplus,\boxtimes)$ which is 
generated by the elements of $\mG$.\vspace{0.3cm}\\
\subsection{The Differential on the Symbol Algebra}\label{thediff}
Recall, that a generator is a moduli space with
auxiliary data attached to it. As such, we can look at its
codimension-$1$-boundary. The moduli spaces of Whitney polygons are
Gromov compact manifolds, because of the existence of an energy bound
which was shown in \cite{OsZa01}. Hence, approaching the boundary
of a moduli space of Whitney polygons --~a priori~-- disks break and
spheres bubble off. In our case, bubbling can be ruled out easily due
to the fact that holomorphic spheres in the symmetric product all have
non-zero intersection number $n_z$. Our goal will be to use Gromov's
notion of codimension-$1$-boundary to introduce a map on the
symbol algebra, whose square vanishes identically:
We define $\parsh(\ohat)=0$ for a start. Now suppose we are given a 
generator $\mA=A_{(\pointed,\fin,\fout)}$. 
We define $\parcoone(\pi(\mA))$ as the codimension-$1$ boundary 
of $\pi(\mA)$ (see~\S\ref{sec:wpash} for a definition of $\pi$). 
We know from Gromov 
compactness that the boundary 
components are Cartesian products of moduli spaces of Whitney 
polygons. Each of the boundary components is obtained in the 
following way. Each pair $(a,b)$ with $a$, $b\in B'$ such 
that $a\not=b$ specifies a pair of edges in the polygon. Here,
traversing along the boundary in the preferred direction (i.e.~given by
the orientation induced on the boundary) starting
at the flow-out vertex, we first move over $b$ and then over $a$.
Given such a pair $(a,b)$, we choose a path $\gamma$ from 
the $a$-boundary to the $b$-boundary and then shrink it to a point. The 
polygon will 
split into two polygons $A_1$ and $A_2$, where each of these polygons 
carryies a new vertex, $v_1$ say for $A_1$ and $v_2$ say 
for $A_2$ (cf.~Figure~\ref{Fig:break}). These
polygons represent maps $\phi_i$, $i=1,2$.
\begin{figure}[t!]
\definecolor{myred}{HTML}{B00000}
\labellist\small\hair 2pt
%
%
\pinlabel{$q_2$} [br] at 118 920
\pinlabel{$q_6$} [bl] at 562 920
\pinlabel{$q_1$} [r] at 38 841
\pinlabel{$p_1$} [l] at 650 841
\pinlabel{$A$} at 332 787
\pinlabel{$b'$} [l] at 707 745
\pinlabel{$q_5$} [l] at 643 728
\pinlabel{$q_3$} [tr] at 118 636
\pinlabel{$q_4$} [tl] at 560 636
\pinlabel{{\color{myred} $\gamma$}} [t] at 126 533
\pinlabel{$a$} [t] at 210 571
\pinlabel{$b$} [t] at 259 571
\pinlabel{$a'$} [t] at 517 571
\pinlabel{{\color{myred} $\mu$}} [t] at 641 556
%
%
\pinlabel{$q_6$} [bl] at 916 640
\pinlabel{$p_1$} [l] at 950 577
\pinlabel{$q_5$} [l] at 1100 551
\pinlabel{$A^{a'b'}_2$} at 824 551
\pinlabel{$q'$} [l] at 902 478
\pinlabel{$q'$} [r] at 1025 478
\pinlabel{$q_4$} [l] at 1097 428
\pinlabel{$A^{a'b'}_1$} [t] at 1016 432
%
%
\pinlabel{$q_2$} [br] at 126 406
\pinlabel{$q_6$} [bl] at 645 406
\pinlabel{$q_1$} [r] at 38 317
\pinlabel{$p_1$} [l] at 738 317
\pinlabel{$A^{ab}_1$} at 113 260
\pinlabel{$q$} [l] at 208 260
\pinlabel{$q$} [r] at 308 260
\pinlabel{$A^{ab}_2$} at 524 260
\pinlabel{$q_5$} [l] at 734 204
\pinlabel{$q_3$} [tr] at 126 111
\pinlabel{$q_4$} [tl] at 648 111
\pinlabel{We obtain two new vertices} [l] at 583 19
\endlabellist
\centering
\includegraphics[width=12.6cm]{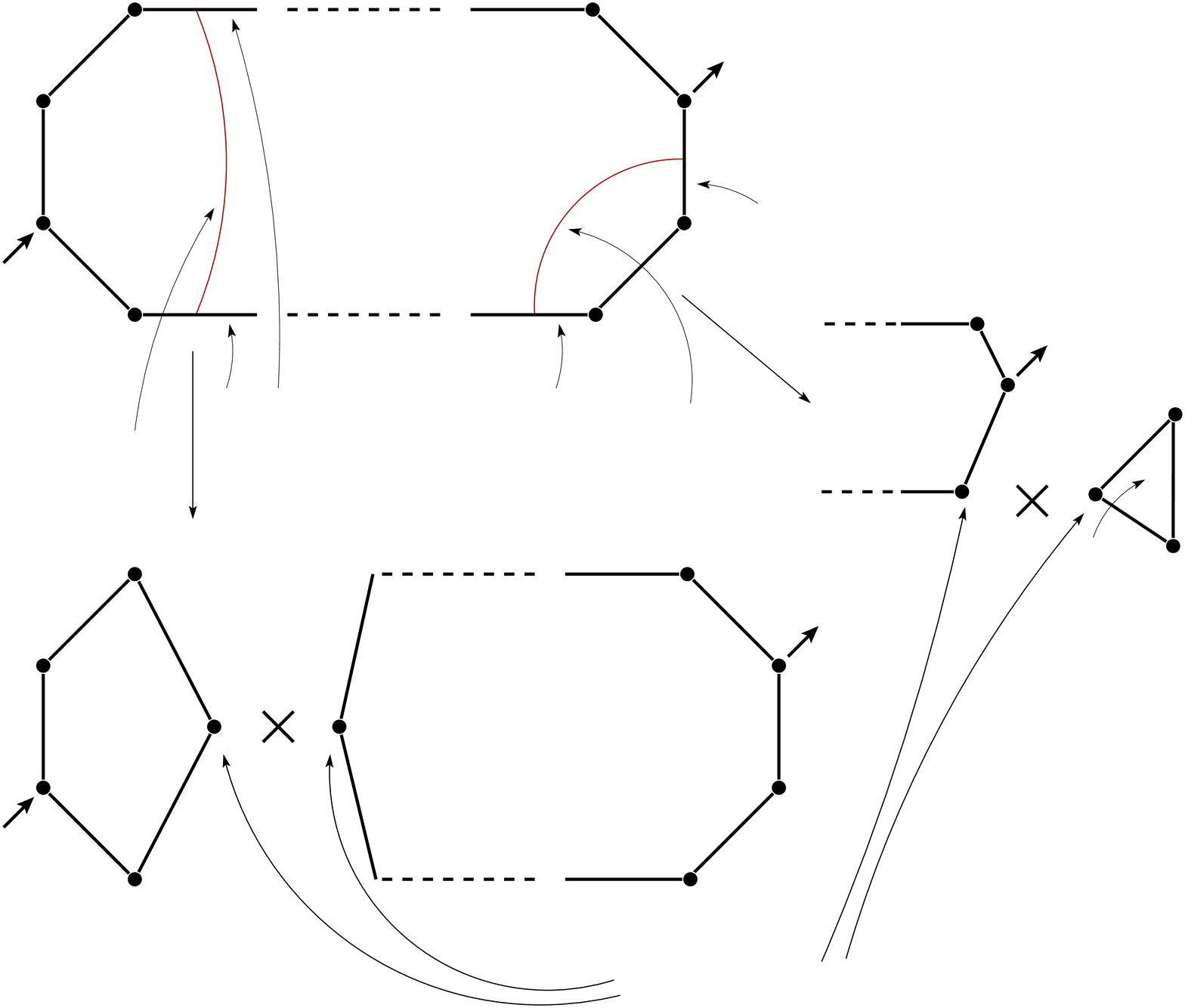}
\caption{Here we see the two cases that might occur in the boundary. 
The shrinking of $\gamma$ gives a product of moduli spaces which is
a $\boxtimes$-product of generators, whereas the shrinking of
$\mu$ gives a product of moduli spaces of whom one is a 
fully pointed pre-generator.}
\label{Fig:break}
\end{figure}
Furthermore, we know that $\phi_1(v_1)=\phi_2(v_2)$ is a point
in $\mbT_a\cap\mbT_b$. We denote this point by $\qfat$. 
The associated boundary component of $\pi(A)$ 
(see~\S\ref{sec:wpash} for a definition of $\pi$) is 
the Cartesian product of two moduli spaces we denote by $A^{ab;\qfat}_1$ and 
$A^{ab;\qfat}_2$. Hence, we have that
\[
 \parcoone\pi(\mA)
 =
 \bigsqcup_{
 \begin{array}{c}
  \scriptstyle{ a,b\in B'}\\ \scriptstyle{a\not=b}
 \end{array}
  }
 \bigsqcup_{q\in\mathbb{T}_a\cap\mathbb{T}_b}
 A^{ab;\qfat}_1
 \times
 A^{ab;\qfat}_2.
\]
Let $V_i$ be the set of vertices of $A^{ab;\qfat}_i$. We decorate the 
vertices in $V_i\backslash\{v_i\}$ with the data from the corresponding
vertices of $\mA$.
In this way, we define 
$\pointed_i$, $\fin_i$ and $\fout_i$ which provide decorations for all
vertices but $v_i$. So, it remains to give an algorithm which attaches a 
decoration to the vertex $v_i$:
We know that $\#(\mD(\fout(\mA)))=1$. Thus, either $A^{ab;\qfat}_1$ or 
$A^{ab;\qfat}_2$ has a flow-out vertex. Without loss of generality, we 
assume it is $A^{ab;\qfat}_2$. 
There are two cases we have to consider which are illustrated 
in Figure~\ref{Fig:break}.

{\bf Case 1.} In the first case $A^{ab;\qfat}_1$ has a vertex which
is a flow-in vertex. An example is illustrated in the lower part of 
Figure~\ref{Fig:break}. We obtain the pictured splitting by shrinking the curve 
$\gamma$ to a point. We {\it attach} to $v_1$ a flow-out arrow and
to $v_2$ a flow-in arrow: More precisely, define
\begin{equation}
 \{\{(a,b),\qfat\}\}\cup\fout_1
 \label{eq:datamod01}
\end{equation}
and --~by abuse of notation~-- denote this set by $\fout_1$. Furthermore, we
define
\begin{equation}
 \{(b,a)\}\cup\fin_2
 \label{eq:datamod02}
\end{equation}
and denote this set by $\fin_2$. We remove all points which were 
attached to the vertices of $A_i^{ab;\qfat}$ and then decorate it with 
the data $(\pointed_i,\fin_i,\fout_i)$. This provides us with 
decorated moduli spaces
we denote by $\mA^{ab;\qfat}_i$, $i=1,2$.
Observe that with this construction both $\mA^{ab;\qfat}_1$ and 
$\mA^{ab;\qfat}_2$ are generators.

{\bf Case 2.}
In the second case, $\mA^{ab;\qfat}_1$ carries no flow-in vertex. An 
illustration is given in the right part of Figure~\ref{Fig:break}. 
Here, we contracted the curve $\mu$. As in the first case, we define 
sets of data $(\pointed_i,\fin_i,\fout_i)$ which are given by 
extracting the data from
$\mA$ that correspond to the old vertices. Instead of the modifications 
given in (\ref{eq:datamod01}) and $(\ref{eq:datamod02})$ we perform the 
following modifications: Set
\begin{equation}
 \{\{(a,b),\qfat\}\}\cup\pointed_1
 \;\;\mbox{and}\;\;
 \{\{(b,a),\qfat\}\}\cup\pointed_2
 \label{eq:datamod03}
\end{equation}
and denote the corresponding set with $\pointed_i$, $i=1,2$. We then continue
as in the first case: We remove the points attached to the vertices of 
$A^{ab;\qfat}_i$, $i=1,2$, and then decorate it with the data $(\pointed_i,\fin_i,\fout_i)$ to obtain pre-generators 
$\mA^{ab;\qfat}_i$, $i=1,2$. Observe that in this case
$\mA^{ab;\qfat}_2$ is a generator but $\mA^{ab;\qfat}_1$ is not. It is
a fully pointed pre-generator and, therefore, an element of the coefficient
algebra $\mbf$. The Cartesian product $A^{ab;\qfat}_1\times A^{ab;\qfat}_2$ thus, 
should be thought of as 
realized by $\mA^{ab;\qfat}_1\cp\mA^{ab;\qfat}_2$.

We apply this algorithm to
all components of $\parcoone(\pi(\mA))$. 
A pair of edges is called {\bf nice} if for the associated boundary component 
the first case of the algorithm applies. 
For $\mA\in\mG$ we define
\[
 \parsh(\mA)
 =
 \Bigl(
 \!\!\!\!
 \bigboxplus{
 \begin{array}{c}
 \scriptstyle{(a,b)\,\mbox{\rm\tiny nice}}\\
 \scriptstyle{\qfat\in\mbT_a\cap\mbT_b}
 \end{array}}
 \!\!\!\!
 \mA^{ab;\qfat}_1\boxtimes\mA^{ab;\qfat}_2\Bigr)
 \boxplus
 \Bigl(
 \!\!\!\!\!\!
 \bigboxplus{
 \begin{array}{c}
 \scriptstyle{(a,b)\,\mbox{\rm\tiny not nice}}\\
 \scriptstyle{\qfat\in\mbT_a\cap\mbT_b}
 \end{array}}
 \!\!\!\!\!\!\!\!
 \mA^{ab;\qfat}_1\cp\mA^{ab;\qfat}_2
\Bigr)
 .
\]
Observe that if $\mA$ is a fully pointed pre-generator, the above 
arguments carry over verbatim to define an assignment $\pco(\mA)$ by
\[
 \pco(\mA)
 =
 \sum_{\begin{array}{c}
 \scriptstyle{(a,b)\,\mbox{\rm\tiny not nice}}\\
 \scriptstyle{\qfat\in\mbT_a\cap\mbT_b}
 \end{array}}
 \!\!\!\!\!\!\!\!\!\!
 \mA^{ab;\qfat}_1\cp\mA^{ab;\qfat}_2.
\]
So, we obtain a map
$
 \pco
 \co
 \mbf
 \lra
 \mbf,
$
by requiring that $\pco$ is $\boxplus$-linear and fulfills a 
Leibniz-rule, i.e.~for elements $\mA$, $\mB\in\mbf$ we have
\[
 \pco(\mA\cp\mB)
 =\pco(\mA)\cp\mB+\mA\cp\pco(\mB).
\]
We extend the definition of $\parsh$ from the generators
to the symbol algebra by requiring that for given $\boxtimes$-products
$\mA_1,\dots,\mA_k$ of generators and elements $f_1,\dots,f_k\in\mbf$ 
the following equalities hold:
\begin{equation}
 \begin{array}{rcl}
 \parsh(f_1\cp\mA_1)
 &=&
 \pco(f_1)\cp\mA_1
 \boxplus
 f_1\cp\parsh(\mA_1)
 \\
 \parsh\Bigl(\bigboxplus{i=1,\dots,k}f_i\cp\mA_i\Bigr)
 &=&
 \bigboxplus{i=1,\dots,k}\parsh(f_i\cp\mA_i)
 \\
\parsh(\mA_1\boxtimes\mA_2)
&=&
 \parsh(\mA_1)
 \boxtimes
 \mA_2
 \boxplus
 \mA_1
 \boxtimes
 \parsh(\mA_2).
 \end{array}
 \label{eq:propdiff1}
\end{equation}
This uniquely extends the map onto the symbol algebra.
\begin{theorem}\label{diffdef} The map
$
  \parsh\co\mShat\lra\mShat
$
is a differential, i.e.~$\parsh\circ\parsh=0$, and $\parsh(\mS)\subset\mS$.
\end{theorem}
\begin{proof} 
To see that it is a differential, 
first observe that this statement is nontrivial for those generators only 
whose underlying moduli space is $1$-dimensional. In this case, we 
have that
\begin{eqnarray*}
 \parsh^2(\mA)
 &=&
 \parsh\Bigl(\!\!\!
\bigboxplus{
 \begin{array}{c}
 \scriptstyle{a,b\in B'\mbox{\rm\tiny nice}}\\ 
 \scriptstyle{q\in\mbT_a\cap\mbT_b}
 \end{array}
 }
 \!\!\!\!\!\mA^{ab;\qfat}_1\boxtimes\mA^{ab;\qfat}_2
 \;\;\boxplus\!\!\!\!\!\!\!\!\!
 \bigboxplus{
 \begin{array}{c}
  \scriptstyle{a',b'\in B'\mbox{\rm\tiny not nice}}\\ \scriptstyle{q\in\mbT_{a'}\cap\mbT_{b'}}
 \end{array}
 }
 \!\!\!\!\!\!\!\!\!\!
 \mB^{a'b';\qfat}_1\cp\mB^{a'b';\qfat}_2
 \Bigr)\\
 &=&
 \bigboxplus{
\begin{array}{c}
 \scriptstyle{a,b\in B'\mbox{\rm\tiny nice}}\\ 
 \scriptstyle{q\in\mbT_a\cap\mbT_b}
 \end{array}
}
 \!\!\!\!\!
 \Bigl(
 \parsh(\mA^{ab;\qfat}_1)\boxtimes\mA^{ab;\qfat}_2
 \boxplus
 \mA^{ab;\qfat}_1\boxtimes\parsh(\mA^{ab;\qfat}_2)
 \Bigr)\\
 &\boxplus&
 \!\!\!\!\!
 \bigboxplus{
 \begin{array}{c}
 \scriptstyle{a',b'\in B'\mbox{\rm\tiny not nice}}\\ 
 \scriptstyle{q\in\mbT_{a'}\cap\mbT_{b'}}
 \end{array}
}
 \!\!\!\!\!\!\!\!\!\!\Bigl(
 \pco(\mB^{a'b';\qfat}_1)\cp\mB^{a'b';\qfat}_2
 \boxplus
 \mB^{a'b';\qfat}_1\cp\parsh(\mB^{a'b';\qfat}_2)
 \Bigr)
\\
 &=&
 0
\end{eqnarray*}
where the third equality holds since $\mA^{ab;\qfat}_1$, $\mA^{ab;\qfat}_2$, 
$\mB_1^{a'b';\qfat}$ and $\mB_2^{a'b';\qfat}$ are $0$-dimensional and, 
thus, $\parsh(\mA^{ab;\qfat}_1)$, $\parsh(\mA^{ab;\qfat}_2)$, 
$\pco(\mB_1^{a'b';\qfat})$ and $\parsh(\mB_2^{a'b';\qfat})$ vanish. Given 
generators $\mB_1,\dots,\mB_k$, we would like to check that
\[
 \parsh^2\Bigl(\bigboxtimes{i=1,\dots,k}\mB_i\Bigr)=0.
\]
This can be derived by induction on $k$ by applying the fact 
that $\parsh^2(\mB_i)=0$. Now let $\mB$ be a $\boxtimes$-product of 
generators and $f\in\mbf$, then
we have the following chain of equalities.
\begin{eqnarray*}
 \parsh^2(f\cp\mB)
 &=&
 \parsh(\pco(f)\cp\mB)
 \boxplus
 \parsh(f\cp\parsh(\mB))\\
 &=&
 \bigl(\pco\bigr)^2(f)\cp\mB
 \boxplus
 2\pco(f)\cp\parsh(\mB)
 \boxplus
 f\cp(\pco)^2(\mB)\\
 &=&
 \bigl(\pco\bigr)^2(f)\cp\mB\\
 &=&
 \Bigl(
 \sum_{(a,b),\qfat}
 \pco(\mB_1^{ab;\qfat}\cp\mB_2^{ab;\qfat})
 \Bigr)\cp\mB\\
 &=&
\Bigl(
 \sum_{(a,b),\qfat}
 (\pco(\mB_1^{ab;\qfat})\cp\mB_2^{ab;\qfat}
 +
 \mB_1^{ab;\qfat}\cp\pco(\mB_2^{ab;\qfat}))
 \Bigr)\cp\mB\\
 &=&
 0.
\end{eqnarray*}
The first and the second equality follow from the equations 
given in \eqref{eq:propdiff1}. The third equality is derived using both the 
commutative $\ztwo$-algebra structure of $\mbf$ and the vanishing of
$\parsh^2$ on products of generators. 
The fourth and the fifth equality are immediate from the definition 
of $\pco$ and the sixth equality rests on the fact that 
the $\mB_i^{ab;q}$, for $i=1,2$, are $0$-dimensional. 
Combining all results, we see that $\parsh^2=0$.

The second statement is a consequence of the definition of both the 
symbol algebra and the map $\parsh$.
\end{proof}
Now we have the elements ready to define our object of interest.
\begin{definition} We define the {\bf symbol homology} $\bh$ as 
the homology theory of the chain complex $(\mShat,\parsh)$.
\end{definition}
In fact, as the symbol algebra, the symbol homology carries the
structure of a semialgebra.
\begin{prop} The map $\pco\co\mbf\lra\mbf$ is a differential. Denote by
$\cohom$ the homology theory $H_*(\mbf,\pco)$, then $\bh$ is a
$\cohom$-semialgebra.
\end{prop}
\begin{proof} The vanishing of $\pco\circ\pco$ follows from the considerations
provided in the proof of Theorem~\ref{diffdef}.
Given elements $f\in\mbf$ with
$\pco(f)=0$ and $[\mA]\in\bh$, we would like to see that $[f\cp\mA]$ just
depends on the homology classes of $f$ and $\mA$: Suppose we are 
given an element $g\in\mbf$, then
\[
 \begin{array}{rcl}
 (f+\pco(g))\cp\mA
 &=&
 f\cp\mA\boxplus\pco(g)\cp\mA\\
 &=&
 f\cp\mA\boxplus\pco(g)\cp\mA\boxplus g\cp\parsh(\mA)
 =
 f\cp\mA
 \boxplus
 \parsh(g\cp\mA)
\end{array}
\]
where the second equality holds since $\mA$ is closed and where
the last equality is given by (\ref{eq:propdiff1}). For an element
$\mB$ we have
\[
 \begin{array}{rcl}
 f\cp(\mA\boxplus\parsh(\mB))
 &=&
 f\cp\mA
 \boxplus
 f\cp\parsh(\mB)\\
 &=&
 f\cp\mA
 \boxplus
 f\cp\parsh(\mB)
 \boxplus
 \pco(f)\cp\mB
 =
 f\cp\mA
 \boxplus
 \parsh(f\cp\mB),
\end{array}
\]
where the second equality holds since $\pco(f)=0$ and where the
last equality is given by applying $(\ref{eq:propdiff1})$. 
This shows that the product $\bolddot$ descends to a map
\[
 \bolddot
 \co
 \cohom
 \times
 \bh
 \lra
 \bh.
\]
To show that $\boxplus$ and $\boxtimes$ descend to maps on
the symbol homology we can apply
arguments standard in algebraic topology: this proof follows
the same lines as the proof which shows that wedging on 
differential forms induces a product on cohomology. Therefore, 
we omit these arguments.
\end{proof}

\subsection{Symbol Homology and Floer Homology}\label{sec:shafh}
Suppose we have fixed a set $B$ of attaching circles. For simplicity
we will work with $\ztwo$-coefficients. To every
element $a=(\afat,\bfat)\in B^{\times 2}\backslash\Delta_2$ we can
associate the $\ztwo$-vector space $\cfhat_a=\cfhat(\afat,\bfat)$. 
By building tensor products of these $\cfhat_a$'s we
can construct a wide variety of $\ztwo$-vector spaces. Denote 
by $V$ and $W$ two of them. We define $\map$ as the 
union of all morphisms between $V$ and $W$, where $V$ and $W$ 
vary among all vector spaces we can define as described above.
We set
\[
 \maphat=\map\cup\{\ohat\}
\]
and equip it with a semialgebra structure as follows: We define 
a map
\[
  \pu\co \maphat\times\maphat\lra\maphat
\]
by sending a pair $f$, $g\in\map$ to $f+g$ if the source and
destination of $f$ and $g$ agree and we send the pair to 
$\ohat$, otherwise. If the source and destination agree, we say that
{\bf the data of $f$ and $g$ match}. In addition, we 
define $\ohat\pu f=\ohat$ for every $f\in\maphat$. Furthermore, we 
define a map
\[
  \cu\co\maphat\times\maphat\lra\maphat
\]
in the following way: Suppose we are given a pair $f,g\in\map$. 
If the source of $f$ agrees with the destination
of $g$, then we define $f\cu g=f\circ g$. If $f$ and $g$ are
of the form
\begin{eqnarray*}
 f
 \co
 C_1\otimes\dots\otimes C_k
 &\lra&
 D_i\\
 g\co D_1\otimes\dots\otimes D_i\otimes\dots\otimes D_l
 &\lra& E
\end{eqnarray*}
with $C_1,\dots,C_k,D_1,\dots,D_l,E$ all Heegaard Floer
chain complexes, then we define
\[
 g\cu f
 =
 g\circ
 (
 \mbox{\rm id}_{D_1\otimes\dots\otimes D_{i-1}}
 \otimes
 f
 \otimes
 \mbox{\rm id}_{D_{i+1}\otimes\dots\otimes D_{l}}
 ).
\]
If none of the above cases apply, we set $f\cu g=\ohat$. If $f\cu g$ does not
equal $\ohat$ we say that {\bf the data of $f$ and $g$ match}. Finally, for 
every $f\in\maphat$ we define $\ohat\cu f=f\cu\ohat=\ohat$.
\begin{lem} 
The triple $(\maphat,\pu,\cu)$ is a $\ztwo$-semialgebra.
\end{lem}
The proof rests on the fact that the composition of maps is 
bilinear with respect to taking sums of maps. Furthermore, taking 
sums is an associative operation.
\begin{proof} Given three maps $f$, $g$, $h\in\map$, the triple sum 
$(f\pu g)\pu h$ is not sent to $\ohat$ if any only if all the maps 
have matching data (in the sense defined above). The same holds 
for $f\pu(g\pu h)$. Thus, we get the equality
\[
 (f\pu g)\pu h= f\pu(g\pu h).
\]
The commutativity of $\pu$ follows immediately from its definition 
and the fact that taking sums of maps is commutative. A similar 
discussion shows that $\cu$ is associative. For $f, g, h\in\maphat$ 
the composition $f\cu(g\pu h)$ equals $\ohat$ unless the destination 
and source of both $g$ and $h$ match and the destination of $g$ 
equals the source of $f$. If all data are matching, then we either 
have 
\[
  f\cu(g\pu h)=f\circ(g+h)=f\circ g+f\circ h=f\cu g\pu f\cu h,
\]
or
\begin{eqnarray*}
 f\cu(g\pu h)
 &=&
 f\circ
 (\mbox{\rm id}_A\otimes(g+h)\otimes\mbox{\rm id}_B)\\
 &=&
 f\circ(\mbox{\rm id}_A\otimes g\otimes\mbox{\rm id}_B)
 +
 f\circ(\mbox{\rm id}_A\otimes h\otimes\mbox{\rm id}_B)\\
 &=& 
 f\cu g\pu f\cu h
\end{eqnarray*}
for suitable tensor products $A$ and $B$ of Heegaard Floer chain 
complexes. If the data do not match, then $f\cu(g\pu h)$ 
equals $\ohat$. However, in this case the same is true 
for $f\cu g\pu f\cu h$.
\end{proof}
Define a map 
\begin{equation}
\evf\co\mbf\lra\ztwo
\label{eq:ctdef}
\end{equation}
by sending fully pointed pre-generators $\mA$ to 
$\evf(\mA)=\#(\pi(\mA))$ (see~\S\ref{sec:wpash} for a definition of $\pi$)
and extending $\evf$ as a $\ztwo$-algebra morphism to $\mbf$. This 
{\it counting operation} has a natural counterpart on the symbol algebra
which we explain in the following: Given a generator $\mA=\genA$, we 
introduce the following notation. By
$\cfhat_{\fin}$ we define the tensor product of Heegaard Floer chain
modules determined by the boundary conditions at the vertices in $\fin$.
Similarly, we define $\cfhat_{\fout}$. First of all , we define $\ev(\ohat)$
to be the element $\ohat\in\maphat$. Second of all, suppose we are given a 
generator $\mA=\genA$ with $\qfat_1,\dots,\qfat_k$
the slot points in $\pointed$ and $\rfat$ the slot point in $\fout$,
denote by $\qfu$ the element $(\qfat_1,\dots,\qfat_k)$.
A map
\[
  \ev(\mA)
  \co
  \cfhat_{\!\fin}
  \lra
  \cfhat_{\!\fout}
\]
is defined by sending a generator 
$\xfu=\xfat_1\otimes\dots\otimes\xfat_f$ of $\cfhat_{\fin}$ to
\[
 \ev(\mA)(\xfu)
 =
 \#A(\xfu,\qfu,\rfat)
 \cdot
 \rfat
\]
and extending as a linear map of $\ztwo$-vector spaces. Now, for 
$i=1,\dots,k$, let $\mA_i$ be a $\boxtimes$-product
of generators and denote by $f_i$ an element of $\mbf$. Then, we
require the following equalities
\begin{equation}
 \begin{array}{lcr}
 \ev(f_1\cp\mA_1)
  &=&
  \evf(f_1)\cdot\ev(\mA_1)
 \\
  \ev\Bigl(\bigboxplus{i=1,\dots,k}f_i\cp\mA_i\Bigr)
  &=&
  \sum_{i=1,\dots,k}\evf(f_i)\cdot\ev(\mA_i)
 \\
  \ev(\mA_1\boxtimes\mA_2)
  &=&
  \ev(\mA_2)\cu\ev(\mA_1)
 \end{array}
 \label{eq:mor}
\end{equation}
These definitions provide a map
\[
 \ev\co(\mShat,\boxplus,\boxtimes)\lra(\maphat,\pu,\cu)
\]
which is uniquely determined by the above.

\begin{prop}\label{intro:ev} 
The map $\ev$ vanishes on boundaries, i.e.~$\ev\circ\parsh=0$. Hence, 
it descends to map
$
  \ev_*
  \co
  \bh
  \lra
  \maphat
$
which is a $(\cohom,\ztwo)$-morphism of semialgebras.
\end{prop}
\begin{proof} 
The vanishing of $\ev\circ\parsh$ is a consequence of the fact that 
$1$-dimensional manifolds have an even number of boundary 
components. Hence, $\ev$ induces a map $\ev_*$ on the symbol
homology. The induced map $\ev_*$ is a $(\cohom,\ztwo)$-morphism of
semialgebras since $\ev$ is a $(\mbf,\ztwo)$-morphism of semialgebras.
\end{proof}
Observe that by construction, every map $f$ between Heegaard 
Floer chain complexes that is defined by counting holomorphic polygons 
with suitable boundary conditions admits a preferred element $\fraks_f$ 
in the symbol algebra such that $\ev(\fraks_f)=f$. We call $\fraks_f$
the {\bf canonical symbol} of $f$. Sometimes, by abuse of notation, we
will also refer to $[\fraks]\in\bh$ as the canonical symbol of $f$.

\section{Filtered Symbol Homology}\label{sec:fsh}
Suppose we are given a set $B$ of attaching circles. 
In \S\ref{sec:tshp} our focus lay on moduli spaces of Whitney 
polygons $\phi$ with $n_z(\phi)=0$. Fixing an additional 
point $w$ of the Heegaard surface $\Sigma$ that lies in
the complement of the attaching circles given in $B$, we 
may look at polygons $\phi$ as before, with the additional 
condition $n_w(\phi)=0$ imposed. We call the associated 
moduli spaces {\bf $w$-filtered} to distinguish them from 
the moduli spaces used in \S\ref{sec:tshp}. As a path of almost 
complex structures we choose one which is $w$-respectful 
(see \S\ref{prelim:01:1}). We use the $w$-filtered 
spaces to define the notions of pre-generators, generators 
and fully pointed pre-generators the same way we did in 
Definition~\ref{def:generators} and denote by $\mG^w$ the
set of generators. Then, following the 
construction procedure from \S\ref{sec:wpash} and \S\ref{thediff},
we define the {\bf $w$-filtered symbol algebra} 
$\mShat^w$. The 
associated homology theory is denoted by $\bh^w$ 
and called {\bf $w$-filtered symbol homology}. To fix notation, 
we introduce the following notational conventions: We denote 
by $\mbf^w$ the coefficient algebra and denote by $\mT^w$ the 
non-commutative polynomial $\mbf^w$-algebra generated by the set 
of pre-generators which are not fully pointed.  And finally, 
write $\cohom^w$ for $H_*(\mbf^w,\partial_{\mbf^w})$, the 
coefficient algebra of the $w$-filtered symbol homology. 
\subsection{The Filtering Morphism}\label{sec:tfm}
Given a moduli space $A$ of polygons, we define the $w$-filtered 
space as
\[
 A^w
 =
 \bigl\{
 \phi\in A\,|\, n_w(\phi)=0
 \bigr\}.
\]
With this in place, we construct a map by assigning to a 
pre-generator $\mA=A_{(\pointed,\fin,\fout)}$ the
element $\mff(\mA)=(A^w)_{(\pointed,\fin,\fout)}$ and
extending to a map
\[
  \mff
  \co
  (\mT,+,\times)
  \lra
  (\mT^w,+,\times)
\]
as a $(\mbf,\mbf^w)$-morphism of algebras, where
$
 \mff(\mC\cp\mA)=(C^w)_{(\pointed,\fin,\fout)}\cp\mA
$
for a fully pointed pre-generator $\mC=C_{(\pointed,\fin,\fout)}$ and
pre-generator $\mA$.
It is easy to see that $\mff$ sends generators to generators. 
More precisely, $\mff$ restricts to a bijection between $\mG$ 
and $\mG^w$. 
Hence, with the convention $\mff(\ohat)=\ohat$, the map $\mff$ 
restricts to a map
\[
 \mff
 \co
 \mShat
 \lra
 \mShat^w
\]
between the symbol algebras. 
\begin{theorem}\label{thm:filtmap} The map $\mff\co\mShat\lra\mShat^w$ is 
a $(\mbf,\mbf^w)$-morphism of semialgebras. Furthermore,
if $\mJ_s$ is $w$-respectful, then the map $\mff$ 
is a chain map and, thus, induces a map
\[
  \mff_*
  \co
  \bh
  \lra
  \bh^w.
\]
which is a $(\cohom,\cohom^w)$-morphism of semialgebras.
\end{theorem}
We both call $\mff$ and the induced map $\mff_*$ 
the {\bf filtering morphism}. This morphism is the main object 
of interest in this section.
\begin{proof} Since $\mJ_s$ respects the point $w$, we know 
that $V_w=\{w\}\times\symgmo$ is a complex submanifold 
of $\symg$. Thus, every intersection of a $\mJ_s$-holomorphic
polygon with $V_w$ is positive. Furthermore, we know that the 
intersection number $n_w$ is a homotopical invariant and 
behaves additive under splicing. Thus, the $w$-filtered boundary 
of a moduli space $A$ equals the boundary of the 
$w$-filtered moduli space $A^w$, i.e.~
\[
 \bigl(\parcoone(A)\bigr)^w
 =
 \parcoone(A^w).
\]
Interpreted in the language of the symbol algebra, this translates 
into $\mff\circ\parsh=\parsh^w\circ\mff$. Thus, $\mff$ is chain.
To see that $\mff$ is a morphism of semialgebras, recall that there 
is a semialgebra morphism
\[
  \Phi
  \co
  (\mThat,+,\times)
  \lra
  (\mThat,\boxplus,\boxtimes)
\]
(see (\ref{eq:gensymbal})) with the following property: The image
of $\left.\Phi\right|_\mQ$, i.e.~the image of $\Phi$ restricted to 
the subalgebra $\mQ$ (see~Definition~\ref{def:symalg}), is 
the symbol algebra. It is 
easy to see from its definition that 
$\left.\Phi\right|_{\mQ\backslash\Phi^{-1}(\ohat)}$ is 
injective. Hence, the following square is commutative.
\[
  \xymatrix{
  \mQ\ar[d]_{\Phi}\ar[r]^{\mff}
  & 
  \mQ^w\ar[d]^{\Phi} \\
  \mS(\mG)=\Phi(\mQ)\backslash\ohat\ar[r]^{\mff}
  & 
  \Phi(\mQ^w)\backslash\ohat=\mS(\mG^w)
  }
\]
The map $\Phi$ is a morphism of semialgebras and it surjects 
onto the symbol algebra. Hence, the map $\mff$ restricted
to the symbol algebra is a morphism of semialgebras. Since
$\mff$ is a $(\mbf,\mbf^w)$-morphism of semialgebras, its
induced map in homology is a $(\cohom,\cohom^w)$-morphism of
semialgebras.
\end{proof}

\subsection{Property P of Morphisms/Symbols}\label{sec:ppomas}
Denote by $\bh(\mG)\{X\}$ the non-commutative polynomial algebra
in one variable, defined using the sum $\boxplus$ and
the product $\boxtimes$. Define $\mbp(\mG)$ as the polynomials
of degree at least one.
\begin{definition}\label{def:propp} For an element $\fraks\in\bh(\mG)$ we 
say that {\bf $\fraks$ has property $\boldsymbol{P}$} if $P$ is a polynomial
in $\mbp(\mG)$ with root $\fraks$. Furthermore, we say that 
$\ev(\fraks)$ has property $P$, if $\fraks$ has property $P$.
\end{definition}
Observe that the filtering morphism $\mff_*$ extends to a morphism
\[
 \bh(\mG)\{X\}\lra\bh(\mG^w)\{X\},\,p\lmt p_\mff
\]
by defining $\mff_*(X)=X$. As such, this map restricts to a morphism
$
 \mbp(\mG)
 \lra
 \mbp(\mG^w)
$.
\begin{theorem}\label{thm:propp} Suppose we are given a 
symbol $\fraks\in\bh$. If $\fraks$ has property $P$, then 
$\mff_*(\fraks)\in\bh^w$ has property $P_\mff$.
\end{theorem}
\begin{proof}
This is an 
immediate application of the fact that
$\mff_*$ is a morphism of semialgebras. We have that
$P_\mff(\mff_*(\fraks))=\mff_*(P(\fraks))=\mff_*(0)=0$.
\end{proof}
A map
$
 f\co
 \bigotimes_{i}
 \cfhat_i
 \lra
 \cfhat'
$
between suitable Heegaard Floer chain modules which is defined by counting
elements of moduli spaces of Whitney polygons corresponds to a homology 
class $\fraks_f\in\bh$ via the morphism $\ev_*$, 
i.e.~$\ev_*(\fraks_f)=f$. A property of $f$ can be encoded into a polynomial
expression $P$ with coefficients in $\map$ such that $P(f)=0$ if and only if 
$f$ fulfills the property. Given such a polynomial $P$ is there a method to
relate $P$ to a polynomial $P'$ such that $P'(f^w)=0$? Morphisms between symbol
homologies give us a method to do that as Theorem~\ref{thm:propp} 
indicates (cf.~\S\ref{parone} and cf.~\S\ref{sec:eai}).

\section{Examples and Ideas}\label{sec:eai}
In this section we communicate some of the 
ideas behind the construction we gave in the previous sections. 
The ideas behind the theory are very simple and all the 
operations we provide are based upon simple algorithms. Although
the ideas are simple, writing these concepts down formally turns
out to be difficult and extensive. We think that reading this section 
will help the reader to familiarize with the techniques. In fact,
we introduce some notation for indicating the decorations of
a moduli space which makes the whole construction intuitive.
In the following sections we will use the notation introduced here. 
Furthermore, we present two explicit calculations
of symbol homologies in easy situations (see~Example~\ref{ex:calc1} 
and Example~\ref{ex:calc2}) and two examples which
should help indicating in what way the symbol homologies can be
of benefit (see~Example~\ref{example01} and Example~\ref{example02}): 
One of the benefits of this theory is that it unifies the Floer chain 
level and the moduli space level into one object (see~the discussion 
below, cf.~\S\ref{sec:rhftfsh} and \S\ref{sec:shafh}). A consequence
of this unification is that it provides a systematic and immediate way 
to transfer properties between different Floer theoretic settings 
without difficulty.
Proofs of properties which need the moduli space machinery now do not 
need to be repeated in different settings but can now just be accepted 
by pointing to our results. In Example~\ref{example02} at 
the end of this section we give an easy demonstration how this
transfer is done when explicitly worked out. This technique
will be applied in an invariance proof of cobordism maps
between knot Floer homologies (see~Theorem~\ref{cobmapinvar}) and 
for a surgery exact triangle in knot Floer 
homologies (see~Theorem~\ref{thm:set}).\\

Suppose we are given a set $B=\{\afat,\bfat,\gfat\}$ of attaching
circles in a surface $\Sigma$. In the previous sections we decorate
moduli spaces with auxiliary data. This is done by the following
rules which we exemplify on the moduli space 
$\mM^0_{(\afat,\gfat,\bfat)}$ of holomorphic Whitney triangles $\phi$ 
with Maslov index $0$, $n_z(\phi)=0$ and boundary conditions given by
$\afat$, $\gfat$ and $\bfat$. The boundary conditions specify conditions
on the edges of the triangle $\phi$. Additionally, we may impose conditions
on the vertices of the triangle, i.e.~specify points they have to be 
mapped to. In the literature this is indicated as follows
$\mM^0_{(\afat,\gfat,\bfat)}(\xfat,\yfat,\qfat)$ where 
$\xfat\in\talpha\cap\tbeta$, $\yfat\in\tbeta\cap\tgamma$ and 
$\qfat\in\talpha\cap\tgamma$. Instead of attaching just points to the
vertices there are now three different types of decorations. We either
attach $\overdown{\cdot}$, $\overup{\xfat}$ or $\overdot{\xfat}$ 
(see~Definition~\ref{def:generators}). The first
makes the corresponding vertex a flow-in vertex, the second makes it a
flow-out vertex and the third a pointed vertex 
(see~Definition~\ref{def:generators}). The 
point $\xfat$ is called the slot-point of the vertex. The additional 
information given by the decorations allow us to interpret moduli spaces 
in various ways: For instance, the space 
$
 \mA
 =
 \mM^0_{(\afat,\gfat,\bfat)}
 (\,\overdown{\cdot}\,,\overdot{\yfat},\overup{\qfat})
$
can be interpreted as a map
\[
 \ev(\mA)\co\cfhat(\Sigma,\afat,\bfat)\lra\cfhat(\Sigma,\afat,\gfat)
\]
such that 
$
 \ev(\mA)(\xfat)
 =
 \#(\mM^0_{(\afat,\gfat,\bfat)}(\xfat,\yfat,\qfat))\cdot\qfat
$ (see~Proposition~\ref{intro:ev}). Put in words, the flow-in 
vertices determine the source of 
the map, the flow-out vertices the destination and pointings 
serve as boundary conditions. In this way, all maps in Heegaard Floer 
theory (which are defined by counting elements of moduli spaces 
of Whitney polygons) can be represented as decorated moduli 
spaces if we find a suitable way to express sums of maps in terms of 
moduli spaces and if we find a suitable way to express 
compositions of maps in terms of moduli spaces.

To this end, suppose we are given an additional decorated space
$
 \mB=\mM^0_{(\afat,\gfat,\bfat)}
 (\,\overdown{\cdot}\,,\overdot{\rfat},\overup{\qfat})
$, the map $\ev(\mA)+\ev(\mB)$ sends an element $\xfat$ to
\[
 \begin{array}{rcl}
 (\ev(\mA)+\ev(\mB))(\xfat)
 &=&
 \ev(\mA)(\xfat)
 +
 \ev(\mB)(\xfat)\\
 &=&
 (\#\mM^0_{(\afat,\gfat,\bfat)}
 (\xfat,\yfat,\qfat)
 +
 \#\mM^0_{(\afat,\gfat,\bfat)}
 (\xfat,\rfat,\qfat))
 \cdot\qfat\\
 &=&
 \#(\mM^0_{(\afat,\gfat,\bfat)}
 (\xfat,\yfat,\qfat)\sqcup\mM^0_{(\afat,\gfat,\bfat)}
 (\xfat,\rfat,\qfat))\cdot\qfat
\end{array} 
\]
So, it makes sense to define $\mA\boxplus\mB$ as the
disjoint union of the spaces (see~\S\ref{sec:tshp}). However, 
there are some difficulties that arise as sources and destinations
that are specified by the flow-in vertices of decorated 
moduli spaces might not be matching. This produces some issues which
require some consideration.

In a similar vein, we proceed
to get a candidate for a product. Suppose we are given 
the decorated space 
$\mC=\mMhat^1_{(\afat,\bfat)}(\,\overdown{\cdot}\,,\overup{\xfat})$. 
The composition $\ev(\mC)\circ\ev(\mA)$ sends an element $\zfat$ to 
\[
 \ev(\mC)\circ\ev(\mA)(\zfat)=
 \#(\mMhat^1_{(\afat,\bfat)}(\zfat,\xfat))
 \cdot
 \#(\mM^0_{(\afat,\gfat,\bfat)}(\xfat,\yfat,\qfat)
 \cdot
 \qfat
.\] 
Since
\[
 \#(\mMhat^1_{(\afat,\bfat)}(\zfat,\xfat))
 \cdot
 \#(\mM^0_{(\afat,\gfat,\bfat)}(\xfat,\yfat,\qfat))
 =
 \#(
 \mMhat^1_{(\afat,\bfat)}(\zfat,\xfat)
 \times
 \mM^0_{(\afat,\gfat,\bfat)}(\xfat,\yfat,\qfat)),
\]
the Cartesian product is the right candidate for the product 
$\mC\boxtimes\mA$ (see~\S\ref{sec:wpash}). However, the 
product has to change decorations 
suitably (cf.~Definition~\ref{def:operators}). We define
\begin{eqnarray*}
 \mC\boxtimes\mA
 &=&
 \mMhat^1_{(\afat,\bfat)}(\,\overdown{\cdot}\,,\overup{\xfat})
 \boxtimes
 \mM^0_{(\afat,\gfat,\bfat)}(\,\overdown{\cdot}\,,\overdot{\yfat},\overup{\qfat})\\
 &=&
 \mMhat^1_{(\afat,\bfat)}(\,\overdown{\cdot}\,,\overdot{\xfat})
 \times
 \mM^0_{(\afat,\gfat,\bfat)}(\overdot{\xfat},\overdot{\yfat},\overup{\qfat}).
\end{eqnarray*}
Put in words, the algorithm goes as follows: Observe that the flow-out 
vertex of the decorated space $\mC$ appears as a flow-in vertex 
of $\mA$ (in the sense of Definition~\ref{def:operators} (d)). We transform 
the flow-out vertex of $\mC$ to a pointed 
vertex (by changing its decoration) while keeping the 
slot-point --~$\xfat$ in our case~-- unchanged. 
Furthermore,  we transform the corresponding flow-in vertex of 
$\mA$ to a pointed vertex and add $\xfat$ as its 
slot-point (see~\S\ref{sec:wpash}). In this way, the 
product uniquely specifies the source and 
destination in such a way that it coincides 
with the source and destination of $\ev(\mA)\circ\ev(\mC)$. Again, 
there are technical difficulties arising in this construction 
which require some consideration. 

To be able to see 
properties of the Floer homologies in our setting, we have to 
realize boundaries of moduli spaces as suitable $\boxplus$-sums 
of $\boxtimes$-products of elements in the symbol 
algebra (see~\S\ref{thediff}). We tried to indicate 
this in Figure~\ref{Fig:boundop}: Let us consider the 
decorated moduli space 
$
 \mM^1_{(\afat,\gfat,\bfat)}
 (\,\overdown{\cdot}\,,\overdot{\qfat},\overup{\yfat})
$.
We first compute the codimension-$1$ boundary of the space
$
\mM^1_{(\afat,\gfat,\bfat)}
(\,\cdot\,\qfat,\yfat)
$.
This is indicated in part (a) of Figure~\ref{Fig:boundop}.
\begin{figure}[t!]
\labellist\small\hair 2pt
\definecolor{mymagenta}{HTML}{FF00FF}
\pinlabel{{\color{mymagenta}$\xfat$}} [b] at 530 601
\pinlabel{{\color{mymagenta}$\xfat$}} [l] at 999 606
\pinlabel{{\color{mymagenta}$\xfat$}} [t] at 315 518
\pinlabel{{\color{mymagenta}$\xfat$}} [t] at 384 513
\pinlabel{$\yfat$} [t] at 483 513
\pinlabel{$\qfat$} [t] at 569 513
\pinlabel{{\color{mymagenta}$\xfat$}} [t] at 780 518
\pinlabel{{\color{mymagenta}$\xfat$}} [l] at 859 500
\pinlabel{$\yfat$} [l] at 943 500
\pinlabel{$\qfat$} [l] at 1043 500
\pinlabel{{\color{mymagenta}$\xfat$}} [t] at 1235 500
\pinlabel{{\color{mymagenta}$\xfat$}} [l] at 1323 500
\pinlabel{$\yfat$} [l] at 1414 500
\pinlabel{$\qfat$} [l] at 1511 500
\pinlabel{$\yfat$} [l] at 23 300
\pinlabel{$\qfat$} [l] at 121 300
\pinlabel{{\color{mymagenta}$\rfat$}} [t] at 269 319
\pinlabel{{\color{mymagenta}$\rfat$}} [t] at 342 314
\pinlabel{$\qfat$} [t] at 427 314
\pinlabel{$\yfat$} [t] at 526 314
\pinlabel{{\color{mymagenta}$\rfat$}} [b] at 530 401
\pinlabel{{\color{mymagenta}$\rfat$}} [t] at 739 319
\pinlabel{{\color{mymagenta}$\rfat$}} [l] at 815 300
\pinlabel{$\qfat$} [l] at 901 300
\pinlabel{$\yfat$} [l] at 998 300
\pinlabel{{\color{mymagenta}$\rfat$}} [b] at 1000 408
\pinlabel{{\color{mymagenta}$\rfat$}} [t] at 1192 303
\pinlabel{{\color{mymagenta}$\rfat$}} [l] at 1269 300
\pinlabel{$\qfat$} [l] at 1369 300
\pinlabel{$\yfat$} [l] at 1465 300
\pinlabel{{\tiny Compute Boundary}} [t] at 172 263
\pinlabel{{\tiny Add decorations}} [t] at 626 263
\pinlabel{{\color{mymagenta} $\tfat$}} [t] at 313 121
\pinlabel{$\qfat$} [b] at 384 203
\pinlabel{$\yfat$} [t] at 484 116
\pinlabel{{\color{mymagenta} $\tfat$}} [t] at 384 116
\pinlabel{{\color{mymagenta} $\tfat$}} [t] at 569 116
\pinlabel{part (a)} [t] at 432 70
\pinlabel{{\color{mymagenta} $\tfat$}} [t] at 781 121
\pinlabel{{\color{mymagenta} $\tfat$}} [l] at 860 100
\pinlabel{$\yfat$} [l] at 944 100
\pinlabel{{\color{mymagenta} $\tfat$}} [l] at 1042 100
\pinlabel{$\qfat$} [l] at  859 210
\pinlabel{part (b)} [t] at 909 70
\pinlabel{{\color{mymagenta} $\tfat$}} [t] at 1235 105
\pinlabel{{\color{mymagenta} $\tfat$}} [l] at 1326 100
\pinlabel{$\yfat$} [l] at 1411 100
\pinlabel{{\color{mymagenta} $\tfat$}} [l] at 1511 100
\pinlabel{$\qfat$} [l] at 1326 205
\pinlabel{part (c)} [t] at 1377 70
\endlabellist
\centering
\includegraphics[width=15cm]{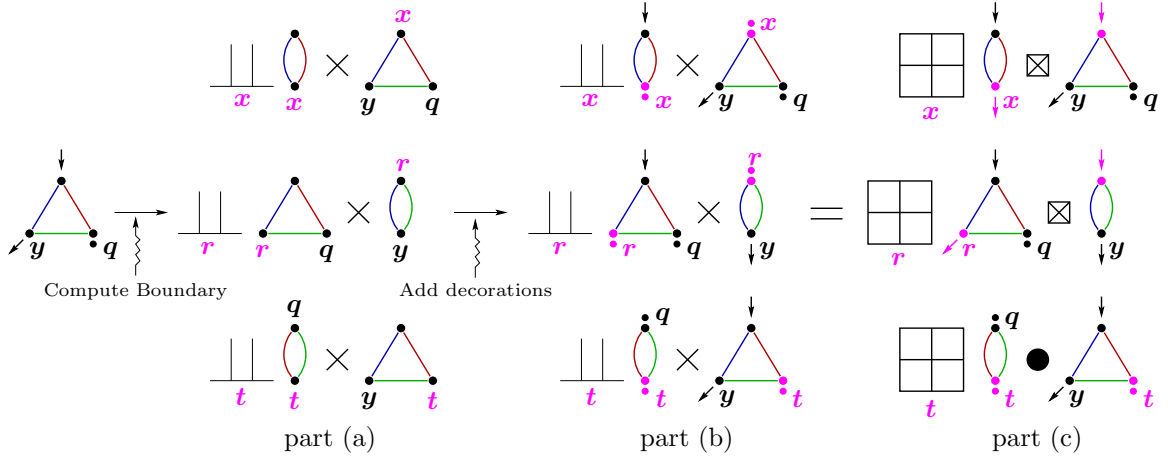}
\caption{How to define the boundary $\parsh$: We first compute the boundary
and then decorate the vertices with data. The black decorations are 
inherited from the old space and the pink decorations are attached 
due to the principle that {\it we need a unique and well-defined 
flowing direction}.}
\label{Fig:boundop}
\end{figure}
In part (b) of Figure~\ref{Fig:boundop} we decorate the spaces with data.
Observe that each moduli space in the boundary admits two types of vertices. The 
(old) vertices which coincide with vertices of the space 
$
\mM^1_{(\afat,\gfat,\bfat)}
(\,\cdot\,\qfat,\yfat)
$
and vertices which are new, i.e.~generated in the boundary. The old
vertices are indicated as black dots in Figure~\ref{Fig:boundop}
and the new vertices are colored pink. We decorate the black vertices with
the same data as the corresponding vertices in 
$
 \mM^1_{(\afat,\gfat,\bfat)}
 (\,\overdown{\cdot}\,,\overdot{\qfat},\overup{\yfat})
$
and the new vertices according to the principle of having
a {\it unique and well-defined flowing direction}. We point the reader to
\S\ref{thediff} for the description of the algorithm. In 
this process, the following phenomenon appears which complicates the 
construction slightly: The bottom of part (b) shows a Cartesian product
\[
 \mMhat^1_{(\afat,\bfat)}
 (\overdot{\xfat},\overdot{\tfat})
 \times
 \mM^0_{(\afat,\gfat,\bfat)}
 (\,\overdown{\cdot}\,,\overdot{\tfat},\overup{\zfat}) 
\]
The space $\mMhat^1_{(\afat,\bfat)}
 (\overdot{\xfat},\overdot{\tfat})$ cannot be interpreted as 
a map using the algorithm presented above for $\ev$. So, the 
given Cartesian product cannot be a $\boxtimes$-product since 
the latter should correspond to taking compositions of 
maps (cf.~Proposition~\ref{intro:ev}). So, we have to interpret 
$
 \mMhat^1_{(\afat,\bfat)}
 (\overdot{\xfat},\overdot{\tfat})
$ 
as a coefficient in our algebraic setting. For this reason, 
fully pointed pre-generators generate the coefficient 
algebra (see~\S\ref{sec:wpash}). This complication in 
the definition repairs this issue and is also supplemented by the 
following observation: The algorithm underlying the definition 
of $\ev$ can be applied not only for generators. For instance,
\[
 \begin{array}{rll}
\ev(\mMhat^1_{(\afat,\bfat)}
 (\overdot{\xfat},\overdot{\yfat}))&
 =\#\bigl(\mMhat^1_{(\afat,\bfat)}(\xfat,\yfat)\bigr)&
 \in\ztwo\\ 
 \ev(\mMhat^1_{(\afat,\bfat)}
 (\,\overdown{\cdot}\,,\overdot{\yfat}))&
 =\bigl(
 \xfat\lmt
 \#\mMhat^1_{(\afat,\bfat)}(\xfat,\yfat)\bigr)& 
 \in \mbox{\rm Hom}(\cfhat(\afat,\bfat),\ztwo)\\
 \ev(\mMhat^1_{(\afat,\bfat)}
 (\,\overdot{\xfat}\,,\overup{\yfat}))&
 =
 \#\bigl(\mMhat^1_{(\afat,\bfat)}(\xfat,\yfat)\bigr)
 \cdot\yfat
 &\in\cfhat(\afat,\bfat).
\end{array}
\]
Hence, fully pointed pre-generators naturally correspond to elements
in the coefficient ring of Heegaard Floer homology. 

Furthermore, this observation indicates that our setting provides a unified 
language for all elements of Floer homologies. As we 
will see in the next section, 
this will lead us to some kind of {\it reformulation} of Heegaard 
Floer homology in terms of our setting.

\begin{example}\label{ex:calc1} We would like to calculate the symbol homology
in a simple situation. Suppose we are given a Heegaard diagram
$\mH=(T^2,\afat,\bfat)$ where $\afat=\{\mu\}$ consists of a meridian 
$\mu$ and $\bfat=\{\lambda\}$ of a longitude $\lambda$ such that
$\#(\mu,\lambda)=1$. Denote by $\xfat$ the intersection point of $\mu$
and $\lambda$. Now set $B=\{\afat,\bfat\}$. Since $\mI_B=\{[(\afat,\bfat)]\}$
is a one-point set we drop it from the notation of all moduli spaces. 
The fully pointed pre-generator $\mM^0(\overdot{\xfat},\overdot{\xfat})$
is the only existing non-trivial fully pointed element. However, by definition
$\mM^0(\overdot{\xfat},\overdot{\xfat})=1$ inside $\mbf$. Hence, $\mbf\cong\ztwo$
with $\pco=0$ such that $\cohom\cong\ztwo$. 

There is only one non-trivial generator, namely 
$
 X=
 \mM^0(\,\overdown{\cdot}\,,\overup{\xfat}).
$
Now, observe that 
$
 X^{\boxtimes 2}
 =
 \mM^0(\,\overdown{\cdot}\,,\overdot{\xfat})
 \times
 \mM^0(\overdot{\xfat},\overup{\xfat})
$ and that
\[
 \begin{array}{lclcl}
 X^{\boxtimes 3}
 &=&
 \bigl(
 \mM^0(\,\overdown{\cdot}\,,\overdot{\xfat})
 \times
 \mM^0(\overdot{\xfat},\overup{\xfat})\bigr)
 \boxtimes
 \mM^0(\,\overdown{\cdot}\,,\overup{\xfat})
 &\hspace{-0.3cm}=&
 \hspace{-0.3cm}\mM^0(\,\overdown{\cdot}\,,\overdot{\xfat})
 \times
 \bigl(\mM^0(\overdot{\xfat},\overup{\xfat})
 \boxtimes
 \mM^0(\,\overdown{\cdot}\,,\overup{\xfat})\bigr)\\
 &=&
 \mM^0(\,\overdown{\cdot}\,,\overdot{\xfat})
 \times
 \bigl(\mM^0(\overdot{\xfat},\overdot{\xfat})
 \cp
 \mM^0(\overdot{\xfat},\overup{\xfat})\bigr)
 &\hspace{-0.3cm}=&
 \hspace{-0.3cm}\mM^0(\overdot{\xfat},\overdot{\xfat})
 \cp
 \bigr(\mM^0(\,\overdown{\cdot}\,,\overdot{\xfat})
 \times
 \mM^0(\overdot{\xfat},\overup{\xfat})\bigl)
 \\
 &=&
 X^{\boxtimes 2}.&&
 \end{array}
\]
This is the only existing relation and, hence, $\mS$ is isomorphic to 
the associative $\ztwo$-algebra $\mR\subset M_3(\ztwo)$ which is 
generated by the matrix
\[
 A
 =
 \left(
 \begin{matrix}
 1 & 1 & 0 \\
 0 & 0 & 1 \\
 0 & 0 & 0 
 \end{matrix}
 \right).
\]
The associated symbol homology can be written as 
$\bh=\mR\cup\{[\ohat]\}$ where $\bh\backslash[\ohat]$ 
is isomorphic to $\mR$ as a $\ztwo$-algebra.
\end{example}
\begin{example}\label{ex:calc2} Suppose we are given a Heegaard diagram 
$\mH=(T^2,\afat,\bfat)$ where both $\afat=\{\mu_1\}$ and 
$\bfat=\{\mu_2\}$ consist of a meridian such that $\mu_1$ 
and $\mu_2$ intersect in a canceling pair of intersection points 
$\xfat_1$, $\xfat_2$ where $\xfat_1$ denotes the one with higher 
relative grading.
The only non-trivial fully pointed pre-generator is 
$\mMhat^1(\overdot{\xfat_1},\overdot{\xfat_2})$. Thus, the 
coefficient algebra $\mbf$ is isomorphic to $\ztwo[X]$. Furthermore, 
there are only three non-trivial generators, namely
\[
 X^0_1=\mM^0(\,\overdown{\cdot}\,,\overup{\xfat_1}),
 \hspace{0.1cm}
 X^0_2=\mM^0(\,\overdown{\cdot}\,,\overup{\xfat_2})
 \hspace{0.15cm}
 \mbox{\rm and }
 X^1_2= \mMhat^1(\,\overdown{\cdot}\,,\overup{\xfat_2}).
\]
These fulfill the following relations
\[
 \begin{array}{lclclcl}
 (X^0_1)^{\boxtimes 3}&=&(X^0_1)^{\boxtimes 2} 
 &\hspace{0.5cm}& 
 X^0_1\boxtimes X^0_2&=&0\\
 (X^0_2)^{\boxtimes 3}&=&(X^0_2)^{\boxtimes 2} 
 & & 
 X^1_2\boxtimes (X^0_2)^{\boxtimes 2}&=&X^1_2\boxtimes X^0_2\\
 (X^1_2)^{\boxtimes 2}&=&0 
 &\mbox{\rm and }& 
 (X^0_1)^{\boxtimes 2}\boxtimes X^1_2&=&X^0_1\boxtimes X^1_2\\
 X^1_2\boxtimes X^0_1 &=&0
 &\hspace{0.5cm}& 
 X^0_2\boxtimes X^1_2&=&0\\
 X^0_2\boxtimes X^0_1&=&0
 &\hspace{0.5cm}&&&
 \end{array}
\]
So, $\mS$ is isomorphic to the associative 
$\ztwo[X]$-algebra $\mR\subset M_6(\ztwo)[X]$ which is
generated by the matrices
\[
 B=\left(\begin{matrix} A & 0 \\ 0 & 0 \end{matrix}\right),
 \hspace{0.2cm}
 C=\left(\begin{matrix} 0 & 0 \\ 0 & A \end{matrix}\right)
 \hspace{0.2cm}
 \mbox{\rm and }
 D,
\]
where $D_{ij}=1$ for $(i,j)$ equal to 
$(1,3)$, $(4,4)$, $(4,5)$, $(4,6)$, $(5,4)$, $(5,5)$, $(5,6)$
and $D_{ij}=0$ otherwise.
We see that $\bh=\mR\cup\{[\ohat]\}$ such 
that $\bh\backslash\{[\ohat]\}\cong\mR$ as $\ztwo[X]$-algebras.
\end{example}
To get an idea how the techniques from symbol homologies can be
applied, we discuss the following two examples. In Example~\ref{example01}
we intend to illustrate the relationship between maps in
Heegaard Floer homology and elements of the symbol 
algebra (see~\S\ref{sec:shafh}). In Example~\ref{example02}
we give an easy but explicitly worked out example for a
transfer of a property from Heegaard Floer homology to
knot Floer homology. We will apply the presented technique 
in \S\ref{sec:implications} and \S\ref{sec:knotcob} (cf.~\S\ref{parone}).
\begin{example}\label{example01} Given a Heegaard triple diagram 
$(\Sigma,\afat,\bfat,\gfat,z)$, denote by $\mH_1$ the Heegaard diagram
$(\Sigma,\afat,\bfat)$, denote by $\mH_2$ the Heegaard diagram 
$(\Sigma,\bfat,\gfat)$ and denote by $\mH_3$ the Heegaard 
diagram $(\Sigma,\afat,\gfat)$. We define
a map
\[ 
 \Fhat_{\afat,\bfat\gfat}
 \co
 \cfhat(\mH_1)
 \otimes
 \cfhat(\mH_2)
 \lra
 \cfhat(\mH_3)
 \]
in the following way: Let $B=\{\afat,\bfat,\gfat\}$ and denote 
by $B'$ the element $[(\afat,\gfat,\bfat)]$ of $\mI_B$. For
$\xfat\in\talpha\cap\tbeta$ and $\yfat\in\tbeta\cap\tgamma$, we 
define 
\[
 \Fhat_{\afat,\bfat\gfat}(\xfat\otimes\yfat)
 =
 \sum_{\qfat\in\talpha\cap\tgamma}
 \!\!\!\!
 \#\mM_{B'}^0(\xfat,\yfat,\qfat)\cdot\qfat
\]
and extend it to $\cfhat(\mH_1)\otimes\cfhat(\mH_2)$ as a bilinear
map.
We interpret the moduli space $\mM_{B'}^0(\,\cdot,\,\cdot,\qfat)$ as a
generator, by attaching data to its vertices. Instead of using the 
notations from the previous sections, we indicate the decorations
like introduced at the beginning of this section: we write
$
 \mM_{B'}^0
 (\,\overdown{\cdot},
 \,\overdown{\cdot},
 \overup{\qfat}
 )
$
for the generator $(\mM_{B'}^0)_{(\pointed,\fin,\fout)}$, where
$\pointed=\emptyset$, $\fin=\{(\bfat,\afat),(\gfat,\bfat)\}$ and
$\fout=\{\{(\afat,\gfat),\qfat\}\}$. Now consider 
the following element of the symbol algebra:
\[
 \fraks
 =
 \bigboxplus{\qfat\in\mbT_{\afat}\cap\mbT_{\gfat}}
 \mM_{B'}^0
 (\,\overdown{\cdot},
 \,\overdown{\cdot},
 \overset{\uparrow}{\qfat}).
\]
This symbol induces an element in the symbol homology, since all 
moduli spaces in the $\boxplus$-sum are $0$-dimensional and it is 
easy to see that
$
 \ev(\fraks)
 =\Fhat_{\afat,\bfat\gfat}.
$
Analogously, consider the following element
\[
 \fraks_{\afat\bfat}
 =
 \bigboxplus{\yfat\in\mbT_{\afat}\cap\mbT_{\bfat}}
 \mMhat_{(\afat,\bfat)}^1
 (\,\overdown{\cdot},
 \overup{\yfat})
\]
for which $\ev(\fraks_{\afat\bfat})=\parhat_{\mH_1}$ holds.
In a similar vein, we define elements $\fraks_{\bfat\gfat}$ 
and $\fraks_{\afat\gfat}$ with the properties that
$\ev(\fraks_{\bfat\gfat})=\parhat_{\mH_2}$ and 
$\ev(\parhat_{\afat\gfat})=\parhat_{\mH_3}$. All these elements
induce classes in $\bh$ we denote by 
$\fraks_{\afat\bfat;*}$, $\fraks_{\bfat\gfat;*}$, $\fraks_{\afat\gfat;*}$
and $\fraks_*$.
We would like to prove that
$
 \mathfrak{q}_*=[\mathfrak{q}]$ with
 $\mathfrak{q}
 =
 \fraks\boxtimes\fraks_{\afat\gfat}
 \boxplus
 \fraks_{\afat\bfat}\boxtimes\fraks
 \boxplus
 \fraks_{\bfat\gfat}\boxtimes\fraks
$
vanishes in $\bh$. In fact, 
\[
\parsh
 \Bigl(
 \bigboxplus{\yfat\in\talpha\cap\tgamma}\mM^1_{\afat\bfat\gfat}
   (\,\overdown{\cdot},\,\overdown{\cdot},\overup{\yfat},)\Bigr)
 =
 \bigboxplus{\yfat\in\talpha\cap\tgamma}
 \parsh
 \bigl(\mM^1_{\afat\bfat\gfat}
 (\,\overdown{\cdot},\,\overdown{\cdot},\overup{\yfat},)\bigr),\\
\]
where
$
\parsh
 \bigl(\mM^1_{\afat\bfat\gfat}
 (\,\overdown{\cdot},\,\overdown{\cdot},\overup{\yfat})\bigr)
$
equals

\[
\begin{array}{lclc}
 \bigboxplus{\xfat\in\talpha\cap\tbeta}
 \mMhat^1_{\afat\bfat}
 (\,\overdown{\cdot}\,,\overup{\xfat})
 \boxtimes
 \mM^0_{\afat\bfat\gfat}
 (\,\overdown{\cdot}\,,\,\overdown{\cdot}\,,\overup{\yfat})
 &\boxplus& 
 \bigboxplus{\qfat\in\tbeta\cap\tgamma}
 \mMhat^1_{\bfat\gfat}
 (\,\overdown{\cdot}\,,\overup{\qfat})
 \boxtimes
 \mM^0_{\afat\bfat\gfat}
 (\,\overdown{\cdot}\,,\,\overdown{\cdot}\,,\overup{\yfat})
 &\boxplus\\
 \bigboxplus{\rfat\in\talpha\cap\tgamma}
 \mM^0_{\afat\bfat\gfat}
 (\,\overdown{\cdot}\,,\overdown{\cdot},\overup{\rfat})
 \boxtimes
 \mMhat^1_{\afat\gfat}
 (\,\overdown{\cdot}\,,\overup{\yfat}).
&
\end{array}
\]
Hence, $\parsh
 \Bigl(
 \bigboxplus{\yfat\in\talpha\cap\tgamma}\mM^1_{\afat\bfat\gfat}
   (\,\overdown{\cdot},\,\overdown{\cdot},\overup{\yfat},)\Bigr)=\mathfrak{q}$
which implies $\mathfrak{q}_*=0$.
Thus, we have
\[
\begin{array}{ccccc}
 0=\ev_*(\mathfrak{q}_*)
 &=&
 \ev_*(\fraks_{\afat\gfat;*})
 \cu
 \ev_*(\fraks_*)
 &+&
 \ev_*(\fraks_*)
 \cu
 \ev_*(\fraks_{\afat\bfat;*})\\
 &&&+&
 \ev_*(\fraks_*)
 \cu
 \ev_*(\fraks_{\bfat\gfat;*})\\
&&&&\\
&=&
\parhat_{\mH_3}
\circ
\Fhat_{\afat,\bfat\gfat}
&+&
\Fhat_{\afat,\bfat\gfat}
\circ
\bigl(\parhat_{\mH_1}\otimes\mbox{\rm id}_{\bfat\gfat}\bigr)\\
&&&+&
\Fhat_{\afat,\bfat\gfat}
\circ
\bigl(\mbox{\rm id}_{\afat\bfat}\otimes\parhat_{\mH_2}\bigr).
\end{array}
\]
Consequently, $\Fhat_{\afat,\bfat\gfat}$ is a chain map. This 
illustrates that maps between Floer theories can be expressed 
as elements in the symbol homology and that properties of these maps 
are encoded in the image of $\parsh$. In this way, a property of a 
map is equivalent to the vanishing of a suitable obstruction class in
the symbol homology.
\end{example}

\begin{example}\label{example02} 
We point the reader to Example~\ref{example01} for the notations and 
definitions used here. 
We have seen in Example~\ref{example01} that $\Fhat_{\afat,\bfat\gfat}$ 
is a chain map from $\cfhat(\mH_1)\otimes\cfhat(\mH_2)$ to
$\cfhat(\mH_3)$. We denoted by $\fraks$ its canonical symbol 
(see~\S\ref{sec:shafh}).
We proved that $\Fhat_{\afat,\bfat\gfat}$ is a chain map by showing
that $\mathfrak{q}_*=[\mathfrak{q}]$ vanishes in homology. However, 
obverse that $\mathfrak{q}_*=P(\fraks_*)$ for 
\[
 P(X)
 =
 X\boxtimes\fraks_{\afat\gfat;*}
 \boxplus
 \fraks_{\afat\bfat;*}\boxtimes X
 \boxplus
 \fraks_{\bfat\gfat;*}\boxtimes X.
\]
Hence, the polynomial $P(X)\in\mbp(\mG)$ encodes the chain map 
property. Applying Theorem~\ref{thm:propp}, we see 
that $\mff_*(\fraks_*)$ fulfills the property $P_\mff$, i.e.~the 
equality
\begin{equation}
 0=P_\mff(\mff_*(\fraks_*))
 =
 \mff_*(\fraks_*)\boxtimes\mff_*(\fraks_{\afat\gfat;*})
 \boxplus
 \mff_*(\fraks_{\afat\bfat;*})\boxtimes\mff_*(\fraks_*)
 \boxplus
 \mff_*(\fraks_{\bfat\gfat;*})\boxtimes\mff_*(\fraks_*)
 \label{eq:maineq}
\end{equation}
holds. Now observe that 
$\mff_*(\fraks_{\afat\gfat;*})$, $\mff_*(\fraks_{\afat\bfat;*})$ 
and $\mff_*(\fraks_{\bfat\gfat;*})$ are the canonical symbols 
of $\partial^{\bullet,\bullet}_{\mH_3}$, 
$\parhat^{\bullet,\bullet}_{\mH_1}$ and
 $\parhat^{\bullet,\bullet}_{\mH_2}$, respectively 
(cf.~\S\ref{knotfloerhomology}).
By Proposition~\ref{intro:ev}, equation 
\eqref{eq:maineq} implies
\[
 \begin{array}{lclcl}
 0 &=&
 \!\!\!\!
 \ev_*(\mff_*(\fraks_*))
 \boxtimes
 \ev_*(\mff_*(\fraks_{\afat\gfat;*}))
 &\boxplus&
 \ev_*(\mff_*(\fraks_{\afat\bfat;*}))
 \boxtimes
 \ev_*(\mff_*(\fraks_*))\\
 &&&\boxplus&
 \ev_*(\mff_*(\fraks_{\bfat\gfat;*}))
 \boxtimes
 \ev_*(\mff_*(\fraks_*))
 \\
 &=&
 \!\!\!\!
 \parhat^{\bullet,\bullet}_{\mH_3}
 \circ\ev_*(\mff_*(\fraks_*))
 &+&
 \ev_*(\mff_*(\fraks_*))
 \circ
 (\parhat^{\bullet,\bullet}_{\mH_1}\otimes\id_{\bfat\gfat})\\
 &&&+&
 \ev_*(\mff_*(\fraks_*))
 \circ
 (\id_{\afat\bfat}\otimes\parhat^{\bullet,\bullet}_{\mH_2}).
\end{array}
\]
Hence, 
\[
 \ev_*(\mff_*(\fraks_*))
 \co
 \cfkhat(\mH_1)\otimes\cfkhat(\mH_2)
 \lra
 \cfkhat(\mH_3)
\]
is a chain map. But, observe that it equals 
$F^{\bullet,\bullet}_{\afat,\bfat\gfat}$ from \S\ref{sec:knotcob}.
As we see, the chain map property of 
$F^{\bullet,\bullet}_{\afat,\bfat\gfat}$ is a consequence
of the chain map property of $\Fhat_{\afat,\bfat\gfat}$ and 
Theorem~\ref{thm:propp}.
\end{example}
The technique presented in Example~\ref{example02} is 
used in \S\ref{sec:knotcob} to give an invariance proof for 
cobordism maps in knot Floer homology and for a surgery exact
triangle in \S\ref{sec:implications} (see~Theorem~\ref{thm:set}).

\section{Recovering Heegaard Floer Theory from Symbol 
Homology}\label{sec:rhftfsh}
\subsection{Homology}\label{sec:homology} Consider 
in $(\mT,+,\times)$ the $\mbf$-subalgebra $C$ generated 
by pre-generators for which $\#(\fin)=0$ and $\#(\fout)=1$ holds.
The subsemialgebra $\mpq^1=\Phi(C)$ of $(\mThat,\boxplus,\boxtimes)$ 
is naturally equipped with a differential
$
 \parsh
 \co
 \mpq^1
 \lra
 \mpq^1
$
by applying the algorithm presented in \S\ref{thediff}. Denote
by $\hpq=H_*(\mpq^1,\parsh)$ the associated homology theory. Given 
an element  $[(\afat,\bfat)]\in\mI_B$, we consider
\[
 \mpq^0_{(\afat,\bfat)}
 =
 \cohom\bigl<[\mA]
 =
 [A_{(\pointed,\fin,\fout)}]\in\hpq
 \,|\,
 A=\mM^0_{(\afat,\bfat)},\#(\fin)=0,\#(\fout)=1\bigr>
\]
which is a $\cohom$-submodule of $\mpq_*$ and consider
\[
  \oab
  =
  \bigboxplus{\yfat\in\talpha\cap\tbeta}
  {\bigl[\mM_{(\afat,\bfat)}^0
  (\,\overdown{\cdot},\overup{\yfat})\bigr]}
\]
which is an element in the symbol homology $\bh$. Denote 
by $\mH$ the Heegaard diagram $(\Sigma,\afat,\bfat)$ and denote
by $\fraks_{\afat\bfat}$ the canonical 
symbol (see~\S\ref{sec:shafh}) 
of $\parhat_{\mH}$. Then we obtain the following 
interpretation of Heegaard Floer homology.
\begin{theorem}\label{thm:alterhf} Multiplication from 
the right with the element 
$\fraks_{\afat\bfat}\boxtimes\oab$ defines a 
differential $\partial_\mpq$ 
on $\mpq^0_{(\afat,\bfat)}$. Denote 
by $\bigl(\mpq_{(\afat,\bfat)}^0\bigr)_*$ 
the induced homology theory. 
\begin{enumerate}
\item[(i)] We have that
\begin{equation}
\mpq_{(\afat,\bfat)}^0\otimes_{\cohom}\ztwo
\cong
\cfhat(\mH),
\end{equation}
where the left side is equipped with the 
differential $\partial_\mpq\otimes\id$ and 
where we equip $\ztwo$ with the structure of 
a $\cohom$-module using the 
map $(\evf)_*$ (see~\eqref{eq:ctdef}). Furthermore,
\begin{equation}
  \bigl(\mpq_{(\afat,\bfat)}^0\bigr)_*\otimes_{\cohom}\ztwo
  =
  H_*(\mpq_{(\afat,\bfat)}^0\otimes_{\cohom}\ztwo,\partial_\mpq\otimes\id)  
  \cong
  \hfhat(\mH),\label{eq:sfhf}
\end{equation}
where $\ztwo$ carries the structure of a $\cohom$-module.
\item[(ii)] 
Suppose we are given a map 
\[
 F
 \co
 \bigotimes_{i=1}^{n-1}\cfhat(\afat_i,\afat_{i+1})
 \lra
 \cfhat(\afat_n,\afat_1)
\]
with canonical symbol $\fraks_F$. Multiplication from the right
with the symbol $\fraks_F\boxtimes\os_{(\afat_n,\afat_1)}$ defines 
a map
\[
 \,\cdot\,\boxtimes\fraks_F\boxtimes\os_{(\afat_n,\afat_1)}\co
 \Bigl(\bigotimes_{i=1}^{n-1}\mpq_{(\afat_i,\afat_{i+1})}^0\Bigr)_*
 \lra
 \bigl(\mpq_{(\afat_n,\afat_1)}^0\bigr)_*
\]
such that, under the isomorphism given in part $(i)$, this map 
corresponds to $F$ (even on the chain level).
\end{enumerate}
\end{theorem}
\begin{proof} Denote by $\mH$ the Heegaard diagram $(\Sigma,\afat,\bfat)$. 
For $\xfat\in\talpha\cap\tbeta$ denote by $X_\xfat$ the element 
$\bigl[\mM_{(\afat,\bfat)}^0
(\,\overdot{\cdot},\overset{\uparrow}{\xfat})\bigr]\otimes 1$. 
Define $\rho(X_\xfat)=\xfat$ and extend to a map
\[
 \rho\co
 \mpq_{(\afat,\bfat)}^0\otimes_{\cohom}\ztwo
 \lra
 \cfhat(\mH)
\]
as a morphism of modules, where
\begin{equation}
  \rho(f\cp\mA\otimes 1)
  =
  \evf(f)\cdot\rho(\mA\otimes 1)
  \label{eq:rholinearity}
\end{equation}
for $\mA\in\mpq_{(\afat,\bfat)}^0$ and $f\in\cohom$. It is easy 
to see that $\rho$ is a bijection. Furthermore, consider the 
following chain of equalities:
\begin{eqnarray*}
(\rho\circ\partial_\mpq\otimes\id)(X_\xfat)
&=&
\rho\Bigl(\bigboxplus{\yfat\in\talpha\cap\tbeta}
{\bigl[\widehat{\mM}^1_{(\afat,\bfat)}(\overdot{\xfat},\overdot{\yfat})\bigr]
\cp
X_\yfat}\Bigr)\\
&=&
\sum_{\yfat\in\talpha\cap\tbeta}
\#\widehat{\mM}^1_{(\afat,\bfat)}(\xfat,\yfat)\cdot\yfat\\
&=&
\parhat_{\mH}(\xfat)\\
&=&
\parhat_{\mH}
\Bigl(\rho(X)\Bigr).
\end{eqnarray*}
Since the map $\rho$ is a bijection, the given computation shows 
that $\rho$ induces a map 
\[
\rho_*\co
H_*(\mpq^0_{(\afat,\bfat)}\otimes_{\cohom}\ztwo,\partial_\mpq\otimes\id)
\lra
\hfhat(\mH)
\] 
which is an isomorphism. This proves that the right equality 
in \eqref{eq:sfhf} is true. The left equality in \eqref{eq:sfhf} is
clear.\\
Given a map $F\co\cfhat(\afat,\bfat)\lra\cfhat(\afat',\bfat')$ which 
is defined by counting elements of moduli spaces of Whitney polygons,
there are generators $\mR_{i,j}$, $j=1,\dots,l$ and $i=1,\dots,k_j$
such that
\[
 \fraks_F
 =
 \bbigboxplus{j=1}{l}\,
 \bbigboxtimes{i=1}{k_j}
 \,[\mR_{i,j}]
\]
is its canonical symbol. The data of $\mR_{i,j}$ are denoted by 
$\fin(i,j)$, $\fout(i,j)$ and $\pointed(i,j)$. Since $\fraks_F$ is
the canonical symbol of $F$, for every $i$ and $j$ the flow-in vertices
of $\mR_{i,j}$ appear as flow-out vertices of $\mR_{i-1,j}$. We denote
by $\fin$, $\fout$ and $\pointed$ the decorations of $\mM^0_{(\afat,\bfat)}(\overdot{\xfat},\overup{\xfat})$ and 
we denote by $I_\yfat$ the set of $j$ for which 
$\{(\afat',\bfat'),\yfat\}\in\pointed(\mR_{k_j,j})$. Hence, 
we have that
\[
[\mM^0_{(\afat,\bfat)}(\overdot{\xfat},\overup{\xfat})]
 \boxtimes
 \fraks_F
 \boxtimes
 \os
 =
\bigboxplus{\yfat\in\mbT_{\afat'}\cap\mbT_{\bfat'}}
\!\!\!\!\!\!
Q_\yfat\cp
\bigl[\mM^0_{(\afat',\bfat')}(\overdot{\yfat},\overup{\yfat})\bigr]
\]
with
\[
\begin{array}{rcl}
Q_\yfat
&=&
\bigboxplus{j\in I_\yfat}
 \Bigl(\bigl[\mM^0_{(\afat,\bfat)}(\overdot{\xfat},\overdot{\xfat})
 \cp
 \mKinq(\mKoutq(\mR_{1,j})) 
 \cp
 \bigl(
 \bigcp{i=2}{k_j}\mKin{\fout(i-1,j)}(\mKoutq(\mR_{i,j}))
 \bigr)\Bigr)\\
 &=&
 \bigboxplus{j\in I_\yfat}
 \Bigl(
 \bigl[\mKinq(\mKoutq(\mR_{1,j})) 
 \cp
 \bigl(
 \bigcp{i=2}{k_j}\mKin{\fout(i-1,j)}(\mKoutq(\mR_{i,j}))
 \bigr)\bigr]
 \Bigr),
\end{array}
\]
where the second equality holds since in the coefficient algebra  
$\mM^0_{(\afat,\bfat)}(\overdot{\xfat},\overdot{\xfat})=1$. But it
is not hard to see that $\evf(Q_\yfat)=\left.F(\xfat)\right|_\yfat$.
So, under the morphism $\rho$ the multiplication from the right with
$\fraks_F\boxtimes\os$ corresponds to the map $F$.
\end{proof}
For every element $\fraks\in\bh$, there
exists a unique pair of attaching circles $\afat$, $\bfat$ such that
$\fraks\boxtimes\oab\not=\ohat$. Hence, in products we can suppress 
the attaching circles from the notation and just write $\fraks\boxtimes\os$ 
instead. 
\begin{definition}
For a symbol $\fraks\in\bh$, define $\fraks_\os$ to be the 
product $\fraks\boxtimes\os$. Furthermore, for an 
element $P\in\mbp$ we define $P_\os$ to be the polynomial expression 
we obtain by replacing all coefficients $c_i$ by $c_i\boxtimes\os$.
\end{definition}
\begin{prop} A symbol $\fraks_*\in\bh$ fulfills a property $P$ if and only if
$\fraks_\os$ fulfills the property $P_\os$.
\end{prop}
\begin{proof} For symbols $\fraks_*=[\fraks]$, $\frakt_*=[\frakt]\in\bh$ 
a simple calculation shows that
$\fraks\boxtimes\os\boxtimes\frakt=\fraks\boxtimes\frakt$.  
Consequently, we have that
\[
  P_{\os}(\fraks_\os)=P_{\os}(\fraks_*\boxtimes\os)=P(\fraks_*)\boxtimes\os.
\]
Now suppose that $P(\fraks_*)$ vanishes, then $P_\os(\fraks_*\boxtimes\os)=0$. Conversely, given that $P_\os((\fraks_*)_\os)$ vanishes, we have that $P(\fraks_*)\boxtimes\os=0$. So, there is an element $\mathfrak{q}$ such that
$P(\fraks)\boxtimes\os=\parsh(\mathfrak{q})$. Since $\os$ is not 
an element in the image of $\parsh$, the 
element $\mathfrak{q}$ is
of the form $\frakt\boxtimes\os$ which 
implies that $P(\fraks)=\parsh(\frakt)$ and, hence, $P(\fraks_*)=0$. 
\end{proof}
\begin{proof}[Proof of Theorem~\ref{main}]
Recall that by Theorem~\ref{thm:alterhf} we may regard the modules 
$\mpq^0_{(\afat,\bfat)}$ and their respective homology theories 
$(\mpq_{(\afat,\bfat)}^0)_*$ as being equivalent to Heegaard Floer 
homology. Furthermore, symbols of type $\fraks_\os$ can be regarded as 
maps between Floer homologies. Applying Theorem~\ref{thm:propp} and 
the fact that {\it filtering} induces a morphism $\mff_*$, we see that 
if a map $\fraks_\os$ fulfills a property $P_\os$, then
$\mff_*(\fraks_\os)=\mff_*(\fraks)_\os$ fulfills the property 
$(P_\mff)_\os$. 
\end{proof}

\subsection{Cohomology} In a similar vein it is possible to recover Heegaard
Floer cohomology in terms of the symbol homology theory. To do that, we
have to give a couple of definitions: Consider in $(\mT,+,\times)$ 
the $\mbf$-subalgebra $C$ generated by pre-generators with the property that
$\#(\fin)>0$ and $\#(\fout)\leq1$. The semialgebra $\mpq^{coh}=\Phi(C)$ of
$(\mThat,\boxplus,\boxtimes)$ is naturally equipped with a differential
$
  \parsh
  \co
  \mpq^{coh}
  \lra
  \mpq^{coh}
$
by applying the algorithm presented in \S\ref{thediff}. Denote by $\mpq^{coh}_*$
the homology theory associated to the complex $(\mpq^{coh},\parsh)$. Given
an element $[(\afat,\bfat)]\in\mI_B$, consider
\[
 \mpq^{coh}_{(\afat,\bfat)}
 =
 \cohom\bigl<[\mA]
 =
 [A_{(\pointed,\fin,\fout)}]\in\mpq^{coh}_*
 \,|\,
 A=\mM^0_{(\afat,\bfat)},
 \#(\fin)=1,\#(\fout)=0
 \bigr>.
\]
Define $\mH$, $\fraks_{\afat\bfat}$ and $\oab$ as in \S\ref{sec:homology}.
\begin{theorem}\label{thm:altercohf} Multiplication from the left with 
the element $\oab\boxtimes\fraks_{\afat\bfat}$ defines a 
differential $\partial_\mpq^{coh}$ on $\mpq^{coh}_{(\afat,\bfat)}$. Denote 
by $(\mpq_{(\afat,\bfat)})^{coh}_*$ the induced homology theory. 
\begin{enumerate}
 \item[(i)] We have that
\begin{equation}
\mpq_{(\afat,\bfat)}^{coh}\otimes_{\cohom}\ztwo
\cong
\cfhat\,\!^*(\mH)
\end{equation}
where the left side is equipped with the differential 
$\partial_\mpq^{coh}\otimes\id$ and where we equip $\ztwo$ with the 
structure of an $\cohom$-module using the map $(\evf)_*$ (see~\eqref{eq:ctdef}). 
Furthermore,
\begin{equation}
  \bigl(\mpq_{(\afat,\bfat)}^{coh}\bigr)_*\otimes_{\cohom}\ztwo
  =
  H_*(\mpq_{(\afat,\bfat)}^{coh}\otimes_{\cohom}\ztwo,
  \partial_\mpq^{coh}\otimes\id)  
  \cong
  \hfhat\,\!^*(\mH).\label{eq:sfchf}
\end{equation}
where $\ztwo$ carries the structure of a $\cohom$-module.
\item[(ii)] Given a map between two Heegaard Floer chain complexes
$\cfhat(\bfat,\afat)$ and $\cfhat(\bfat',\afat')$, denote by
$\fraks_F$ its associated canonical symbol. Then, multiplication from
the left with the element $\oab\boxtimes\fraks_F$ induces a map
\[
 \oab\boxtimes\fraks_F\boxtimes\,\cdot\,\co
 \bigl(\mpq_{(\afat,\bfat)}^{coh}\bigr)_*
 \lra
 \bigl(\mpq_{(\afat',\bfat')}^{coh}\bigr)_*
\]
such that, under the isomorphism given in part $(i)$, this map
corresponds to $F$ (even on the chain level).
\end{enumerate}
\end{theorem}
\begin{proof} The proof goes the same way as the proof of 
Theorem~\ref{thm:alterhf}.
\end{proof}

\section{U-equivariant Symbol Homology}\label{sec:uesh}
The construction of symbol homology and its ambient algebra can 
be altered in various ways without destroying the 
properties derived in \S\ref{sec:tshp}. Here, we present a variant of 
this theory by entering a $U$-variable. This modification is necessary 
to capture information from flavors of Heegaard Floer homology which also 
admit a $U$-variable in their definition. For a moduli 
space $\mM^\mu_{B'}$, we define 
\[
  \mM^{\mu;i}_{B'}
  =\{\phi\in\mM^\mu_{B'}\,|\,n_w(\phi)=i\}.
\]  
In the following, we will decorate the $\mM^{\mu;i}_{B'}$ with data and 
follow the construction process as outlined in \S\ref{sec:tshp} with 
some slight adaptions: Instead of $\mbf$ as coefficients we use 
$\mbf[U]$ as coefficients to generate the algebra $\mThat$. Furthermore, we
define the sum $\boxplus$ and the product $\boxtimes$ as in 
\S\ref{sec:tshp}, with the additional condition
\[
  (U\cp\mA)\boxplus(U\cp\mB)=U\cp(\mA\boxplus\mB)
\]
and
\[
  (U\cp\mA)\boxtimes\mB=U\cp(\mA\boxtimes\mB)=\mA\boxtimes(U\cp\mB)
\]
imposed. The symbol algebra we obtain with this new construction will be 
denoted by $\mShat_U$ and the associated symbol homology by $\bhu$.
Every moduli space $\mM^\mu_{B'}$ is a disjoint union of 
the $\mM^{\mu;i}_{B'}$, for $i\geq0$. Inspired by 
this, it is possible to define a map from the symbol 
algebra $\mShat$ to $\mShat_U$ in the following way: Given a generator 
$\mA=(\mM^\mu_{B'})_{(\pointed,\fin,\fout)}$, 
we denote by $\mA^i$ the element $(\mM^{\mu;i}_{B'})_{(\pointed,\fin,\fout)}$. 
We require that $\mff_U(\ohat)=\ohat$ and that for
a generator $\mA$ we have
\[
 \mff_U(\mA)
 =
 \bigboxplus{i}{}{\,U^i\cp\mA^i}.
\]
Observe that the sum is finite. This assignment extends to a morphism
\[
 \mff_U
 \co
 \mShat
 \lra
 \mShat_U
\]
of semialgebras. 
\begin{prop} If $\mJ_s$ is $w$-respectful, then the map $\mff_U$ is a 
chain map with respect to the differential $\parsh$ and, thus, descends 
to a $(\cohom,\mathfrak{f}_{U,*})$-morphism
\[
 \mff_{U;*}
 \co
 \bh
 \lra
 \bhu
\]
 of semialgebras.
\end{prop}
\begin{proof} The statement that $\mff_U$ is a chain map follows from the 
fact that intersection numbers are homotopical invariants and behave 
additive under splicing. The fact, that the map $\mff_{U;*}$ is a morphism 
follows from the fact that $\mff_U$ is a morphism which in turn is true by 
its definition.
\end{proof}
\begin{theorem}\label{main2} If a map $\fraks$ from a tensor product of 
Heegaard Floer chain complexes to another Heegaard Floer chain complex 
fulfills a property $P$, then the map $\mff_{U;*}(\fraks)$ between 
the corresponding $\cfkminus$-knot Floer chain complexes fulfills the 
property $P_{\mff_U}$.
\end{theorem}
\begin{proof}
The proof of this theorem 
follows the same lines as the proof of Theorem~\ref{main}.
\end{proof}
Although we did not write this down explicitly, there is 
a $U$-equivariant version of Theorem~\ref{thm:alterhf} (and 
Theorem~\ref{thm:altercohf}) giving a model for $\hfkminus$ in terms of 
the symbol homology theory.

\section{Perturbed Symbol Homologies}\label{sec:psh}
In this section we briefly sketch a necessary extension of the 
symbol homology theories. Observe that the isomorphisms between Floer
theories that are induced by --~for instance~-- isotopies or perturbations 
of the path of almost complex structures use dynamic boundary conditions.
Hence, these maps cannot be presented as elements in the symbol homology
we defined in \S\ref{sec:tshp}. However, we can extend the symbol homology
theory by including moduli spaces with dynamic boundary conditions to the
set of generators. The constructions
given in \S\ref{sec:tshp} and \S\ref{sec:fsh} then carry over 
verbatim. Because of the similarity of 
these approaches, we will just specify which moduli spaces we have 
to include into our considerations.\vspace{0.3cm}\\
\subsection{Perturbations of the Almost Complex Structure}\label{sec:potacs}
Given a perturbation $\mJ_{s,t}$ of the path $\mJ_{s,0}$, we additionally
fix a homotopy $\mJ_{s,t}(\tau)$ where $\mJ_{s,t}(0)=\mJ_{s,t}*\mJ_{s,1-t}$ 
and $\mJ_{s,t}(1)=\mJ_{s,0}$. For a given set of attaching circles $B$ we
consider the following moduli spaces.
\begin{enumerate}
\item[(1)] For every $B'\in\mI_B$ we consider $\mM_{B';\mJ_{s,0}}^\mu$ with
restrictions to the Maslov-index as given in part (1) of 
Definition~\ref{def:generators}.
\item[(2)] For every $B'\in\mI_B$, we consider $\mM_{B';\mJ_{s,1}}^\mu$ with
restrictions to the Maslov-index as given in part (1) of 
Definition~\ref{def:generators}.
\item[(3)] For every element $B'\in\mI_B$, we consider 
$\mM_{B';\mJ_{s,t}}^{i}$, for $i=0,1$, which is the set of 
$\mJ_{s,t}$-holomorphic Whitney polygons with boundary 
conditions specified by $B'$ with Maslov index $i$ 
(see~\cite{OsZa01} or cf.~\cite{Saha02}). Furthermore, we consider
$\mM_{B';\mJ_{s,1-t}}^{i}$, $i=-1,0$, and 
\[
 \mM_{B';\tau}^{i}
 =
 \bigcup_{\tau\in[0,1]}
 \mM_{B';\mJ_{s,t}(\tau)}^{i}
\]
for $i=-1,0$. 
\end{enumerate}
The constructions given in \S\ref{sec:tshp} carry over verbatim 
to provide a symbol homology theory with the following slight 
adaptions: We need to define $\parsh$ for the generators we obtain 
from (3). The algorithm presented in \S\ref{thediff} applies here, as 
well, to provide a definition of $\parsh$ for all generators given by 
moduli spaces that come from (1), (2) and the spaces 
$\mM_{B';\mJ_{s,t}}^{i}$. For $\mM_{B';\tau}^{0}$ the following 
algorithm applies: There are two types of ends, the broken ends and 
the ends coming from $\tau\to0/1$ which are $\mM_{B';\mJ_{s,t}(0)}^{0}$ 
and $\mM_{B';\mJ_{s,t}(1)}^{0}$. To the first type, i.e.~the broken 
ends, the algorithm from \S\ref{thediff} applies. To the second type 
we apply the following procedure: There is a canonical one-to-one 
correspondence between the vertices of $\mM_{B';\tau}^{0}$ and the 
vertices of $\mM_{B';\mJ_{s,t}(0)}^{0}$ (or $\mM_{B';\mJ_{s,t}(1)}^{0}$). 
Hence, we may take the decorations of the vertices of 
$\mM_{B';\tau}^{0}$ and attach them to the vertices of 
$\mM_{B';\mJ_{s,t}(0)}^{0}$ (and $\mM_{B';\mJ_{s,t}(1)}^{0}$). 
Analogously, we get a definition of $\partial_\mbf$ for fully pointed 
pre-generators coming from moduli spaces of (3). 
\begin{example}
For $\yfat\in\talpha\cap\tbeta$ we have
\begin{eqnarray*}
 \parsh
 \Bigl(
 \mM_{(\afat,\bfat);\tau}^{0}
(\,\overset{\downarrow}{\cdot},\overset{\uparrow}{\yfat})
 \Bigr)
 &=&
 \mM_{(\afat,\bfat);\mJ_{s,t}(0)}^0
 (\,\overset{\downarrow}{\cdot},\overset{\uparrow}{\yfat})
 \boxplus
 \mM_{(\afat,\bfat);\mJ_{s,t}(1)}^0
 (\,\overset{\downarrow}{\cdot},\overset{\uparrow}{\yfat})
 \\
 &\boxplus&
 \bigboxplus{\qfat\in\talpha\cap\tbeta}
 {\mM_{(\afat,\bfat);\mJ_{s,t}(\tau)}^{0}
 (\,\overset{\downarrow}{\cdot},\overset{\uparrow}{\qfat})
 \boxtimes
 \widehat{\mM}_{(\afat,\bfat);\mJ_{s,t}(1)}^1
 (\,\overset{\downarrow}{\cdot},\overset{\uparrow}{\yfat})}
 \\
 &\boxplus&
 \bigboxplus{\qfat\in\talpha\cap\tbeta}
 {
 \widehat{\mM}_{(\afat,\bfat);\mJ_{s,t}(0)}^1
 (\,\overset{\downarrow}{\cdot},\overset{\uparrow}{\qfat})
 \boxtimes
 \mM_{(\afat,\bfat);\mJ_{s,t}(\tau)}^0
 (\,\overset{\downarrow}{\cdot},\overset{\uparrow}{\yfat})} 
\end{eqnarray*}
For the notation of the decorations we point the reader to \S\ref{sec:eai}.
The boundaries in lines $2$ and $3$ are given by the algorithm 
presented in \S\ref{thediff}. The first line is defined by the 
procedure presented above: we move the data from 
$\mM_{(\afat,\bfat);\tau}^{0}
(\,\overset{\downarrow}{\cdot},\overset{\uparrow}{\yfat})$ to 
its boundary components.
\end{example}

\subsection{Isotopies}\label{sec:i}
This is done precisely the same way as the case of perturbations of 
the almost complex structure. For a given Hamiltonian isotopy $\Phi$ 
we include the moduli spaces with dynamic boundary conditions into 
the theory.

\section{Implications I -- Knot Cobordisms}\label{sec:knotcob}
In this section we will give the construction of cobordism 
maps in knot Floer homologies for $\hfkhat$, $\hfkminus$ and $\hfkinfty$ 
for simplicity, i.e.~for theories in which the intersection 
condition $n_z=0$ is imposed and $n_w$ is arbitrary. In 
\S\ref{sec:okft} we sketch the adaptions that have to be made 
for theories for which $n_z$ is arbitrary. In the following, we will 
work with diagrams which are weakly admissible, where weak 
admissibility is defined with respect to the 
point $z$ (see~\cite[Definition~4.10]{OsZa01}). We will 
write $\hfkoc$ to indicate, that we consider 
$\hfkhat=\mbox{\rm HFK}^{\bullet,\bullet}$ or 
$\hfkminus$ (or even $\hfkinfty$).\vspace{0.3cm}\\
We expect the reader to be familiar with the
work \cite{OsZa03} of Ozsv\'{a}th and Szab\'{o}. The idea to associate to 
a cobordism $W$ between two $3$-manifolds $Y$ and $Y'$ a map between 
the associated Floer homologies is similar to the constructions in 
\cite{OsZa03} and goes as follows: First observe that the cobordism
$W$ admits a handle decomposition relative to the boundary component 
$Y$ with no $0$-handles and no $4$-handles (cf.~\cite{GoSt}). To use 
the notation of 
\cite{GoSt}, the boundary $\partial W$ can be written as 
$\partial_+W\sqcup\overline{\partial_-W}$, where one of these components 
might be empty. In our case, $\partial_-W=Y$ and $\partial_+W=Y'$. Since both 
components are non-empty, we do not require $0$-handles and $4$-handles 
(see~\cite[Proposition~4.2.13]{GoSt}). Furthermore, we may think the handles to 
be attached in order of increasing index and the handles of the same index to be 
attached simultaneously (see~\cite[Proposition~4.2.7]{GoSt}). 
Thus, we may split up $W$ as
\[
 W=W_1
 \cupb
 W_2
 \cupb
 W_3
\]
where $W_i$ is built by the handles of index $i$. To associate a 
map $\Foc_{W}$ to $W$ we will choose a splitting of $W$ into 
handles and then associate to $W_i$ a map between the Floer 
homologies of the boundary components. Then, $\Foc_{W}$ will 
be the composition of these three maps. The maps defined, here, 
will be defined similarly as Oszv\'{a}th and Szab\'{o} do it in 
their paper.
\begin{definition} Let $Y$, $Y'$ be closed, oriented $3$-manifolds 
with knots $K\subset Y$ and $K'\subset Y'$. A cobordism $W$ 
between $(Y,K)$ and $(Y',K')$ is a pair $(N,\phi)$ where $N$ is 
a four-manifold with boundary $\partial N=-Y\sqcup Y'$ and $\phi$ 
is a proper embedding of the cylinder $[0,1]\times\sone$ into $N$ 
which maps its boundary to $K\sqcup K'$. We call $W$ 
a {\bf knot cobordism} from $(Y,K)$ to $(Y',K')$.
\end{definition}
For example, such a cobordism is given by attaching a $2$-handle 
$\sh^{(4,2)}$ in the complement of the knot $K$. The cobordism 
$N=[0,1]\times Y\cupb\sh^{(4,2)}$ admits a canonical embedding $\phi$ 
of the cylinder into $N$, i.e.~the embedding is given by the 
canonical inclusion
\[
 [0,1]\times K
 \hookrightarrow
 N
\]
and we define $K'=\{1\}\times K$.
\begin{definition} We say that two knot cobordisms $W$ and $W'$ are {\bf 
equivalent} if there is a diffeomorphism $\psi\co N\lra N'$ such that 
$\phi'=\psi\circ\phi$.
\end{definition}
\begin{definition} A handle decomposition of a knot cobordism $W=(N,\phi)$ 
between $(Y,K)$ and $(Y',K')$ is a handle decomposition $\mathcal{H}$ 
of $N$ relative to $Y$ using handles whose attaching spheres are all 
disjoint from $K$.
\end{definition}
Observe that the embedding $\phi$ can be extended to an 
embedding $\psi$ of $[0,1]\times\nu K$ into $N$ which again is disjoint 
from the handles. 
Thus, a handle decomposition of a knot cobordism $N$ can be defined 
equivalently as a handle decomposition of the 
manifold $\overline{N\backslash([0,1]\times\nu K)}$ relative
to $Y\backslash\nu K$, which we regard as a cobordism 
between $Y\backslash\nu K$ and $Y'\backslash\nu K'$. Observe, that our 
notion of knot cobordism is equivalent to the notion of 
{\it special cobordisms} of Juh\'{a}sz given in \cite{JuCOB}.

\begin{theorem}\label{handle} Let $\mathcal{H}_1$ and $\mathcal{H}_2$ be 
two handle decompositions of a knot cobordism $N$ between $(Y,K)$ 
and $(Y',K')$ then we can transform $\mathcal{H}_1$ into $\mathcal{H}_2$ 
through handle decompositions of knot cobordisms using a finite sequence 
of handle slides, handle pair creation/annihilation and isotopies within 
levels.
\end{theorem}
\begin{proof} This follows from the observation that a knot cobordism
is especially a special cobordism of Juh\'{a}sz and the results from
\cite{JuCOB}.
\end{proof}

\subsection{One-handles} Suppose we are given a closed, oriented 
$3$-manifold $Y$ with knot $K$ in it. Attach to the trivial 
cobordism $[0,1]\times Y$ a $4$-dimensional $1$-handle to the boundary
$\{1\}\times Y$, where the attaching spheres of the $1$-handles should 
be attached in the complement of $K$ in $Y$. Denote the resulting 
cobordism by $U$. The boundary of the cobordism $U$ is given as 
$\partial U=-Y\sqcup Y\#\stwo\times\sone$. Observe 
that $[0,1]\times K$ admits a natural embedding into $U$ 
with $\{1\}\times K$ being mapped into $Y\#(\stwo\times\sone)$. Let
 $(\Sigma,\afat,\bfat,w,z)$ be a Heegaard diagram adapted to the 
pair $(Y,K)$ and let $(E,\afat_0,\bfat_0,z_0)$ be a 
standard Heegaard diagram for $\stwo\times\sone$ 
(see~\cite[Definition~2.8]{OsZa03}) so that the $\alpha$-circle in $\afat_0$ 
and the $\beta$-circle in $\bfat_0$ meet in a single pair of 
intersection points. Denote by $\mE$ the diagram $(\Sigma,\afat_0,\bfat_0,w_0,z_0)$
where $w_0$ is a point in $\Sigma\backslash(\afat_0\cup\bfat_0)$ which lies
in the same component as $z_0$. Furthermore, denote by 
$\theta$ the intersection point with higher relative grading. By 
\cite[Corollary~6.8]{Saha}, we know that 
$(\Sigma',\afat',\bfat',w,z')=(\Sigma,\afat,\bfat,w,z)\#\mE$ is 
a Heegaard diagram adapted to $K$ and (the arguments given in 
Corollary~6.8 carry over verbatim for the $\hfkminus$-case)
\begin{equation}
\begin{array}{ccl}
  \hfkoc(Y\#\stwo\times\sone,K)
  &\cong&
  H_*(\cfkoc(\Sigma,\afat,\bfat,w,z)
  \otimes_{\ztwo[U]}
  \cfkoc(\mE))\\
  &\cong&
  \hfkoc(Y,K)\otimes_{\ztwo[U]}\hfkoc(\stwo\times\sone,U)
\end{array}
\label{eq:consum}
\end{equation}
Thus, we define a map
\[
 \gund_{U,\fraks}
 \co
 \cfkoc(\afat,\bfat;\fraks)
 \lra
 \cfkoc(\afat',\bfat';\fraks\#\frakso)
\]
by sending an element $U^i\cdot\xfat$, with $\xfat\in\talpha\cap\tbeta$, 
to $U^i\cdot\xfat\otimes\theta$. The moduli spaces in the definition of the 
differential on the right split as in the case 
of the $\hfhat$-theory (see~\cite[Corollary~6.8]{Saha} for the hat-theory 
and \cite[Theorem~1.5]{OsZa02}) and, thus, the map 
is chain. Denote by $\Ghatund_{U,\fraks}$ the induced map between the 
knot Floer homologies.

\subsection{Three-handles} Let $Y$ be a closed, oriented $3$-manifold 
and $K\subset Y$ a knot in $Y$. Suppose $V$ is a cobordism
obtained by adding a single three-handle along a non-separating 
$2$-sphere in $Y$ which is disjoint from $K$. The boundary components 
of $V$ are $(Y,K)$ and $(Y',K')$ with $Y'=Y\#\stwo\times\sone$ 
and $K'$ sitting inside $Y$. We would like to remind the reader 
of \cite[Lemma~4.11]{OsZa03}: In this situation we can find a
Heegaard diagram
\[
 \mH'=(\Sigma',\afat',\bfat',z')
 =
 \mH
 \#
 \mE
\]
of $Y'$ where $\mE$ is a standard Heegaard diagram for 
$\stwo\times\sone$ and $\mH=(\Sigma,\afat,\bfat,z)$ a Heegaard 
diagram for $Y$. Furthermore, if two such split diagrams 
$\mH'_i=\mH_i\#\mE$, for $i=1,2$, are equivalent, then 
$\mH_1$ and $\mH_2$ are equivalent.\vspace{0.3cm}\\
It is easy to find $\mH'$ which is adapted to the knot $K'$ and 
still splits into $\mH$ and $\mE$. Furthermore, observe that $\mH$ 
is a Heegaard diagram of $Y$ adapted to $K$. Thus, we may 
proceed as done in \cite{OsZa03}: We define a map
\[
 \eund_{V,\fraks}
 \co
 \cfkoc(\mH',\left.\fraks\right|_{Y'})
 \lra
 \cfkoc(\mH,\left.\fraks\right|_{Y})
\]
by sending an element $U^i\cdot\xfat\otimes\yfat$ with 
$\xfat\otimes\yfat\in\talphaprime\cap\tbetaprime$ to
$U^i\cdot\xfat$ if $\yfat$ is the minimal intersection point 
of $\stwo\times\sone$ and to 
zero otherwise. As in the case of $1$-handles, the moduli spaces 
of Whitney disks of $\cfkoc(\mH')$ split (see Corollary~6.8 
of \cite{Saha}) and, thus, $\eund_{V,\fraks}$ is a chain map. 
Hence, we get
\[
 \Ehatund_{V,\fraks}
 \co
 \hfkoc(Y',K';\left.\fraks\right|_{Y'})
 \lra
 \hfkoc(Y,K;\left.\fraks\right|_{Y}).
\]
Observe that the definition of the maps $\Ehatund$ and $\Ghatund$ do 
not use the base point $w$. 

\subsection{Two-handles} Suppose we are given a closed, oriented $3$-manifold $Y$
and a knot $K\subset Y$. Furthermore, let $\link$ be a framed link in $Y$ which
is disjoint from $K$. We call such a link {\bf admissible}.
Analogous to the case of knots it is possible to find a Heegaard diagram {\bf 
subordinate} to the link $K\sqcup\link$ (see~\cite{OsZa04} or \cite{Saha,Saha02}).
To describe such a diagram, let $\link_1,\dots,\link_k$ be the 
components of $\link$. Then, there is a Heegaard diagram $(\Sigma,\afat,\bfat,w,z)$ 
of $Y$ such that $K$ intersects $\beta_1$ once, transversely and is disjoint from 
the other $\bfat$-circles and for $i=1,\dots,k$ the knot $\link_i$ intersects 
$\beta_{i+1}$, transversely and is disjoint from the other $\bfat$-circles. 
The pair 
$(w,z)$ determines the knot $K$ in the usual way, i.e.~in the way introduced by
Ozsv\'{a}th and Szab\'{o} in \cite{OsZa04}. The 
diagram is characterized by the property that it comes from a handle decomposition 
(see \cite{GoSt} for a definition of relative handle decompositions) of 
$Y\backslash\nu(K\sqcup\link)$ relative to 
$\partial(\nu(K\sqcup\link))$ (cf.~\cite{Saha,Saha02}). Using the Kirby 
calculus picture behind these subordinate diagrams as explained for instance 
in \cite{Saha, Saha02}, it is easy to see the following.
\begin{prop}\label{prop:bouquet} Let $Y$ be a closed, oriented $3$-manifold 
and $K\subset Y$ a
knot. Let $\link$ be a framed link and 
$\mH_i=(\Sigma_i,\afat_i,\bfat_i,w,z)$ be two diagrams subordinate to the 
link $K\sqcup\link$. Denote by $I$ the interval inside $K$ connecting $z$ 
with $w$, interpreted as sitting in $\Sigma$. Then these diagrams are 
isomorphic after a finite sequence of the following moves:
\begin{enumerate}
 \item[($m_1$)] Handle slides and isotopies among the $\afat$-curves. 
 These isotopies may not cross $I$.
 \item[($m_2$)] Handle slides and isotopies among the 
 $\beta_{k+2},\dots,\beta_g$. These isotopies may not cross $I$.
 \item[($m_3$)] Handle slides of one of the $\beta_i$, $i\leq k+1$ over 
 one of the $\beta_j$, $j\geq k+2$. These isotopies may not cross $I$.
 \item[($m_4$)] Stabilizations/destabilizations.
\end{enumerate}
\end{prop}
This proposition can be proved with a straightforward adaption 
of the proof of \cite[Lemma~2.3]{Saha}.
Given a subordinate diagram we may deduce a Heegaard triple 
$(\Sigma,\afat,\bfat,\gfat,w,z)$ with the following properties: 
$\gamma_1$ and $\gamma_{k+2},\dots,\gamma_g$ are small isotopic translates 
of $\beta_1$ and $\beta_{k+2},\dots,\beta_g$. In fact, every pair 
$\beta_i$, $\gamma_i$ meet in a canceling pair of intersection points. 
The curve $\gamma_i$, for $i=2,\dots,k+1$, is determined by the framing 
of $\link_{i-1}$. Recall from \cite[\S8.1.]{OsZa01} (or \cite[\S4.1.]{OsZa03}) 
that the triple diagram $(\Sigma,\afat,\bfat,\gfat)$ determines a 
cobordism $X_{\afat,\bfat,\gfat}$ with three boundary components 
denoted by $Y_{\afat,\bfat}$, $Y_{\bfat,\gfat}$ and $Y_{\afat,\gfat}$.
For $\xfat\in\talpha\cap\tbeta$, 
$\yfat\in\talpha\cap\tbeta$ and $\fraks\in\spinc(X_{\afat,\bfat,\gfat})$ 
we define
\[
 \foc_\fraks(\xfat\otimes\yfat)
 =
 \sum_{\qfat\in\talpha\cap\tgamma,i\geq0}
 \#\bigl(
 \mM_{(\afat,\gfat,\bfat);\fraks}^{0;i}(\xfat,\yfat,\qfat)
 \bigr)\cdot U^i\thinspace\qfat,
\]
where $\mM_{(\afat,\gfat,\bfat);\fraks}^{0;i}(\xfat,\yfat,\qfat)$ denotes
the moduli space of holomorphic Whitney triangles $\phi$ which connect 
$\xfat$, $\yfat$ and $\qfat$ with boundary conditions in $\talpha$, 
$\tgamma$ and $\tbeta$ such that $n_w(\phi)=i$ and $\fraks_z(\phi)=\fraks$. 
We extend this to a bilinear pairing
\[
 \foc_\fraks
 \co
 \cfkoc(\afat,\bfat,w,z;\fraks_{\afat\bfat})
 \otimes
 \cfkoc(\bfat,\gfat,w,z;\fraks_{\bfat\gfat})
 \lra
 \cfkoc(\afat,\gfat,w,z;\fraks_{\afat\gfat})
\]
where $\fraks_{\afat\bfat}=\left.\fraks\right|_{Y_{\afat,\bfat}}$, 
$\fraks_{\bfat\gfat}=\left.\fraks\right|_{Y_{\bfat,\gfat}}$ and
$\fraks_{\afat\gfat}=\left.\fraks\right|_{Y_{\afat,\gfat}}$.
\begin{rem} Our notation of the boundary conditions at the edges
of the triangle differs from the notation introduced by Ozsv\'{a}th
and Szab\'{o}. We adapted the notation to the conventions introduced
at the beginning of \S\ref{sec:wpash}.
\end{rem}

\begin{lem} The Heegaard diagram $(\Sigma,\bfat,\gfat,w,z)$ is subordinate
to the unknot $U$ in the manifold $\#^k(\stwo\times\sone)$.
\end{lem}
\begin{proof} The diagram clearly represents $\#^k(\stwo\times\sone)$ 
(see~\cite{OsZa01, OsZa02}). By definition, the points $w$ and $z$ can 
be connected in $\Sigma$ in the complement of the $\beta$-curves. Let 
$c$ be such a curve. We claim that we can choose the curve 
$c$ to sit in the complement of the $\gfat$-curves. If we are 
able to show this, we are done, as we can use a small push-off 
$c'$ of $c$ and connect $z$ with $w$ in the complement of 
the $\gfat$-curves. By definition, the union $c'\cup c$ is 
isotopic to the knot represented by the pair $(w,z)$. The curve 
$c'\cup c$ is contractible and, thus, $(w,z)$ represents 
the unknot.\\
The curve $\gamma_1$ is a small isotopic translate of the curve 
$\beta_1$. Since $c$ and $\beta_1$ are disjoint, the curves $c$ 
and $\gamma_1$ are disjoint, too. The curves $\beta_i$, 
$i=2,\dots,k+1$ are meridians of torus components of the surface 
$\Sigma$. The curve $\gfat_i$ is isotopic to $n_i\beta_i+\lambda_i$ 
where $n_i$ is a suitable integer and $\lambda_i$ a longitude of 
the corresponding torus component associated to $\beta_i$. Hence, 
the $\gamma_i$ can be thought of as staying outside of the torus 
component in which $c$ lies in. Finally, there are the curves 
$\gamma_j$, $j\geq k+2$. The curve $\gamma_j$ is an isotopic 
translate of the curve $\beta_j$ which is disjoint from $c$. 
Hence, $\gamma_j$ can be thought of as being disjoint from $c$.
\end{proof}
Consequently, we have that $\hfkoc(\bfat,\gfat,w,z)$ 
admits a top-dimensional generator, $\hattheta$ say 
(see~\cite{OsZa03}). Denote by $Y'$ the manifold obtained by the 
surgery along the framed link $\link$ and denote by $K'$ the 
knot $K$ after the surgery. Since $Y_{\afat,\bfat}\cong Y$ and
$Y_{\afat,\gfat}\cong Y'$, it is possible to define
\[
  \Foc_{\link;\fraks}
  \co
  \hfkcirc(Y,K;\left.\fraks\right|_Y)
  \lra
  \hfkcirc(Y',K';\left.\fraks\right|_{Y'})
\]
as the map induced by $\foc_\fraks(\,\cdot\,,\hattheta)$ in homology.

\subsection{Invariants of Cobordisms} Given a knot cobordism $W$ from $(Y,K)$ to 
$(Y',K')$, we choose a handle decomposition of it, i.e.~we choose a splitting
\[
  W=W_1\cupb W_2\cupb W_3
\]
where $W_i$ is obtained by attaching $i$-handles. Let $\link$ be the framed link associated to the $2$-handle attachments in $W_2$, then we define
\[
  \Foc_{W;\fraks}
  =
  \Ehatund_{W_3;\fraks}
  \circ
  \Foc_{\link;\fraks}
  \circ
  \Ghatund_{W_1;\fraks}.
\]
\begin{proof}[Proof of Theorem~\ref{cobmapinvar}] This theorem has to 
be proved by showing that the map $\Foc_{W;\fraks}$
does not depend on the handle decomposition of $W$ and the data 
associated to it. It is easy to observe that on the chain level 
we have that
\begin{equation}
\begin{array}{rcl}
 \Ghatund_{U;\fraks}&=&G^\circ_{U;\fraks}\\
 \Ehatund_{V;\fraks}&=&E^\circ_{V;\fraks},
\end{array}
\label{eq:gandeequal}
\end{equation}
where $G^\circ_{U;\fraks}$ and $E^\circ_{V;\fraks}$ are the maps 
associated to $1$-handles and $3$-handles, respectively, which were defined by
Ozsv\'{a}th and Szab\'{o} in \cite{OsZa03}.
Thus, the maps on the left of \eqref{eq:gandeequal} and on the right 
of \eqref{eq:gandeequal} basically have the same properties. Examining 
the work \cite{OsZa03}, we see that to make the invariance work in the knot 
Floer case there is one property that is central: We have to prove that 
the map $\Foc_{\link;\fraks}$ is independent of the choice of subordinate
Heegaard diagram. In \cite{OsZa03}, Ozsv\'{a}th and Szab\'{o} prove this
for the maps $\Fcirc_{\link;\fraks}$ by showing that they {\it commute}
with all maps induced by admissible Heegaard moves. We will do this for 
the maps $\Foc_{\link;\fraks}$ in the sequel:
Given two Heegaard diagrams $\mH=(\Sigma,\afat,\bfat,w,z)$ and 
$\mH'=(\Sigma',\afat',\bfat',w,z)$ subordinate to the link $\link$, 
there is a sequence of subordinate Heegaard diagrams $\mH_1,\dots,\mH_n$
with $\mH=\mH_1$ and $\mH'=\mH_n$ such that we get from $\mH_i$ to
$\mH_{i+1}$ by one of the moves introduced in Proposition~\ref{prop:bouquet}.
Each of these moves induces an isomorphism
\[
 \Psi^\oc_{\fraks;i}
 \co
 \hfkoc(\mH_i;\fraks)
 \lra
 \hfkoc(\mH_{i+1};\fraks)
\]
between the respective homologies. We have to prove that
\begin{equation}
 \Psi^\oc_{\fraks;i}
 \circ
 \Foc_{\link;\fraks}
 +
 \Foc_{\link;\fraks}
 \circ
 \Psi^\oc_{\fraks;i}
 =0.
\label{eq:goal2}
\end{equation}
If $\Psi^\oc_{\fraks;i}$ is induced by a stabilization/destabilization,
the proof from the Heegaard Floer case, i.e.~the proof 
of \cite[Lemma 4.7]{OsZa03}, carries over verbatim. Otherwise, we proceed 
as follows: We first start defining
\[
 \Foc_{\link}
 =
 \sum_{\fraks}
 \Foc_{\link;\fraks}
\]
and, correspondingly, we define $\Psi^\oc_i$. Observe, that the sum is finite 
due to the fact that we demand the condition $n_z=0$ and use diagrams that are 
weakly admissible with respect to the point $z$. We know that, forgetting the 
point $w$, we obtain isomorphisms $\Psi_i$ and maps $\Fhat_{\link;\fraks}$ 
between the associated Heegaard Floer theories which fulfill the equation
\begin{equation}
 \widehat{\Psi}_i
 \circ
 \Fhat_{\link}
 +
 \Fhat_{\link}
 \circ
 \widehat{\Psi}_i
 =0\label{eq:needthis01}
\end{equation}
as shown by Ozsv\'{a}th and Szab\'{o}. They derived this equation by counting
ends of a suitable $1$-dimensional moduli space. Hence, their proof has a 
formulation in terms of the symbol homology theory: As shown in the previous 
sections, there is a symbol homology theory $\bh$ (see~\S\ref{sec:potacs} 
and \S\ref{sec:i}) such that 
both $\widehat{\Psi}_i$ and $\Fhat_{\link}$ admit canonical symbols in this 
theory, $\fraks_{\Psi}$ and $\fraks_{\Fhat}$ say. Ozsv\'{a}th 
and Szab\'{o}'s proof can be interpreted in the language of symbol homology 
which gives 
$
 \fraks_{\Fhat}\boxtimes\fraks_{\Psi}
 \boxplus
 \fraks_{\Fhat}\boxtimes\fraks_{\Psi}=0
$
in $\bh$. Since the Heegaard move underlying the map $\Psi_i$ 
respects the point $w$, we know that there are filtering morphisms
\begin{eqnarray*}
 \mff_*
 \co
 \bh
 \lra
 \bh^w\\
 \mff_{U;*}\co
 \bh
 \lra
 \bhu\\
\end{eqnarray*}
such that Theorem~\ref{thm:propp} holds (alternatively, Theorem~\ref{main} and Theorem~\ref{main2}). Since $\fraks_{\Fhat}$ fulfills the property $P$ with
$
 P(X)
 =
 X\boxtimes\fraks_{\Psi}
 \boxplus
 X\boxtimes\fraks_{\Psi},
$
the filtered symbols $\mff_*(\fraks_{\Fhat})$ and 
$\mff_{U;*}(\fraks_{\Fhat})$ fulfill
the properties $P_\mff$ and $P_{\mff_U}$, respectively. But, 
by construction, $\ev_*(\mff_*(\fraks_{\Fhat}))=\Foo_{\link;\fraks}$ 
and $\ev_*(\mff_*(\fraks_\Psi))=\Psi^{\bullet,\bullet}_i$. So, 
property $P_\mff$ implies 
\begin{eqnarray*}
 0
 =
 \ev_*\bigl(P_\mff(\mff_*(\fraks_{\Fhat}))\bigr)
 &=&
 \ev_*
 \bigl(\mff_*(\fraks_{\Fhat})\boxtimes\mff_*(\fraks_{\Psi})
 \boxplus
 \mff_*(\fraks_{\Fhat})\boxtimes\mff_*(\fraks_{\Psi})
 \bigr)\\
 &=&
 \ev_*\bigl(\mff_*(\fraks_{\Psi})\bigr) 
 \circ
 \ev_*\bigl(\mff_*(\fraks_{\Fhat})\bigr)
 +
 \ev_*\bigl(\mff_*(\fraks_{\Psi})\bigr)
 \circ
 \ev_*\bigl(\mff_*(\fraks_{\Fhat})\bigr)
 \\
 &=&
 \Psi^{\bullet,\bullet}_i
 \circ
 F^{\bullet,\bullet}_{\link}
 +
 F^{\bullet,\bullet}_{\link}
 \circ
 \Psi^{\bullet,\bullet}_i
\end{eqnarray*}
and, correspondingly, property $P_{\mff_U}$ implies
\[
 0=\ev_*(P_{\mff_U}(\mff_{U;*}(\fraks_{\Fhat})))
 =
 \Psi^{\bullet,-}_i
 \circ
 F^{\bullet,-}_{\link}
 +
 F^{\bullet,-}_{\link}
 \circ
 \Psi^{\bullet,-}_i.
\]
Hence, the map $\Foc_{\link}$ is an invariant as stated.\\
To get the refined statements, i.e.~equation \eqref{eq:goal2}, we make 
the following adaptions:
Observe, that to encode the maps $\Psi_i$ and $\Foc_\link$ as symbols, we need
a set $B$ which contains at most four sets of attaching circles. Hence,
in the corresponding symbol algebra only bigons, triangles and rectangles
appear. For all of these $n$-gons, Ozsv\'{a}th and Szab\'{o} introduced
the notion of associated $\spinc$-structure 
(see~\cite[Proposition 8.5 and \S8.1.5]{OsZa01}). We alter 
the theory $\bh$ by attaching an additional datum to the 
generators: We decorate moduli spaces as done in \S\ref{sec:tshp} and, 
additionally, attach a $\spinc$-structure. A choice of $\spinc$-structure
on a Whitney polygon especially induces a choice of $\spinc$-structures
on its vertices. We follow the lines from \S\ref{sec:tshp} verbatim except
for two issues: First, when defining the $\boxtimes$-product 
of pre-generators $\mA_1$ and $\mA_2$ with $\spinc$-structures we 
bring in the chosen $\spinc$-structure by saying that the data of two 
elements match, if the data of the factors without $\spinc$-structures match 
(in the sense given in \S\ref{sec:wpash}) and if the $\spinc$-structure
at their common vertex coincides (see~Definition~\ref{def:operators}).
Second, when defining the differential $\parsh$ of a generator $\mA$ with 
$\spinc$-structure $\fraks$, we bring in the additional datum in the following
way: Recall, that $\parsh$ is modeled on $\parcoone(\pi(\mA))$. For each 
component $A_1^{ab;\qfat}\times A_2^{ab;\qfat}$ of $\parcoone(\pi(\mA))$ 
(cf.~\S\ref{thediff}) the $\spinc$-structure $\fraks$ of $\mA$ 
induces $\spinc$-structures $\fraks_i$ on $A_i^{ab;\qfat}$, for $i=1,2$. We 
will attach $\fraks_i$ to the corresponding $\mA^{ab;\qfat}_i$ 
(cf.~\S\ref{thediff}) and then proceed as in \S\ref{thediff}.
The resulting homology theory shall be denoted by $\bh^c$. On this
homology theory there exists a map
\[
 \ev^c
 \co
 \bh^c
 \lra
 \maphat
\]
which is defined as $\ev$: Using the notation from 
Proposition~\ref{intro:ev}, for
a generator with $\spinc$-structure $\fraks$,
 i.e.~$\mA=A_{(\pointed,\fin,\fout);\fraks}$, we define
\[
 \ev^c(\mA)(\xfat)=\#_\fraks\bigl(A(\xfat,\qfat,\rfat)\bigr)\cdot\rfat
\]
where $\#_\fraks$ only counts elements whose $\spinc$-structure 
equal to $\fraks$. We extend to the symbol algebra as done in 
\eqref{eq:mor}. This theory also comes 
with variants $\bh^{w;c}$ and $\bhu^c$ and those
defined in \S\ref{sec:fsh} and \S\ref{sec:uesh}. We start with the
equation
\[
 \widehat{\Psi}_{\fraks;i}
 \circ
 \Fhat_{\link;\fraks}
 +
 \Fhat_{\link;\fraks}
 \circ
 \widehat{\Psi}_{\fraks;i}
 =0,
\]
which was proved by Ozsv\'{a}th and Szab\'{o} and copy the arguments 
from above and apply Theorem~\ref{main2}. This will provide us
with equality \eqref{eq:goal2}.
\end{proof}

\subsection{Other Knot Floer Theories}\label{sec:okft}
The construction of the maps presented above with the obvious 
notational adaptions --~i.e.~swapping the roles of $w$ and $z$~-- also
provide cobordism maps in the theories $\hfkco$. To apply the symbol homology 
for the invariance proof we have to alter the theory a little. In 
the proof of Theorem~\ref{cobmapinvar} the invariances of the maps 
were transferred from the $\hfhat$-theory using the filtering 
maps $\mff$ and $\mff_U$. 
In $\hfkco$, invariances cannot be related to the $\hfhat$-theory
but to the $\hfcirc$-theory. 
We have to define a symbol homology that captures the theory $\hfcirc$ (in the sense specified in the introduction). This 
is done by decorating moduli spaces of Whitney polygons for which both $n_z$ and 
$n_w$ are arbitrary. We also attach a choice of $\spinc$-structure to the moduli 
spaces (as done in the proof of Theorem~\ref{cobmapinvar}). The construction 
then follows the lines of \S\ref{sec:tshp}. Because of the $\spinc$-structures 
in the construction we will use $\ev^c$ to interpret the elements of the 
associated symbol algebra as maps. As filtering morphism we will need one of 
type $\mff$, i.e.~one analogous as the one defined in \S\ref{sec:tfm}. 
Then, the invariance proof will proceed exactly as the proof of Theorem~\ref{cobmapinvar}.

\section{Implication II -- A Surgery Exact Triangle}\label{sec:implications}
In \S\ref{sec:knotcob} we constructed cobordism maps for 
knot Floer homologies and used the techniques from symbol 
homologies in their invariance proof. Here, we would like to 
present other examples of how the techniques can be applied.
The fact that the knot Floer homologies are invariants of
a pair $(Y,K)$ is well-known, however the proof of 
Corollary~\ref{knotinv} shows that these invariances are in some
way {\it inherited} from the invariances of the $\hfhat$-theory. 
Furthermore, we
prove a generalization of the surgery exact sequence presented in
\cite[Theorem~8.2]{OsZa04}. Again, the proof rests on the 
symbol homology approach we discussed.
\begin{cor}[see~\cite{OsZa04}]\label{knotinv} The knot 
Floer homologies $\hfkoc(Y,K)$ are invariants of the pair $(Y,K)$ 
because $\hfhat(Y)$ is an invariant of $Y$. 
\end{cor} 
\begin{proof} Suppose we are given two Heegaard diagrams $\mH$ 
and $\mH'$ subordinate to a pair $(Y,K)$. We have to see that the Floer
homologies associated to these diagrams are isomorphic. There is a 
sequence $\mH_1,\dots,\mH_n$ of Heegaard diagrams subordinate to $(Y,K)$
such that we transform $\mH_i$ to $\mH_{i+1}$ by using one of the moves given
in Proposition~\ref{prop:bouquet}. The only non-trivial statements to prove are
independence of the choice of almost complex structure and invariance under
isotopies of the attaching circles. The arguments Ozsv\'{a}th and Szab\'{o} gave for stabilizations and handle 
slides in the Heegaard Floer case are completely independent of the introduction of 
an additional base point $w$. Thus, these proofs immediately carry over to the knot 
Floer case. We will proceed to give the
arguments for an isotopy $I_t$ of the $\bfat$-circles. The other cases can be 
proved in the same fashion. Thus, suppose we obtained
$\mH_{i+1}=(\Sigma,\afat,\bfat')$ from $\mH_{i}=(\Sigma,\afat,\bfat)$ by the 
isotopy $I_t$. The isotopy $I_t$ and its inverse $I_{1-t}$ induce maps 
\[
\begin{array}{rl}
  \pto:&\hspace{-0.2cm}
  \cfhat(\mH)
  \lra
 \cfhat(\mH')\\
 \pot:&\hspace{-0.2cm}
  \cfhat(\mH')
  \lra
  \cfhat(\mH),
\end{array}
\]
which are chain maps and their composition $\pot\circ\pto$ is 
chain homotopic to the identity. Denote 
by $\fraks_{\mH}$ and $\fraks_{\mH'}$ the canonical symbols
of the differentials $\parhat_{\mH}$ and $\parhat_{\mH'}$.
Furthermore, denote by $H$ the chain homotopy and by $\fraks_H$ the associated
canonical symbol (in the perturbed symbol homology). Finally, denote by
$\fraks_{to}$ and $\fraks_{ot}$ the canonical symbol of $\pto$ and $\pot$. 
Using the symbol homologies we get
\[
\begin{array}{rl}
  \ev_*(\mff_*(\fraks_{to})):&\hspace{-0.2cm}
  \cfkoo(\mH)
  \lra
  \cfkoo(\mH')\\
  \ev_*(\mff_*(\fraks_{ot})):&\hspace{-0.2cm}
  \cfkoo(\mH')
  \lra
  \cfkoo(\mH)\\
  \ev_*(\mff_{U;*}(\fraks_{to})):&\hspace{-0.2cm}
  \cfkminus(\mH)
  \lra
 \cfkminus(\mH')\\
 \ev_*(\mff_{U;*}(\fraks_{ot})):&\hspace{-0.2cm}
  \cfkminus(\mH')
  \lra
  \cfkminus(\mH).
\end{array}
\]
We want to prove that the first map is a chain map and that the
composition of the upper two is chain homotopic to the 
identity: The fact that $\pto$ is a chain map is given by the 
property $P$ with
\[
 P(X)
 =
 X
 \boxtimes
 \fraks_{\mH'}
 \boxplus
 \fraks_{\mH}
 \boxtimes
 X.
\]
The fact that $H$ is a chain homotopy to the identity is encoded by
$Q$ with
\[
 Q(X)
 =
 \fraks_{to}\boxtimes\fraks_{ot}
 \boxplus
 \os
 \boxplus
 X
 \boxtimes
 \fraks_{\mH'}
 \boxplus
 \fraks_{\mH}
 \boxtimes 
 X.
\]
The invariance proof of Ozsv\'{a}th and Szab\'{o} shows that
$P(\fraks_{to})=0$ and that $Q(\fraks_H)=0$. By Theorem~\ref{thm:propp} 
the symbol $\mff_*(\fraks_{to})$
fulfills property $P_\mff$ and the symbol $\mff_*(\fraks_H)$ fulfills the
property $Q_\mff$. Thus, both $P_\mff(\mff_*(\fraks_{to}))$ and
$Q_\mff(\mff_*(\fraks_H))$ vanish. Applying the morphism $\ev_*$ we 
see that 
\[
0=\ev_*(P_\mff(\mff_*(\fraks_{to})))\\
=
\ev_*(\mff_*(\fraks_{to}))\circ\partial^{\bullet,\bullet}_{\mH}
+
\partial^{\bullet,\bullet}_{\mH'}
\circ
\ev_*(\mff_*(\fraks_{to}))
\]
holds and, correspondingly, that
\begin{eqnarray*}
0
&=&
\ev_*(Q_\mff(\mff_*(\fraks_{H})))\\
&=&
\ev_*(\mff_*(\fraks_{ot}))\circ\ev_*(\mff_*(\fraks_{to}))
+
\id
+
\partial^{\bullet,\bullet}_{\mH'}
\circ 
\ev_*(\mff(\fraks_H))
+
\ev_*(\mff(\fraks_H))\circ\partial^{\bullet,\bullet}_{\mH}.
\end{eqnarray*}
Hence, $\ev_*(\mff_*(\fraks_{to}))$ is a chain map and the composition
$\ev_*(\mff_*(\fraks_{ot}))\circ\ev_*(\mff_*(\fraks_{to}))$ chain 
homotopic to the identity. Similarly, we prove that 
$\ev_*(\mff_*(\fraks_{ot}))$ is chain and that
$\ev_*(\mff_*(\fraks_{to}))\circ\ev_*(\mff_*(\fraks_{ot}))$ is chain 
homotopic to the identity. Using the filtering map $\mff_{U;*}$, we can 
prove the corresponding statements for $\ev_*(\mff_{U;*}(\fraks_{to}))$, 
$\ev_*(\mff_{U;*}(\fraks_{ot}))$, 
$\ev_*(\mff_{U;*}(\fraks_{to}))\circ\ev_*(\mff_{U;*}(\fraks_{ot}))$ 
and 
$\ev_*(\mff_{U;*}(\fraks_{ot}))\circ\ev_*(\mff_{U;*}(\fraks_{to}))$.
\end{proof}
Now suppose we are given a closed, oriented $3$-manifold 
$Y$ and a knot $K\subset Y$. Given a knot $L\subset Y$ disjoint from
$K$ with framing $n$, we define $(Y_n,K')$ as the pair we obtain by 
performing surgery along $L$. Denote by $W_1$ the induced knot cobordism. 
Then, we denote by $(Y_{n+1},K'')$ the pair we obtain 
from $(Y_n,K')$ by performing a $(-1)$-surgery along a meridian 
$L$, $\mu$ say, and we denote by $W_2$ the associated knot cobordism. Finally, denote by $W_3$ the knot cobordism
obtained by performing a $(-1)$-surgery along a meridian of $\mu$.
\begin{theorem}\label{thm:set} In the situation defined above, the 
following sequence is exact.
\[
 \xymatrix@C=3pc@R=0.5pc{
 \hfkoc(Y,K)\ar[r]^{\Foc_{W_1}}
 &
 \hfkoc(Y_n,K')\ar[r]^{\Foc_{W_2}}
 &
 \hfkoc(Y_{n+1},K'')\ar@/^1.5pc/[ll]^{\Foc_{W_3}}
 }
\]
\end{theorem}
\begin{proof} We use the mapping cone proof approach of 
Ozsv\'{a}th and Szab\'{o} from \cite{OsZa05}. They use an 
algebraic trick using mapping cones to prove exactness of 
the sequence, namely \cite[Lemma~4.2]{OsZa05}. 
To apply this lemma, they need to prove two properties, where
the one is an associativity property of cobordism maps and 
the other a chain homotopy relation of cobordism maps. Both 
of these properties can be encoded as a property $P$ in the 
corresponding symbol homologies.
Hence, by Theorem~\ref{main} and Theorem~\ref{main2} the 
corresponding associativity property and
chain homotopy relation also hold in the knot Floer 
case which allows us to apply their Lemma~4.2 to get exactness.
\end{proof}
In a similar vein, other properties and statements about cobordism
maps can be easily transferred. Since the strategy of the proofs is
always the same, we will leave this to the interested reader.

\end{document}